\documentclass{imsart}

\RequirePackage{amsthm,amsmath,amsfonts,amssymb,mathrsfs}

\RequirePackage[numbers]{natbib}
\RequirePackage{graphicx}
\usepackage{multirow}
\usepackage{booktabs}
\usepackage{enumitem}
\usepackage[colorlinks]{hyperref}
\usepackage{xcolor}
\usepackage{comment}
\usepackage{float}

\startlocaldefs
\numberwithin{equation}{section}

\newtheorem{axiom}{Axiom}
\newtheorem{claim}[axiom]{Claim}
\newtheorem{theorem}{Theorem}[section]
\newtheorem{thm}{Theorem}[section]
\newtheorem{lem}[theorem]{Lemma}

\newtheorem{rem}[theorem]{Remark}
\newtheorem{defn}[theorem]{Definition}

\newtheorem{condition}[theorem]{Condition}
\newtheorem{cor}[theorem]{Corollary}
\newtheorem{prop}[theorem]{Proposition}

\endlocaldefs

\begin{document}
	
	\begin{frontmatter}
		\title{Tracy-Widom limit for the largest eigenvalue of high-dimensional
			covariance matrices in elliptical distributions}
		\runtitle{Tracy-Widom limit in elliptical distributions}
		
		\begin{aug}
			\author[B]{\fnms{Wen} \snm{Jun}\ead[label=e1]{jun.wen@u.nus.edu}},
			\author[B]{\fnms{Xie} \snm{Jiahui}\ead[label=e2]{jiahui.xie@u.nus.edu}},
			\author[B]{\fnms{Yu} \snm{Long}\ead[label=e3]{stayl@nus.edu.sg}}
			\and
			\author[B]{\fnms{Zhou} \snm{Wang}\ead[label=e4,mark]{wangzhou@nus.edu.sg}}
			
			\address[B]{Department of Statistics and Applied Probability,
				National University of Singapore,
				\printead{e4}}
		\end{aug}
		
		\begin{abstract}
			Let $X$ be an $M\times N$ random matrix consisting of independent
			$M$-variate elliptically distributed column vectors $\mathbf{x}_{1},\dots,\mathbf{x}_{N}$
			with general population covariance matrix $\Sigma$. In the literature,
			the quantity $XX^{*}$ is referred to as the sample covariance matrix after scaling,
			where $X^{*}$ is the transpose of $X$. In this article, we prove that the limiting behavior of the scaled largest eigenvalue of $XX^{*}$
			is universal for a wide class of elliptical distributions, namely,
			the scaled largest eigenvalue converges weakly to the same limit regardless of the distributions
			that $\mathbf{x}_{1},\dots,\mathbf{x}_{N}$ follow as
			$M,N\to\infty$ with $M/N\to\phi_0>0$ if the  weak fourth moment of the radius of $\mathbf{x}_{1}$ exists . In particular,
			via comparing the Green function with that of the sample covariance
			matrix of multivariate normally distributed data, we conclude that
			the limiting distribution of the scaled largest eigenvalue is the
			celebrated Tracy-Widom law. 
		\end{abstract}

		
	\end{frontmatter}
	
	\section{Introduction}

	Suppose one observed independent and identically distributed (i.i.d.) data $\mathbf{x}_{1},$ $\dots,\mathbf{x}_{N}$
	with mean $0$  from $\mathbb{R}^{M}$, where the positive integers
	$N$ and $M$ are the sample size and the dimension of data respectively.
	Define ${\cal W}=N^{-1}\sum_{i=1}^{N}\mathbf{x}_{i}\mathbf{x}_{i}^{*}$,
	referred to as the sample covariance matrix of $\mathbf{x}_{1},\dots,\mathbf{x}_{N}$,
	where $*$ is the conjugate transpose of matrices throughout this
	article. A fundamental research question in statistics is to analyze the behavior
	of ${\cal W}$. Let $X=(\mathbf{x}_{1},\dots,\mathbf{x}_{N})$ be
	the $M\times N$ matrix, while  $\Sigma=\mathbb{E}\mathbf{x}_{1}\mathbf{x}_{1}^{*}$ is defined as
	the population covariance matrix. In recent decades, fruitful results
	exploring the asymptotic property of ${\cal W}$ have been established
	by random matrix theory under the high-dimensional asymptotic regime.
	In contrast to the traditional low dimensional asymptotic regime where
	the dimension $M$ is usually fixed or small and the sample size $N$
	is large, the high-dimensional asymptotic regime refers to that both
	$N,M$ are large and even of comparable magnitude. For a list of introductory
	materials on random matrix theory, see e.g. \cite{Bai1999review,BaiandSilverstein2010,locallawlecturenotes,Ding2017,ErdosSurvey,Silverstein_notes}.
	Let $T$ be a matrix such that $TT^*=\Sigma$,
	It is worth noting that most
	works on the inference of high-dimensional covariance matrices using
	random matrix theory assume that $X=TY$ with the $M\times N$
	matrix $Y$ consisting of i.i.d. entries with  mean $0$ and variance
	$1$. This assumption excludes many practically useful statistical
	models, for instance, almost all members in the family of elliptical
	distributions. One exception is the case where $\mathbf{x}_{1},$ $\dots,\mathbf{x}_{N}$
	follow  $M$-variate normal distribution with mean $0$ and population
	covariance matrix $\Sigma$. If $\mathbf{x}_{1},$ $\dots,\mathbf{x}_{N}$
	are not normal, the entries in each column of $X$ are only guaranteed
	to be uncorrelated instead of independent, the latter being a much
	stronger notion than the former.
	
	In this article, we consider the case where $\mathbf{x}_{1},$ $\dots,\mathbf{x}_{N}$
	follow elliptical distribution which is a family of probability distributions
	widely used in statistical modeling. See e.g. the technical report
	\cite{fang1990} for a comprehensive introduction to elliptical distributions.
	Generally, we say a random vector $\mathbf{y}$ follows elliptical
	distribution if $\mathbf{y}$ can be written as
	\begin{equation}
	\mathbf{y}=\xi A\mathbf{u},\label{eq:elliptical model}
	\end{equation}
	where $A\in\mathbb{R}^{M\times M}$ is a nonrandom
	matrix with ${\rm rank}(A)=M$, $\xi\ge0$ is a scalar random variable
	representing the radius of $\mathbf{y}$, and $\mathbf{u}\in\mathbb{R}^{M}$
	is the random direction, which is independent of $\xi$ and uniformly
	distributed on the $M-1$ dimensional unit sphere $\mathbb{S}^{M-1}$
	in $\mathbb{R}^{M}$, denoted as $\mathbf{u}\sim U(\mathbb{S}^{M-1})$.
	See e.g. \cite{Dobriban2018b,Dobriban2018a,ZhouWang2018a,Yang2017}
	for some recent advances on statistical inference for elliptically
	distributed data. The current paper focuses on the problem involving the largest
	eigenvalue of sample covariance matrix with elliptically distributed
	data. Briefly speaking, we show the following result.
	\begin{claim}
		\textit{If $\Sigma$ satisfies some mild assumptions, then the rescaled largest eigenvalue $N^{2/3}(\lambda_1(\mathcal{W})-\lambda_{+})$
			converges weakly to the celebrated Tracy-Widom law (\cite{elkaroui2007,Johnstone2001,onatski2008,Tracy1994}) if  the radius satisfies the following tail probability condition:
			\begin{equation}\label{eq:tail condition intro}
			\lim_{s\rightarrow\infty}\limsup_{N\to\infty}s^{2}\mathbb{P}(|N\xi^2-M|\ge \sqrt M \emph{}s)=0
			\end{equation}}
	\end{claim}
	\noindent Our arguments are  built upon the pioneering works \cite{Bao2013,Ding2018,ERDOS2012,Knowles2017,pillai2014,Yin2014}.
	
	The eigenvalues of sample covariance matrix widely appear in statistical
	applications such as principal component analysis (PCA), factor analysis,
	hypothesis testing. As an instance, in PCA, the eigenvalue of covariance matrix represents the
	variance of each component of rotated vector. In many practical
	situations, such as financial asset pricing and signal processing
	(see, e.g. \cite{BaiandNg2002,Cai2015}), the data observed are usually
	of high-dimension but are actually sparse in nature. A common practical act is to keep only the small portion of
	components of large variances suggested by the eigenvalues of
	covariance matrix with others discarded. This way of data manipulation
	often acts as an effective dimension reduction technique in practice.
	However, most of the time, the eigenvalues of population covariance
	matrix are unknown. At this time, the largest sample eigenvalue performs
	as a good candidate for the inference of properties of population
	eigenvalues. See e.g. \cite{bao2015,johnstone2008,Johnstone&Nadler2017,Onatski2009}.
	
	We summarize the contributions of this paper here. We first prove the local law and Tracy-Widom limit for the sample covariance matrices of elliptical high-dimensional random vectors. The weak correlations cross the coordinates differentiate the model from the existing studies, hence facilitate to the applications in more general scenarios. The correlations also bring new challenge to the technical proofs, such as calculating the large deviation bounds and fluctuation averaging errors. Corresponding lemmas in this paper can be of independent interest.  Moreover, we relax the typical moment conditions  in existing results, e.g.\cite{ZhouWang2018a,ZhouWang2018b}, to the sharper tail probability assumption \eqref{eq:tail condition intro}.  We prove that the tail probability assumption is  sufficient for the Tracy-Widom limit under elliptical distribution.   This result will undoubtedly push forward  the research on random matrix theory with elliptically distributed data. 
	
	This article is organized as follows. In Section \ref{sec:2}, we
	introduce our notation and list the basic conditions. In Section \ref{sec:3},
	we present our main results and the sketch of proof. We first prove the deformed local law which is a bound of
	difference between the Stieltjes transform
	of empirical distribution of sample eigenvalues and that
	of its limiting counterpart under some bounded restrictions. This result will be our starting point
	to derive the limiting distribution of the rescaled largest eigenvalue. Meanwhile, it may be of interest for its own right since a number
	of other useful consequences regarding sample covariance matrix can be obtained from it, such
	as eigenvector delocalization and eigenvalue spacing (see e.g. \cite{locallawlecturenotes}). Taking the deformed local law as an input, the next step is to prove a strong average local law with some restrictions on bounded support and the  four leading moments of $\xi$. Finally, we show that the limiting
	distribution of the rescaled largest eigenvalue does not depend on
	the specific distribution of matrix entries under a tail probability condition. Comparison with the
	normally distributed data then indicates that the limiting distribution
	of the rescaled largest eigenvalue is the Tracy-Widom (TW) law. In Sections \ref{sec:Preliminary-results} to \ref{sec:9},
	we give the detailed proof of our main theorems, while several lemmas and results will be also put into the Appendices.

	\section{\label{sec:2}Notation and basic conditions}
	
	Throughout this article, we set $C>0$ to be a constant whose value
	may be different from line to line. $\mathbb{Z}$, $\mathbb{Z}_{+}$,
	$\mathbb{R}$, $\mathbb{R}_{+}$, $\mathbb{C}$, $\mathbb{C}^{+}$
	denote the sets of integers, positive integers, real numbers, positive real
	numbers, complex numbers and the upper half complex plane respectively.
	For $a,b\in\mathbb{R}$, $a\land b=\min(a,b)$ and $a\lor b=\max(a,b)$.
	$\imath=\sqrt{-1}.$ For a complex number $z$, $\operatorname{Re} z$ and $\operatorname{Im} z$
	denote the real and imaginary parts of $z$ respectively. For a matrix
	$A=(A_{ij})$, ${\rm Tr}A$ denotes the trace of $A$, $\|A\|$ denotes
	the spectral norm of $A$ equal to the largest singular value
	of $A$ (usually we use $\|\cdot\|_2$ as well) and $\|A\|_{F}$ denotes the Frobenius norm of $A$ equal to
	$(\sum_{ij}|A_{ij}|^{2})^{1/2}$. For $M\in\mathbb{Z}_{+}$, ${\rm diag}(a_{1},\dots,a_{M})$
	denote the diagonal matrix with $a_{1},\dots,a_{M}$ as its diagonal
	elements. For two sequences of numbers $\{a_{N}\}_{N=1}^{\infty}$,
	$\{b_{N}\}_{N=1}^{\infty}$, $a_{N}\asymp b_{N}$ if there exist
	constants $C_{1},C_{2}>0$ such that $C_{1}|b_{N}|\le|a_{N}|\le C_{2}|b_{N}|$
	and $O(a_{N})$ and $o(a_{N})$ denote the sequences such that $|O(a_{N})/a_{N}|\le C$
	with some constant $C>0$ for all large $N$ and $\lim_{N\to\infty}o(a_{N})/a_{N}=0$.
	$I$ denotes the identity matrix of appropriate size. For a set $A$,
	$A^{c}$ denotes its complement (with respect to some whole set which
	is clear in the context). For some integer $M\in\mathbb{Z}_{+}$,
	$\chi_{M}^{2}$ denotes the chi-square distribution with degrees of
	freedom $M$. For a measure $\varrho$, ${\rm supp}(\varrho)$ denotes
	its support. For any finite set $T$, we let $|T|$ denote the cardinality
	of $T$. For any event $\Xi$, $\mathbf{1}(\Xi)$ denotes the indicator
	of the event $\Xi$, equal to $1$ if $\Xi$ occurs and
	$0$ if $\Xi$ does not occur. For any $a,b\in\mathbb{R}$ with $a\le b$,
	$\mathbf{1}_{[a,b]}(x)$ is equal to $1$ if $x\in[a,b]$ and $0$
	if $x\notin[a,b]$.
	
	We consider the $M\times N$ data matrix ${X}$ as follows. Let $U=(\mathbf{u}_{1},\dots,\mathbf{u}_{N})$ and $\mathscr{D}={\rm diag}(\xi_{1},\dots,\xi_{N})$. Then the corresponding column vectors and data matrices are
	\[
	{ X:=(\mathbf{x}_{1},\dots,\mathbf{x}_{N}),\qquad \Sigma^{1/2}U=:(\mathbf{r}_1,\dots,\mathbf{r}_N)}
	,\]
	\[
	{ X=\Sigma^{1/2}U\mathscr{D}=\Sigma^{1/2}(\xi_1\mathbf{u}_1,\dots,\xi_N\mathbf{u}_N)=(\xi_1\mathbf{r}_1,\dots,\xi_N\mathbf{r}_N)}
	\]
	where $\mathbf{u}_i$'s are i.i.d. from $U(\mathbb{S}^{M-1})$ and $\xi_i$'s are i.i.d. nonnegative random variables independent with all $\mathbf{u}_i$'s.
	Our sample covariance matrix is $XX^*$. So $\xi_i$ has absorbed the usual normalised factor $1/\sqrt{N}$ for all $i$.
	For convenience we write $\hat\xi_i:=\sqrt{N}{\xi}_i$.
	
	The $\Sigma^{1/2}$ in the above equation can be also replaced by some general $M\times M$ matrix $A$ with $AA^*=\Sigma$. We claim that the technical proof is totally the same, using  the singular value decomposition $A=U_AD_AV_A$ and the observation that the distribution of $\mathbf{u}_{1},\dots,\mathbf{u}_{N}$
	is orthogonally invariant. For the same reason, without loss of generality we assume
	that $\Sigma={\rm diag}(\sigma_{1},\dots,\sigma_{M})$, where $\sigma_{1},\dots,\sigma_{M}$ denote the descending eigenvalues.
	
	Denote the empirical spectral density of $\Sigma$ as
	\[
	\pi:=\frac{1}{M}\sum_{i=1}^{M}\delta_{\sigma_i}.
	\]
	Following the general assumptions on $\Sigma$ in the literature, we suppose that for a small enough constant $\tau>0$,
	\begin{equation}\label{2.1}
	\sigma_1\leqslant\tau^{-1}\qquad \pi([0,\tau])\leqslant1-\tau.
	\end{equation}
	Fix $0<\tau<1$, and define
	\[
	\mathbf{D}\equiv\mathbf{D}(\tau,N):=\{z=E+\imath\eta\in\mathbb{C}^{+}:|z|\ge\tau,|E|\le\tau^{-1},N^{-1+\tau}\le\eta\le\tau^{-1}\}.
	\]
	For $z:=E+\imath\eta\in\mathbb{C}^{+}$, further define the following quantities
	\[
	W =X^{*} X,\qquad \mathcal{W}=XX^{*},
	\]
	while the respective Green functions of $W$ and $\mathcal{W}$ are
	\[
	G(z)= (W-zI)^{-1},\qquad{\cal G}(z)=({\cal W}-zI)^{-1}.
	\]
	We denote the respective empirical spectral density of $W$ and $\mathcal{W}$ as
	\[
	\rho_{W}:=\frac{1}{N}\sum_{i=1}^{N}\delta_{\lambda_i(W)},\qquad\rho_{\mathcal{W}}:=\frac{1}{M}\sum_{i=1}^{M}\delta_{\lambda_i(\mathcal{W})}.
	\]
	The stieltjes' transforms of $\rho_{W}$ and $\rho_{\mathcal{W}}$ are given by
	\[
	m_N^{W}:=\int\frac{1}{x-z}\rho_{W}=\frac{1}{N}\operatorname{Tr}G(z),  \qquad m_N^{\mathcal{W}}:=\int\frac{1}{x-z}\rho_{\mathcal{W}}=\frac{1}{M}\operatorname{Tr}\mathcal{G}(z).
	\]
	Throughout the rest of the paper, we denote $m_N(z):=m_N^{W}(z)$ for simplification.
	
	It's easy to see that the eigenvalues of $W$ and ${\cal W}$ are the
	same up to $|M- N|$ number of $0$s. We denote the descending
	eigenvalues of $W$ and ${\cal W}$ in the unified manner as $\lambda_{1}\ge\lambda_{2}\ge\cdots\ge\lambda_{M\lor N}$,
	where $\lambda_{1},\dots,\lambda_{M}$ and $\lambda_{1},\dots,\lambda_{N}$
	are understood to be the eigenvalues of ${\cal W}$ and $W$ respectively.
	In particular, $\lambda_{M\land N+1},\dots,\lambda_{M\lor N}$ are
	all $0$. Consequently,
	\begin{equation}
	\phi_N\rho_{\mathcal{W}}=\rho_{W}+(1-\phi_N)\delta_0
	\end{equation}
	and
	\begin{equation}\label{eq:relationship between m_w and m_N}
	\phi_N m_N^{\mathcal{W}}(z)=\frac{1-\phi_N}{z}+m_N(z),
	\end{equation}
	where $\phi_N:=M/N$. We may suppress the subscript $N$ and use $\phi$ directly hereafter.
	
	Denote the index set ${\cal I}=\{1,\dots,N\}$. For $T\subset{\cal I}$, we introduce
	the notation $X^{(T)}$ to denote the $M\times(N-|T|)$ minor of $X$
	obtained from removing all the $i$th columns of $X$ for $i\in T$. In
	particular, $X^{(\emptyset)}=X$. For convenience, we briefly write
	$(\{i\})$, $(\{i,j\})$ and $\{i,j\}\cup T$ as $(i)$, $(i,j)$
	and $(ijT)$ respectively. Correspondingly,
	\[
	W^{(T)}=(X^{(T)})^{*} X^{(T)},\qquad{\cal W}^{(T)}=X^{(T)}(X^{(T)})^{*},
	\]
	and
	\[
	G^{(T)}(z)=(W^{(T)}-zI)^{-1}, \quad{\cal G}^{(T)}(z)=({\cal W}^{(T)}-zI)^{-1},\quad
	m_{N}^{(T)}(z)=\frac{1}{N}{\rm Tr}G^{(T)}(z).
	\]
	Throughout this article, we denote $X_{ij}$ as the $(i,j)$-th
	entry of a matrix $X$. In particular, in the minor $X_{ij}^{(T)}$ with $i,j\notin T$, we keep the original  indices of $X$.
	
	In the following, we present a notion introduced in \cite{Erdos2013}.
	It provides a simple way of systematizing and making precise statements for two families of random variables $A,B$
	of the form ``$A$ is bounded with high probability by $B$ up to
	small powers of $N$'' .
	\begin{defn}[Stochastic domination]\mbox{}{\newline}
		(a). For two families of nonnegative random variables
		\[
		A=\{A_{N}(t):N\in\mathbb{Z}_{+},t\in T_{N}\},\qquad B=\{B_{N}(t):N\in\mathbb{Z}_{+},t\in T_{N}\},
		\]
		where $T_{N}$ is
		a possibly $N$-dependent parameter set, we say that $A$ is stochastically
		dominated by $B$, uniformly in $t$ if for all (small) $\varepsilon>0$
		and (large) $D>0$ there exists $N_{0}(\varepsilon,D)\in\mathbb{Z}_{+}$
		such that 	as $N\ge N_{0}(\varepsilon,D)$,
		\[
		\sup_{t\in T_{N}}\mathbb{P}\big(A_{N}(t)>N^{\varepsilon}B_{N}(t)\big)\le N^{-D}.
		\]
		If $A$ is stochastically dominated
		by $B$, uniformly in $t$, we use notation $A\prec B$ or $A=O_{\prec}(B)$.
		Moreover, for some complex family $A$ if $|A|\prec B$ we also write $A=O_{\prec}(B)$. \\
		(b). Let $A$ be a family of random matrices and $\zeta$ be a family of nonnegative random variables. Then we denote $A=O_{\prec}(\zeta)$ if $A$ is dominated under weak operator norm sense, i.e. $|\langle\mathbf{v},A\mathbf{w}\rangle|\prec\zeta\|\mathbf{v}\|_2\|\mathbf{w}\|_2$  for any deterministic vectors $\mathbf{v}$ and $\mathbf{w}$.\\
		(c). For two sequences of numbers $\{a_{N}\}_{N=1}^{\infty}$,
		$\{b_{N}\}_{N=1}^{\infty}$, $a_{N}\prec b_{N}$ if for all $\epsilon>0$, $a_N\leq N^{\epsilon}b_N$.
	\end{defn}
	
	\begin{rem}
		The stochastic domination throughout this article holds uniformly
		for the matrix indices and $z\in\mathbf{D}$ (or the set $\mathbf{D}^{e}$ defined later).
		For simplicity, in the proof of each result, we omit the explicit
		indication of this uniformity.
	\end{rem}

	The discussion in this paper highly relies on the following global definitions.
	\begin{defn}[High probability event]
		We say that an $N$-dependent event $\Omega$ holds with overwhelming high probability if there exists constant $c>0$ independent of $N$, such that
		\begin{equation}
		\mathbb{P}(\Omega)\geqslant 1-\exp{(-N^c)},
		\end{equation}
		for all sufficiently large $N$.
	\end{defn}
	\begin{defn}[Bounded support condition]\label{def:Bounded support cond}
		We say that an $N$-dependent random variable $x:=x(N)$ satisfies the bounded support condition with $q\equiv q(N)$ if
		\begin{equation}\label{2.5}
		\mathbb{P}(|x|\leqslant q)\geqslant 1-\exp{(-N^c)},
		\end{equation}
		for some $c>0$.
	\end{defn}
	\begin{rem}
		Note that if $x$ satisfies Condition \eqref{2.5}, it is equivalent to that the event $\{|x|\leqslant q\}$ holds with high probability. Consequently, we can neglect the bad event $\{|x|>q\}$. In other words, in our proof, we are in a high probability whole set $\Omega$. For instance, under $\Omega$, the entries of a data matrix satisfy the bounded support condition.
	\end{rem}
	
	Throughout this article, we assume the following conditions.
	\begin{condition}
		\label{cond:1}$N\to\infty$ with $M\equiv M(N)\to\infty$ such that
		$\phi:=M/N\rightarrow \phi_0\in[a,b]$ for all large $N$ where $a<b$ are two positive
		numbers.
	\end{condition}
	
	\begin{condition}
		\label{cond:2}$\xi_{1},\dots,\xi_{N}$ are independent nonnegative random variables such that $\mathbb{E}\xi_{i}^{2}=\phi$ and
		\begin{equation}\label{2.6}
		\lim_{s\rightarrow\infty}\limsup_{N\to \infty}s^2\mathbb{P}(|\hat{\xi}_i^2-M|\geqslant\sqrt{M} s)=0,
		\end{equation}
		for all $i\in\{1,\cdots,N\}$. Recall that $\xi_i=\hat{\xi}_i/\sqrt{N}$ and $\hat{\xi}_i^2$ has left tight support i.e. $\hat{\xi}^2\ge0$.
	\end{condition}
	We remark that in Condition \ref{cond:2}, the general choice of $s$ diverges with $N$, e.g. $s=N^{1/2-\epsilon}$ in the proof of Theorem \ref{thm:universality} below.
	
	We put an alternative restriction on $\xi_1,\dots,\xi_N$.
	\begin{condition}
		\label{cond:3} $\xi_{1},\dots,\xi_{N}$ are independent nonnegative random variables such that $\mathbb{E}\xi_{i}^{2}=\phi$ and $\xi^2_i-\phi$ has bounded support $q$ in the sense of Definition \ref{def:Bounded support cond} with
		\[
		N^{-1/2}\log N\leqslant q\leqslant N^{-c}
		\]
		for some $c<1/2$ and all $i\in\{1,\cdots,N\}$.
	\end{condition}
	
	\begin{rem}
		Condition \ref{cond:2} excludes some elliptical distributions, such as multivariate student-\emph{t }distributions and normal scale mixtures. The limiting spectral distribution of sample covariance matrix from these distributions do not follow the Mar\v{c}enko-Pastur equation
		(\ref{eq:MP equation}), see (\cite{elkaroui2009,LiandYao2018}),
		and hence is out of scope of this article. Actually there are still
		a wide range of distributions satisfying Condition \ref{cond:2},
		including the multivariate Pearson type II distribution and the family
		of Kotz-type distributions, see the examples and Table 1 in \cite{ZhouWang2018a}.
		In particular, if $\xi^{2}$ can be written as $\xi^{2}=N^{-1}(y_{1}^2+\cdots+y_{M}^2)$
		with $y_{1},\dots,y_{M}$ being an positive i.i.d. sequence such that
		$\mathbb{E}y_{1}=1$ and $\mathbb{E}y_{1}^{4}<\infty$,
		then $\xi^2$ satisfies Condition \ref{cond:2}.
	\end{rem}
	\begin{rem}
		\eqref{2.6} is weaker than
		\[
		\limsup_{N\to \infty}\frac{1}{M}\mathbb{E}|\hat{\xi}_i^2-M|^2<\infty
		\]
		but stronger than
		\[
		\limsup_{N\to \infty}\frac{1}{M}\mathbb{E}|\hat{\xi}_i^2-M|^{2-\delta}<\infty
		\]
		for arbitrary $\delta>0$.
	\end{rem}
	
	One can check that Conditions \ref{cond:1} and \ref{cond:3} are
	sufficient for Theorem 1.1 of \cite{BaiandZhou2008}. Hence
	we have the following result.
	\begin{lem}
		\label{lem:global law}Suppose, given Conditions \ref{cond:1} and
		\ref{cond:3}, $\pi$ converges weakly to a probability distribution
		$\pi_{0}$ and $\phi\to\phi_{0}\in(0,\infty)$. Then, almost surely,
		$\varrho_{W}$ converges weakly to a deterministic limiting probability
		distribution $\varrho_0$ and for any $z\in\mathbb{C}^{+}$, almost
		surely, $m_{N}(z)$ converges to the Stieltjes transform of $\varrho_0$
		which we denote as $m_0(z)$. Moreover, for all $z\in\mathbb{C}^{+}$,
		$m_0(z)$ is the unique value in $\mathbb{C}^{+}$ satisfying the equation
		\begin{equation}
		z=-\frac{1}{m_0(z)}+\phi_{0}\int\frac{x}{1+xm_0(z)}\pi_{0}({\rm d}x).\label{eq:MP equation}
		\end{equation}
	\end{lem}
	
	\begin{rem}
		If we replace $\pi_{0}$ and $\phi_{0}$ by their finite sample counterparts
		$\pi$ and $\phi$ in (\ref{eq:MP equation}) and solve for $m$ for
		each $z\in\mathbb{C}^{+}$, we obtain a Stieltjes transform of a deterministic
		probability distribution. Throughout this article, we denote this
		deterministic probability distribution and its Stieltjes transform
		as $\varrho$ and $m(z)$ respectively. By Lemma \ref{lem:global law},
		when $N$ is large $\varrho_{W}$ and $m_{N}(z)$ are close to $\varrho$
		and $m(z)$. The aim of next section is to evaluate the bound of $|m_{N}(z)-m(z)|$.
	\end{rem}
	
	We define the function $f:\mathbb{C}\to\mathbb{C}$,
	\begin{equation}\label{def:f}
	f(w)=-\frac{1}{w}+\phi\int\frac{x\pi({\rm d}x)}{1+wx},
	\end{equation}
	and assume that
	\begin{equation}\label{2.8}
	f^{\prime}(-\mathbf{c})=0,\quad 0<\liminf_{N\to\infty}\sigma_{M}\le\limsup_{N\to\infty}\sigma_{1}<\infty,\quad\limsup_{N\to\infty}\sigma_{1}\mathbf{c}<1.
	\end{equation}
	for $\mathbf{c}\in(0,\sigma_{1}^{-1})$.
	Let $\lambda_{+}:=f(-\mathbf{c})$, so it can be shown that $\lambda_{+}$ is the rightmost endpoint of ${\rm supp}(\varrho)$
	(see the discussion on page 4 of \cite{Bao2013} or Lemma 2.4 of
	\cite{Knowles2017}), i.e., the edge of $\varrho$.
	
	For $\tau,\tau^{\prime}\in(0,\infty)$, $N\in\mathbb{Z}_{+}$, define
	\begin{equation}
	\mathbf{D}^{e}\equiv\mathbf{D}^{e}(\tau,\tau^{\prime},N):=\{z\in\mathbf{D}(\tau,N):E\in[\lambda_{+}-\tau^{\prime},\lambda_{+}+\tau^{\prime}]\}
	\end{equation}
	as the subset of $\mathbf{D}$ with the real part of $z$ restricted
	to a small closed interval around $\lambda_{+}$.
	
	Also we define the distance to the rightmost edge as
	\begin{equation}
	\kappa\equiv\kappa_{E}:=|E-\lambda_{+}|\qquad\text{for}\quad z=E+i\eta.
	\end{equation}
	
	Now we introduce some  definitions before presenting our main results.
	\begin{defn}[Linearizing block matrix]
		For $z\in\mathbb{C}_+$, we define the $(N+M)\times(N+M)$ block matrix (no commas in matrices)
		\begin{equation}
		H:=
		\begin{pmatrix}
		0 & X\\
		X^{*} &0
		\end{pmatrix},
		\end{equation}
		and
		\begin{eqnarray}
		\mathcal{H}&:=
		&\begin{pmatrix}
		-I_{M\times M} & X\\
		X^{*} &-zI_{N\times N}
		\end{pmatrix}^{-1}\\
		&=&\begin{pmatrix}
		z\mathcal{G} &\mathcal{G} X\\
		(\mathcal{G} X)^{*} &G
		\end{pmatrix}
		=\begin{pmatrix}
		z\mathcal{G} &XG\\
		(XG)^{*} &G
		\end{pmatrix}.
		\end{eqnarray}
	\end{defn}
	
	\begin{defn}[Deterministic limit of $\mathcal{H}$]
		We define the deterministic limit $\Pi$ of  $\mathcal{H}$ as
		\begin{equation}
		\Pi(z):=
		\begin{pmatrix}
		-(1+m(z)\Sigma)^{-1} &0\\
		0 &m(z)I_{N\times N}
		\end{pmatrix}.
		\end{equation}
	\end{defn}
	
	Define the control parameters by
	\begin{eqnarray*}
		\Lambda\equiv\Lambda(z)  :=  \max_{i,j\in{\cal I}}|G_{ij}(z)-\delta_{ij}m(z)|,\qquad\Lambda_{o}\equiv\Lambda_{o}(z):=\max_{i,j\in{\cal I},i\ne j}|G_{ij}(z)|,\\
		\Theta\equiv\Theta(z)  :=  |m_{N}(z)-m(z)|,\qquad\Psi_{\Theta}:=\sqrt{\frac{\Im\operatorname{Im} m(z)+\Theta}{N\eta}},\qquad\Xi:=\{\Lambda\le(\log N)^{-1}\},
	\end{eqnarray*}
	where $\delta_{ij}$ denotes the Kronecker delta, i.e. $\delta_{ij}=1$
	if $i=j$, and $\delta_{ij}=0$ if $i\ne j$ and $\Xi$ is a $z$-dependent event. For simplicity of notation, we occasionally omit the variable $z$
	for those $z$-dependent quantities provided no ambiguity occurs.
	
	\section{\label{sec:3}Main results}
	
	\subsection{Deformed local law}
	\begin{thm}[Deformed strong local law]
		\label{thm:strong local law}Given Conditions \ref{cond:1}, \ref{cond:3} as well as \eqref{2.1} and \eqref{2.8},
		there exists a constant $\tau^{\prime}$ depending only on $\tau$
		such that
		\begin{eqnarray}
		&\Lambda(z)\prec \sqrt{\frac{\operatorname{Im} m(z)}{N\eta}}+\frac{1}{N\eta}+q,\label{eq:entrywise local law}\\
		&|m_{N}(z)-m(z)|\prec\Big(\min\{q,\frac{q^2}{\sqrt{\kappa+\eta}}\}+\frac{1}{N\eta}\Big)\label{eq:average local law},
		\end{eqnarray}
		uniformly for $z\in\mathbf{D}^{e}(\tau,\tau^{\prime},N)$.
	\end{thm}
	\begin{rem}
		Theorem \ref{thm:strong local law} can be strengthened in a simultaneous sense for $z\in\mathbf{D}^{e}(\tau,\tau^{\prime},N)$, using the Lipschitz continuity of $G_{ij}(z), m(z),\Psi_{\Lambda}(z),\Psi_{m}(z)$ and the fact that $\Psi_{\Lambda}(z),\Phi_{m}(z)\ge\frac{1}{N}$ on $\mathbf{D}^{e}(\tau,\tau^{\prime},N)$, where
		\[
		\Phi_{\Lambda}(z):=\sqrt{\frac{\operatorname{Im} m(z)}{N\eta}}+\frac{1}{N\eta}+q, \quad \Phi_{m}(z):=\min\{q,\frac{q^2}{\sqrt{\kappa+\eta}}\}+\frac{1}{N\eta}.
		\]
		The proof is essentially  the same as the one in (\ref{eq:uniform local law}) of Appendix \ref{Appendix C}. We just put down the conclusions as follows,
		\begin{equation}
		\sup_{z\in\mathbf{D}^{e}}\max_{i,j}\frac{\Lambda(z)}{\Phi_{\Lambda}(z)}\prec1,\quad\sup_{z\in\mathbf{D}^{e}}\frac{|m_N(z)-m(z)|}{\Phi_{m}(z)}\prec1,
		\end{equation}
		under the assumptions in Theorem \ref{thm:strong local law}.
	\end{rem}
	A direct consequence is the following theorem.
	\begin{thm}\label{thm:rigidity}
		Under the assumptions in Theorem \ref{thm:strong local law}, we have
		\begin{equation}\label{eq:no eigenvalues outside sepectrum}
		\|\mathcal{H}\|^2\leqslant \lambda_{+}+N^{\epsilon}(q^2+N^{-2/3}).
		\end{equation}
		Furthermore, for any real numbers $a,b$ such that $a\le b$, define $\mathfrak{n}_{N}(a,b)=\int_{a}^{b}\varrho_{N}({\rm d}x)$
		and $\mathfrak{n}(a,b)=\int_{a}^{b}\varrho({\rm d}x)$.
		Then there exists a constant $\tau^{\prime}$ depending only on $\tau$
		such that for any $E_{1},E_{2}\in\{\operatorname{Re} z:z\in\mathbf{D}^{e}(\tau,\tau^{\prime},N)\}$,
		\begin{equation}
		|\mathfrak{n}_{N}(E_{1},E_{2})-\mathfrak{n}(E_{1},E_{2})|\prec N^{-1}+q^3+q^2(\sqrt{\kappa_{E_1}}-\sqrt{\kappa_{E_2}}).\label{eq:local law on small scale}
		\end{equation}
		Consequently, we have for $q\le N^{-1/3}$,
		\begin{equation}\label{eq:rigidity of eigenvalues}
		|\lambda_i-\gamma_i|\prec i^{-1/3}N^{-2/3}+q^2,
		\end{equation}
		uniformly in $i$ such that $\gamma_i\in[\lambda_{+}-c,\lambda_{+}]$ for some $c>0$, where
		\[
		\gamma_i:=\sup_x\{\int_{x}^{\infty}\varrho(x)dx>\frac{i-1}{N}\}.
		\]
	\end{thm}
	The proof of this theorem is the same as the one in \cite{Ding2018}, and we summarize the main arguments in Appendix \ref{Appendix C}.
	
	\subsection{Edge universality with small support}
	\begin{thm}[Edge universality with small support]\label{thm:Edge universality with small support}
		Suppose $X^W$ and $X^V$ are two random matrices satisfying Conditions \ref{cond:1} and \ref{cond:3} with $q\leqslant N^{-5/12+\zeta}$ for some small $\zeta>0$. Then there exist some positive constants $\epsilon,\delta>0$ such that 	for any $s\in \mathbb{R}$
		\begin{equation}
		\begin{split}
		\mathbb{P}^{V}{(N^{2/3}(\lambda_1-\lambda_{+})\le s-N^{-\epsilon})}-N^{-\delta}\le \mathbb{P}^{W}&{(N^{2/3}(\lambda_1-\lambda_{+})\le s)}\\
		&\le\mathbb{P}^{V}{(N^{2/3}(\lambda_1-\lambda_{+})\le s+N^{-\epsilon})}+N^{-\delta},
		\end{split}
		\end{equation}
		where $\mathbb{P}^V$ and $\mathbb{P}^W$ denote the laws of $X^V$ and $X^W$ respectively.
	\end{thm}
	\begin{rem}
		Theorem \ref{thm:Edge universality with small support} can be extended to the case of joint distribution of the largest $k$ eigenvalues for any fixed positive integer $k$, that is, for any real numbers $s_{1},\dots,s_{k}$ which
		may depend on $N$, there exist some positive constants $\varepsilon,\delta>0$
		such that for all large $N$
		\begin{multline}
		\mathbb{P}^{V}(N^{2/3}(\lambda_{1}-\lambda_{+})\le s_{1}-N^{-\varepsilon},\dots,N^{2/3}(\lambda_{k}-\lambda_{+})\le s_{k}-N^{-\varepsilon})-N^{-\delta}\\
		\le\mathbb{P}^W(N^{2/3}(\lambda_{1}-\lambda_{+})\le s_{1},\dots,N^{2/3}(\lambda_{k}-\lambda_{+})\le s_{k})\\
		\le\mathbb{P}^V(N^{2/3}(\lambda_{1}-\lambda_{+})\le s_{1}+N^{-\varepsilon},\dots,N^{2/3}(\lambda_{k}-\lambda_{+})\le s_{k}+N^{-\varepsilon})+N^{-\delta}.\label{eq:k universality with small support}
		\end{multline}
	\end{rem}
	
	\subsection{Edge universality with large support}
	\begin{thm}[Rigidity of eigenvalues with large support]\label{thm:Rigidity of eigenvalues with large support}
		Suppose random matrix $X$ satisfies Conditions \ref{cond:1} and \ref{cond:3} with $q\le N^{-c}$ for some constant $c>0$ and suppose moreover that
		\begin{equation}
		\mathbb{E}|\xi_i^2-\phi|^2\le B N^{-1}\log N,
		\end{equation}
		for some constant $B>0$. Then there exists constant $c_1,\tau,\tau^{'}$ such that
		\begin{equation}\label{eq:strong average law with large support}
		\sup_{z\in\mathbf{D}^{e}}\frac{|m_N(z)-m(z)|}{(N\eta)^{-1}}\prec1,
		\end{equation}
		for sufficient large $N$. Moreover, (\ref{eq:strong average law with large support}) implies that with high probability
		\begin{equation}\label{eq:strong rigidity with large support}
		|\lambda_i-\gamma_i|\prec i^{-1/3}N^{-2/3},
		\end{equation}
		uniformly in $i$ such that $\gamma_i\in[\lambda_{+}-c_1,\lambda_{+}]$, and
		\begin{equation}
		\sup_{E\ge\lambda_{+}-c_1}|\mathfrak{n}_{N}(E)-\mathfrak{n}(E)|\prec\frac{1}{N}.
		\end{equation}
	\end{thm}
	
	\begin{thm}[Edge universality with large support]\label{thm:Edge universality with large support}
		Suppose $X^W$ and $X^V$ are two random matrices satisfying the assumptions in Theorem \ref{thm:Rigidity of eigenvalues with large support}. Then there exist some positive constants $\epsilon,\delta>0$ such that for any $s\in\mathbb{R}$
		\begin{equation}
		\begin{split}
		\mathbb{P}^{V}{(N^{2/3}(\lambda_1-\lambda_{+})\le s-N^{-\epsilon})}-N^{-\delta}\le \mathbb{P}^{W}&{(N^{2/3}(\lambda_1-\lambda_{+})\le s)}\\
		&\le\mathbb{P}^{V}{(N^{2/3}(\lambda_1-\lambda_{+})\le s+N^{-\epsilon})}+N^{-\delta},
		\end{split}
		\end{equation}
		where $\mathbb{P}^V$ and $\mathbb{P}^W$ denote the laws of $X^V$ and $X^W$ respectively.
	\end{thm}
	
	\subsection{ Edge universality}
	\mbox{}\par
	Let $\mathbf{u}_{1},\dots,\mathbf{u}_{N}$ be from $U(\mathbb{S}^{M-1})$, and $\tilde{\xi}_{1},\dots,\tilde{\xi}_{N}$ be i.i.d. non-negative random variables such that $\tilde{\xi}_{1}^{2}$ follows $\chi_{M}^{2}/N$ distribution. Assume the independence of $\{\mathbf{u}_{1},\dots,\mathbf{u}_{N}\}$ and $\{\tilde{\xi}_{1},\dots,\tilde{\xi}_{N}\}$.  Let $\tilde{X} :=\Sigma^{1/2} (\tilde{\xi}_{1}\mathbf{u}_{1},\dots,\tilde{\xi}_{N}\mathbf{u}_{N})$,  so  $\tilde{X}$ is a matrix whose columns are i.i.d. Gaussian random vectors. We claim in the following theorem that for elliptically distributed data $X$ and $\Sigma$ satisfying Condition \ref{cond:1}, (\ref{2.1}) and (\ref{2.8}),  the largest eigenvalue of its sample covariance matrix follows the same limiting distribution as the one with $\tilde X$ if  Condition \ref{cond:2} holds.
	\begin{thm}[Edge universality]
		\label{thm:universality}
		Let $ W=X^{*}X$ be an $(N\times N)$ sample covariance matrix with $X$ satisfying Condition \ref{cond:1}, \eqref{2.1} and \eqref{2.8} .
		If Condition \ref{cond:2} holds, then we have for all $s\in\mathbb{R}$
		\begin{equation}
		\lim_{N\rightarrow\infty}\mathbb{P}(N^{2/3}(\lambda_1-\lambda_{+})\le s)=\lim_{N\rightarrow\infty}\mathbb{P}(N^{2/3}(\tilde{\lambda}_1-\lambda_{+})\le s),
		\end{equation}
		where $\tilde{\lambda}_1$ is the largest eigenvalue of  $ \tilde{X}^{*}\tilde{X}$.
	\end{thm}
	
	\begin{cor}[Tracy-Widom law]
		\label{cor:TW law} Under assumptions  in Theorem \ref{thm:universality}, we have
		\[
		\lim_{N\to\infty}\mathbb{P}\big(\gamma N^{2/3}(\lambda_{1}-\lambda_{+})\le s\big)=F_{1}(s),
		\]
		where $\gamma$ is defined by
		\[
		\frac{1}{\gamma^{3}}=\frac{1}{\mathbf{c}^{3}}\Big(1+\phi\int\Big(\frac{\lambda\mathbf{c}}{1-\lambda\mathbf{c}}\Big)^{3}\pi({\rm d}\lambda)\Big),
		\]
		and $F_{1}(s)$ is the type-1 Tracy-Widom distribution \cite{Tracy1994}.
	\end{cor}
	\begin{rem}
		Theorem \ref{thm:universality} can be extended to the case of joint
		distribution of the largest $k$ eigenvalues for any fixed positive
		integer $k$, namely, for any real numbers $s_{1},\dots,s_{k}$ which
		may depend on $N$, there exist some positive constants $\varepsilon,\delta>0$
		such that for all large $N$
		\begin{multline}
		\mathbb{P}(N^{2/3}(\tilde{\lambda}_{1}-\lambda_{+})\le s_{1}-N^{-\varepsilon},\dots,N^{2/3}(\tilde{\lambda}_{k}-\lambda_{+})\le s_{k}-N^{-\varepsilon})-N^{-\delta}\\
		\le\mathbb{P}(N^{2/3}(\lambda_{1}-\lambda_{+})\le s_{1},\dots,N^{2/3}(\lambda_{k}-\lambda_{+})\le s_{k})\\
		\le\mathbb{P}(N^{2/3}(\tilde{\lambda}_{1}-\lambda_{+})\le s_{1}+N^{-\varepsilon},\dots,N^{2/3}(\tilde{\lambda}_{k}-\lambda_{+})\le s_{k}+N^{-\varepsilon})+N^{-\delta}.\label{eq:k universality}
		\end{multline}
		Accordingly, Corollary \ref{cor:TW law} can be extended to the case of joint
		distribution as follows,
		\[
		\big(\gamma N^{2/3}(\lambda_{1}-\lambda_{+}),\dots,\gamma N^{2/3}(\lambda_{k}-\lambda_{+})\big)
		\]
		converges to the $k$-dimensional joint Tracy-Widom distribution.
		Here we use the term ``joint Tracy-Widom distribution" as in Theorem
		1 of \cite{Soshnikov2002}. The extension (\ref{eq:k universality})
		can be proved by  a similar argument to the one in \cite{pillai2014}.
		Hence we do not reproduce the details.
	\end{rem}
	
	\subsection{Sketch of the proof}
	
	First, we show Theorems \ref{thm:strong local law} which will serve as crucial inputs for the proof of Theorem \ref{thm:rigidity}, Theorem \ref{thm:Edge universality with small support}, Theorem \ref{thm:Rigidity of eigenvalues with large support}, Theorem \ref{thm:Edge universality with large support} and Theorem \ref{thm:universality}. The proof strategy essentially dates back to \cite{ERDOS2012,Knowles2017,pillai2014}. We start by studying each entry
	of the Green function $G(z)$. The general target is to show that
	each diagonal element of $G(z)$ is close to $m(z)$ and the off-diagonal
	elements of $G(z)$ are close to $0$ under the bounded support $q$. Before attaining the final
	goal, our first step is to obtain a weaker but still nontrivial version
	of the local law, i.e. $\Lambda(z) \prec (N\eta)^{-1/4}+q$.
	Compared to previous papers e.g. \cite{bao2015,Bao2013,Knowles2017,pillai2014}
	assuming i.i.d. entries in the data matrix, the main difficulty of
	our work is to deal with dependence among entries in each column
	$\mathbf{x}_{i}$, $i=1,\dots,N$. Due to the dependence, the usual
	large deviation bounds for i.i.d. vectors in \cite{bao2015,Bao2013,Knowles2017,pillai2014}
	are no longer applicable. In Section \ref{sec:Preliminary-results},
	we present the large deviation inequalities (Lemma \ref{lem:large deviation})
	for uniformly spherically distributed random vectors and give their proofs
	in the Appendix \ref{Appendix A}.
	Moreover, the radius variable $\xi_{i}$ causes extra randomness which
	is the reason for the introduction of Condition \ref{cond:3} as
	to reduce the variation. Also due to dependence, the strategy
	in \cite{Knowles2017} to expand the matrix $X$ along both rows
	and columns cannot be applied. We tackle this issue by expanding $X$
	only along columns and bounding the errors emerging from the finite
	sample approximation of the Mar\v{c}enko-Pastur equation. Then the
	weak deformed local law can be achieved by a bootstrapping procedure. Next,
	the weaker bound  is strengthened to
	$$\Lambda(z)\prec \sqrt{\frac{\operatorname{Im} m(z)}{N\eta}}+\frac{1}{N\eta}+q$$
	via the self-improving steps utilizing a so-called fluctuation averaging
	argument.
	This procedure involves estimating the conditional expectation of $ \frac{1}{N}\sum_{i\in{\cal I}}\mathbf{x}_i^{*}\mathcal{G}^{(i)}\Sigma(m_N^{(i)}\Sigma+I)^{-1}\mathbf{x}_i$. The difficulty lies in not only  the dependence among each column but also randomness in $(m_N^{(i)}\Sigma+I)^{-1}$. In order to handle these difficulties, we expand $\mathbf{x}_i^{*}\mathcal{G}^{(i)}$ and $(m_N^{(i)}\Sigma+I)^{-1}$ respectively. It turns out to be several weakly correlated monomials of quadratic forms with entries of $\mathcal{G}^{(i)}$ as coefficients. For these Green function entries, we further expand them by resolvent identities.
	One can refer to Appendix \ref{Appendix B}  for the  details.
	
	With (\ref{eq:average local law}) at hand, Theorem \ref{thm:rigidity}
	follows from a standard argument similar to Proposition
	9.1 of \cite{locallawlecturenotes}, and the Helffer-Sj\"ostrand argument, see e.g. Theorem 2.8
	and Appendix C of \cite{locallawlecturenotes} or (8.6) of \cite{pillai2014}. For the readers' convenience, we write down the details of the proof of Theorem $\ref{thm:rigidity}$ in Appendix \ref{Appendix C}.
	
	For Theorem \ref{thm:Edge universality with small support}, we use the Green function comparison method. The strategy follows \cite{pillai2014}  with a Lindeberg-type column by column replacement due to the dependence within each column of $X$. The details will be provided in Section \ref{sec:6}.
	
	The establishment of Theorem \ref{thm:Rigidity of eigenvalues with large support} and Theorem \ref{thm:Edge universality with large support} is the key step to prove Theorem \ref{thm:universality}. Roughly speaking, we find that the strong average local law holds with larger support  and some mild restrictions on the four leading  moments of $X$.  Such moment restrictions can be further relaxed to the tail probability Condition \ref{cond:2} using the truncation technique, which concludes Theorem \ref{thm:universality}. The main tool is still the Green function comparison method, while the details are put in Sections \ref{sec:8} and \ref{sec:9}.

	\section{\label{sec:Preliminary-results}Preliminary results}
	
	In this section, we present some preliminary results that will be
	used in the derivation of our main theorems in Sections \ref{sec:5}
	and \ref{sec:6}. Lemma \ref{lem:resolvent} is by Shur's complement formula, whose proof can be found in Lemma 4.2 of \cite{Erdos2012_bulk}. The proof of Lemmas \ref{lem:equivalence} and \ref{lem:large deviation}
	are given in Appendix \ref{Appendix A}. Lemma \ref{lem:norm inequality} is by elementary linear algebra whose proof is omitted.

	\begin{lem}
		\label{lem:resolvent}Under the above notation, for any $T\subset{\cal I}$
		\begin{eqnarray*}
			G_{ii}^{(T)}(z) & = & -\frac{1}{z+z\mathbf{x}_{i}^{*}{\cal G}^{(iT)}(z)\mathbf{x}_{i}},\qquad\forall i\in{\cal I}\backslash T,\\
			G_{ij}^{(T)}(z) & = & zG_{ii}^{(T)}(z)G_{jj}^{(iT)}(z)\mathbf{x}_{i}^{*}{\cal G}^{(ijT)}(z)\mathbf{x}_{j},\qquad\forall i,j\in{\cal I}\backslash T,i\ne j,\\
			G_{ij}^{(T)}(z) & = & G_{ij}^{(kT)}(z)+\frac{G_{ik}^{(T)}(z)G_{kj}^{(T)}(z)}{G_{kk}^{(T)}(z)},\qquad\forall i,j,k\in{\cal I}\backslash T,i,j\ne k.
		\end{eqnarray*}
	\end{lem}
	
	\begin{lem}
		\label{lem:equivalence}Let $\{X_{N}\}_{N=1}^{\infty}$ be a sequence
		of random variables and $\Phi_{N}$ be deterministic. Suppose $\Phi_{N}\ge N^{-C}$
		holds for large $N$ with some $C>0$, and that for all $p$ there
		exists a constant $C_{p}$ such that $\mathbb{E}|X_{N}|^{p}\le N^{C_{p}}$.
		Then we have the equivalence
		\[
		X_{N}\prec\Phi_{N}\Leftrightarrow\mathbb{E}X_{N}^{p}\prec\Phi_{N}^{p}\qquad\text{for any fixed }p\in\mathbb{N}.
		\]
	\end{lem}
	
	\begin{lem}
		\label{lem:norm inequality}Let $A,B$ be two matrices with $AB$
		well-defined. Then
		\begin{eqnarray*}
			|{\rm Tr}(AB)| & \le & \|A\|_{F}\|B\|_{F},\\
			\|AB\| & \le & \|A\|\|B\|,\\
			\|AB\|_{F} & \le & \min\{\|A\|_{F}\|B\|,\|A\|\|B\|_{F}\}\leqslant \|A\|_{F}\|B\|_{F},\\
			\|A+B\|_{F}&\leqslant&\|A\|_{F}+\|B\|_{F},\\
			|{\rm Tr}(AB)|&\leqslant&\|A\| {\rm Tr}|B|.
		\end{eqnarray*}
	\end{lem}
	
	\begin{lem}
		\label{lem:large deviation}Let $\mathbf{u}=(u_{1},\dots,u_{M})^{*}$,
		$\tilde{\mathbf{u}}=(\tilde{u}_{1},\dots,\tilde{u}_{M})^{*}$ be $U(\mathbb{S}^{M-1})$
		random vectors, $A=(a_{ij})$  an $M\times M$ matrix and $\mathbf{b}=(b_{1},\dots,b_{M})^{*}$
		an $M$-dimensional vector, where $A$ and $\mathbf{b}$ may be
		complex-valued and $\mathbf{u},\tilde{\mathbf{u}},A,\mathbf{b}$ are
		independent. Then as $M\to\infty$
		\begin{eqnarray}
		|\mathbf{b}^{*}\mathbf{u}| & \prec & \sqrt{\frac{\|\mathbf{b}\|^{2}}{M},}\label{eq:10.1-1}\\
		|\mathbf{u}^{*}A\mathbf{u}-\frac{1}{M}{\rm Tr}A| & \prec & \frac{1}{M}\|A\|_{F},\label{eq:10.2}\\
		\Big|\mathbf{u}^{*}A\tilde{\mathbf{u}}\Big| & \prec & \frac{1}{M}\|A\|_{F}.\label{eq:10.4-2}
		\end{eqnarray}
		Moreover, if $\mathbf{u},\tilde{\mathbf{u}},A,\mathbf{b}$ depend
		on an index $t\in T$ for some set $T$, then the above domination
		bounds hold uniformly for $t\in T$.
	\end{lem}
	
	Recalling the definition of $\kappa$, we then introduce the following two results whose proof can be found in Lemmas A.4
	and A.5 of \cite{Knowles2017}. In particular, the edge regularity condition required in \cite{Knowles2017} is encompassed in \eqref{2.8}.
	\begin{lem}
		\label{lem:basic property of m} Fix $\tau>0$. Given  assumption \eqref{2.8}, there exists $\tau^{\prime}>0$ such that for any
		$z\in\mathbf{D}^{e}(\tau,\tau^{\prime},N)$ we have
		\begin{eqnarray}
		\operatorname{Im} m(z) & \asymp & \begin{cases}
		\sqrt{\kappa+\eta} & \text{if }E\in{\rm supp}(\varrho),\\
		\frac{\eta}{\sqrt{\kappa+\eta}} & \text{if }E\notin{\rm supp}(\varrho),
		\end{cases}\nonumber \\
		|1+m(z)\sigma_{i}| & \ge & \tau,\qquad\forall i\in\{1,\dots,M\}.\label{eq:1+m sigma lower bound}
		\end{eqnarray}
	\end{lem}
	
	\begin{prop}
		\label{prop:stability}Fix $\tau>0$. There exists a constant $\tau^{\prime}>0$
		such that $z=f(m)$ is stable at the edge $\mathbf{D}^{e}(\tau,\tau^{\prime},N)$
		in the following sense. Suppose $\delta:\mathbf{D}^{e}\to(0,\infty)$
		satisfies $N^{-2}\le\delta(z)\le\log^{-1}N$ for $z\in\mathbf{D}^{e}$
		and that $\delta$ is Lipschitz continuous with Lipschitz constant
		$N^{2}$. Suppose moreover that for each fixed $E$, the function
		$\eta\to\delta(E+\imath\eta)$ is nonincreasing for $\eta>0$. Suppose
		that $u:\mathbf{D}^{e}\to\mathbb{C}$ is the Stieltjes transform of
		a probability measure supported on $[0,C]$. Let $z\in\mathbf{D}^{e}$
		and suppose that
		\[
		|f(u(z))-z|\le\delta(z).
		\]
		If $\operatorname{Im} z<1$, suppose also that
		\begin{equation}
		|u-m|\le\frac{C\delta}{\sqrt{\kappa+\eta}+\sqrt{\delta}},\label{eq:stability}
		\end{equation}
		holds at $z+\imath N^{-5}$. Then (\ref{eq:stability}) holds at $z$.
	\end{prop}

	\section{\label{sec:5}Proof of the local law}
	
	In this section, we prove Theorem \ref{thm:strong local law}.   Theorem \ref{thm:rigidity} follows from Theorem \ref{thm:strong local law} directly by standard arguments, whose details are put in Appendix \ref{Appendix C}. Firstly, we prove
	a weaker result.
	\begin{prop}[Deformed weak local law]
		\label{prop:deformed weak local law}Suppose Conditions \ref{cond:1} and \ref{cond:3} as well as (\ref{2.1}) and (\ref{2.8})
		hold. Then there exists a constant $\tau^{\prime}>0$ depending only
		on $\tau$ such that $\Lambda\prec(N\eta)^{-1/4}+q$ uniformly for $z\in\mathbf{D}^{e}(\tau,\tau^{\prime},N)$ with high probability.
	\end{prop}
	
	For $i\in{\cal I}$, define $P_{i}$ as the operator of expectation
	conditioning on all $(\mathbf{u}_{1},\dots,\mathbf{u}_{N})$ and $(\xi_{1},\dots,\xi_{N})$
	except $\mathbf{u}_{i}$. Denote $Q_{i}=1-P_{i}$. Define
	\[
	Z_{i}:=Q_{i}(\mathbf{x}_{i}^{*}{\cal G}^{(i)}\mathbf{x}_{i})=\mathbf{x}_{i}^{*}{\cal G}^{(i)}\mathbf{x}_{i}-\frac{\xi_{i}^{2}}{M}{\rm Tr}({\cal G}^{(i)}\Sigma).
	\]
	We observe from Lemma \ref{lem:resolvent} that,
	\begin{equation}
	\frac{1}{G_{ii}}=-z-z\mathbf{x}_{i}^{*}{\cal G}^{(i)}\mathbf{x}_{i}=-z-\frac{\xi_{i}^{2}}{M}z{\rm Tr}({\cal G}^{(i)}\Sigma)-zZ_{i}.\label{eq:resolvent}
	\end{equation}
	
	In the following, we denote
	\begin{eqnarray*}
		{\cal U}_{i} & = & \frac{1}{M}\{{\rm Tr}({\cal G}\Sigma)-{\rm Tr}({\cal G}^{(i)}\Sigma)\},\qquad i\in \cal{I},\\
		{\cal V} & = & \frac{1}{M}\{{\rm Tr}\{(-zm_{N}\Sigma-zI)^{-1}\Sigma\}-{\rm Tr}({\cal G}\Sigma)\}.
	\end{eqnarray*}
	Note that from (\ref{eq:resolvent}) and the definitions of ${\cal U}_{i}$
	and ${\cal V}$, we have
	\begin{equation}
	\frac{1}{G_{ii}}=-z+z\xi_{i}^{2}{\cal U}_{i}+z\xi_{i}^{2}{\cal V}-z\frac{\xi_{i}^{2}}{M}{\rm Tr}\{(-zm_{N}\Sigma-zI)^{-1}\Sigma\}-zZ_{i}.\label{eq:1/Gii}
	\end{equation}
	
	Before  proceeding to prove Proposition \ref{prop:deformed weak local law},
	we provide the following useful lemmas and propositions \ref{lem:10.5} to \ref{lem:crude bound on G and G inverse}, whose proofs are in Appendix \ref{Appendix B}. Recall that $\Xi$ is the event $\{\Lambda\le(\log N)^{-1}\}.$
	
	\begin{lem}
		\label{lem:10.5}
		\[
		{\cal G}-(-zm_{N}\Sigma-zI)^{-1}=\sum_{i\in{\cal I}}\frac{(m_{N}\Sigma+I)^{-1}}{z(1+\mathbf{x}_{i}^{*}{\cal G}^{(i)}\mathbf{x}_{i})}(\mathbf{x}_{i}\mathbf{x}_{i}^{*}{\cal G}^{(i)}-\frac{1}{N}\Sigma{\cal G}).
		\]
	\end{lem}
	
	\begin{lem}[Ward identity]
		\label{lem:ward}Let $T\subset{\cal I}$ such that $0\le|T|\le C$.
		Then $\|{\cal G}^{(T)}\|_{F}^{2}=\eta^{-1}\operatorname{Im}{\rm Tr}{\cal G}^{(T)}$.
	\end{lem}
	
	\begin{lem}
		\label{lem:trace_difference}For any $i\in{\cal I}$
		\begin{eqnarray*}
			|{\rm Tr}(G^{(i)}-G)| & \le & \eta^{-1},\\
			|{\rm Tr}({\cal G}^{(i)}-{\cal G})| & \le & |z|^{-1}+\eta^{-1},\\
			|\operatorname{Im}{\rm Tr}({\cal G}^{(i)}-{\cal G})| & \le & \eta|z|^{-2}+\eta^{-1}.
		\end{eqnarray*}
	\end{lem}
	
	\begin{prop}[General properties of $m$]
		\label{prop:Imm(z)}Fix $\tau>0$. Given \eqref{2.1} and \eqref{2.8}, there exists a constant $C>0$ such that
		\begin{equation}
		|m(z)|\asymp1,\quad \operatorname{Im} m(z)\ge C^{-1}\eta,\label{eq:Imgeeta}
		\end{equation}
		for all $z\in\mathbb{C}^{+}$ satisfying $\tau\le|z|\le\tau^{-1}$.
	\end{prop}

	\begin{lem}
		\label{lem:crude bound on G and G inverse}Let $T$ be an index set
		such that $0\le|T|\le C_{1}$ for some constant $C_{1}\ge0$ ( $T$
		may be empty set). Then
		\[
		\{\mathbf{1}(\Xi)+\mathbf{1}(\eta\ge1)\}|G_{ij}^{(T)}|+\mathbf{1}(\Xi)\Big|\frac{1}{G_{ii}^{(T)}}\Big|\le C,
		\]
		for some constant $C>0$ uniformly for $i,j\in{\cal I}$ and $z\in\mathbf{D}$.
	\end{lem}

	Now we proceed to prove the weak local law.  We start with the next lemma which provides a good control for the error when $\eta\ge 1$ or $\Xi$ holds.
	
	\begin{lem}
		\label{lem:bounds of Z and Lambda_o}Suppose Conditions \ref{cond:1},
		\ref{cond:3}, (\ref{2.1}) and (\ref{2.8}) hold. Then
		\begin{equation}
		\{\mathbf{1}(\eta\ge1)+\mathbf{1}(\Xi)\}(|Z_{i}|+\Lambda_{o})\prec\Psi_{\Theta},\label{eq:Zi+Lambdao bound}
		\end{equation}
		\begin{equation}\label{eq:V+U}
		\{\mathbf{1}(\eta\ge1)+\mathbf{1}(\Xi)\}(|\mathcal{V}|+|\mathcal{U}_i|)\prec \Psi_{\Theta},
		\end{equation}
		uniformly for $i\in{\cal I}$ and $z\in\mathbf{D}$.
	\end{lem}
	\begin{proof}
		We firstly show \eqref{eq:Zi+Lambdao bound}. Applying Lemmas \ref{lem:resolvent}, \ref{lem:norm inequality} and (\ref{eq:10.4-2}),
		we obtain that uniformly for $z\in\mathbf{D}$ and $i,j\in{\cal I}$
		with $i\ne j$,
		\begin{equation}
		\mathbf{1}(\Xi)|G_{ij}|\le\mathbf{1}(\Xi)|z||G_{ii}G_{jj}^{(i)}||\mathbf{x}_{i}^{*}{\cal G}^{(ij)}\mathbf{x}_{j}|\prec\mathbf{1}(\Xi)|G_{ii}G_{jj}^{(i)}|\xi_{i}\xi_{j}\frac{1}{M}\|\Sigma\|\|{\cal G}^{(ij)}\|_{F}.\label{eq:10.13}
		\end{equation}
		
		Using Lemma \ref{lem:resolvent}, we obtain that for any $k\in{\cal I}\backslash\{i,j\}$,
		\begin{eqnarray*}
			G_{kk}^{(ij)} & = & G_{kk}^{(i)}-\frac{G_{kj}^{(i)}G_{jk}^{(i)}}{G_{jj}^{(i)}}=G_{kk}-\frac{G_{ki}G_{ik}}{G_{ii}}-\frac{(G_{kj}-\frac{G_{ki}G_{ij}}{G_{ii}})(G_{jk}-\frac{G_{ji}G_{ik}}{G_{ii}})}{G_{jj}^{(i)}}\\
			& = & G_{kk}-\frac{G_{ki}G_{ik}}{G_{ii}}-\frac{G_{kj}G_{jk}-\frac{G_{ki}G_{ij}G_{jk}}{G_{ii}}-\frac{G_{kj}G_{ji}G_{ik}}{G_{ii}}+\frac{G_{ki}G_{ij}G_{ji}G_{ik}}{G_{ii}^{2}}}{G_{jj}^{(i)}}\\
			& = & G_{kk}-\frac{G_{ki}G_{ik}}{G_{ii}}-\bigg(\frac{G_{kj}G_{jk}}{G_{jj}^{(i)}}-\frac{G_{ki}G_{ij}G_{jk}}{G_{jj}^{(i)}G_{ii}}-\frac{G_{kj}G_{ji}G_{ik}}{G_{jj}^{(i)}G_{ii}}+\frac{G_{ki}G_{ij}G_{ji}G_{ik}}{G_{jj}^{(i)}G_{ii}^{2}}\bigg).
		\end{eqnarray*}
		
		Then we have from Lemma \ref{lem:crude bound on G and G inverse}
		that
		\begin{multline}
		\mathbf{1}(\Xi)|G_{kk}^{(ij)}-G_{kk}|\\
		\le\mathbf{1}(\Xi)\bigg(\frac{|G_{ki}G_{ik}|}{|G_{ii}|}+\frac{|G_{kj}G_{jk}|}{|G_{jj}^{(i)}|}+\frac{|G_{ki}G_{ij}G_{jk}|}{|G_{jj}^{(i)}G_{ii}|}+\frac{|G_{kj}G_{ji}G_{ik}|}{|G_{jj}^{(i)}G_{ii}|}+\frac{|G_{ki}G_{ij}G_{ji}G_{ik}|}{|G_{jj}^{(i)}G_{ii}^{2}|}\bigg)\\
		\le\mathbf{1}(\Xi)C(\Lambda_{o}^{2}+\Lambda_{o}^{3}+\Lambda_{o}^{4})\le\mathbf{1}(\Xi)C\Lambda_{o}^{2},\label{eq:3.19}
		\end{multline}
		where the last inequality holds because $\Lambda_{o}^{3}+\Lambda_{o}^{4}\le\Lambda_{o}^{2}$ for large $N$ given $\Xi$. Then it follows from (\ref{eq:3.19}) and Lemma \ref{lem:crude bound on G and G inverse}
		that
		\begin{eqnarray}
		\mathbf{1}(\Xi)|\operatorname{Im}{\rm Tr}G^{(ij)}-\operatorname{Im}{\rm Tr}G|\nonumber
		& = & \mathbf{1}(\Xi)\Big|\sum_{k\in{\cal I}\backslash\{i,j\}}\operatorname{Im} G_{kk}^{(ij)}-\sum_{k\in{\cal I}}\operatorname{Im} G_{kk}\Big|\nonumber \\
		& \le & \mathbf{1}(\Xi)\Big|\sum_{k\in{\cal I}\backslash\{i,j\}}(G_{kk}^{(ij)}-G_{kk})\Big|+\mathbf{1}(\Xi)\Big|\operatorname{Im} G_{ii}+\operatorname{Im} G_{jj}\Big|\nonumber \\
		& \le & \mathbf{1}(\Xi)CN\Lambda_{o}^{2}+\mathbf{1}(\Xi)2\operatorname{Im} m(z)+\frac{2}{\log N}.\label{eq:10.18}
		\end{eqnarray}
		We note that
		\begin{equation}
		{\rm Tr}{\cal G}^{(ij)}=\frac{(N-2-M)}{z}+{\rm Tr}G^{(ij)}.\label{eq:10.15}
		\end{equation}
		Applying Lemma \ref{lem:ward} and (\ref{eq:10.15}), we have
		\begin{equation}
		\mathbf{1}(\Xi)\frac{\|{\cal G}^{(ij)}\|_{F}^{2}}{M^{2}}=\mathbf{1}(\Xi)\frac{\operatorname{Im}{\rm Tr}{\cal G}^{(ij)}}{M^{2}\eta}=\mathbf{1}(\Xi)\Big\{\frac{\operatorname{Im}{\rm Tr}G^{(ij)}}{M^{2}\eta}-\frac{(N-2-M)}{M^{2}|z|^{2}}\Big\}.\label{eq:10.14}
		\end{equation}
		It then follows from (\ref{eq:Imgeeta}), (\ref{eq:10.18}), (\ref{eq:10.14})
		and $MN^{-1}\asymp1$ that
		\begin{equation}
		\mathbf{1}(\Xi)\frac{1}{M^{2}}\|{\cal G}^{(ij)}\|_{F}^{2}\le\mathbf{1}(\Xi)C\frac{\operatorname{Im} m(z)+\Theta+\Lambda_{o}^{2}}{N\eta}.\label{eq:10.16}
		\end{equation}
		Using Lemma \ref{lem:crude bound on G and G inverse}, (\ref{eq:10.13}),
		(\ref{eq:10.16}) and the fact $\xi_{i}\prec1$ uniformly for all
		$i\in{\cal I}$, we have
		\[
		\mathbf{1}(\Xi)|G_{ij}|\prec\mathbf{1}(\Xi)\bigg(\frac{\operatorname{Im} m+\Theta+\Lambda_{o}^{2}}{N\eta}\bigg)^{1/2}.
		\]
		Therefore, by the definition of stochastic domination,
		\[
		\mathbf{1}(\Xi)|\Lambda_{o}|\prec\mathbf{1}(\Xi)\sqrt{\frac{\operatorname{Im} m+\Theta}{N\eta}}+\mathbf{1}(\Xi)\frac{\Lambda_{o}}{(N\eta)^{1/2}}\quad \Rightarrow\quad  \mathbf{1}(\Xi)\Lambda_{o}\prec1(\Xi)\Psi_{\Theta}.
		\]

		Now we evaluate the bound for $Z_{i}$. Similarly to (\ref{eq:10.16}),
		we can easily derive that uniformly for any $i\in{\cal I}$,
		\begin{equation}
		\mathbf{1}(\Xi)\frac{1}{M^{2}}\|{\cal G}^{(i)}\|_{F}^{2}\prec\mathbf{1}(\Xi)\frac{\operatorname{Im} m(z)+\Theta+\Lambda_{o}^{2}}{N\eta}.\label{eq:calGFbound}
		\end{equation}
		
		It follows from
		Lemmas \ref{lem:large deviation}, \ref{lem:ward}, (\ref{eq:calGFbound}) and $\xi_{i}^2\prec1$
		by bounded support assumption for $i\in{\cal I}$ that
		\begin{eqnarray*}
			\mathbf{1}(\Xi)Z_{i} & = & \mathbf{1}(\Xi)\{z\mathbf{x}_{i}^{*}{\cal G}^{(i)}\mathbf{x}_{i}-z\frac{\xi_{i}^{2}}{M}{\rm Tr}({\cal G}^{(i)}\Sigma)\}
			\prec  \mathbf{1}(\Xi)|z|\xi_{i}^{2}\frac{1}{M}\|{\cal G}^{(i)}\Sigma\|_{F}\\
			& \le & \mathbf{1}(\Xi)|z|\xi_{i}^{2}\frac{\sigma_{1}^{2}}{M}\|{\cal G}^{(i)}\|_{F}
			\prec  \mathbf{1}(\Xi)\sqrt{\frac{\operatorname{Im} m+\Theta+\Lambda_{o}^{2}}{N\eta}}.
		\end{eqnarray*}
		Using the bound $\mathbf{1}(\Xi)\Lambda_{o}\prec\mathbf{1}(\Xi)\Psi_{\Theta}$,
		we obtain that
		\[
		\mathbf{1}(\Xi)Z_{i}\prec\mathbf{1}(\Xi)\bigg(\sqrt{\frac{\operatorname{Im} m+\Theta}{N\eta}}+\frac{\sqrt{\operatorname{Im} m+\Theta}}{N\eta}\bigg)\prec\mathbf{1}(\Xi)\sqrt{\frac{\operatorname{Im} m+\Theta}{N\eta}}.
		\]
		
		Now we show the result when $\eta\ge1$. Let $i,j\in{\cal I}$ such
		that $i\ne j$. It follows from Lemma \ref{lem:crude bound on G and G inverse},
		(\ref{eq:10.14}) and $\xi_{i}\prec1$ for $i\in{\cal I}$ that
		\begin{equation*}
		\begin{split}
		\mathbf{1}(\eta\ge1)|G_{ij}| & \le \mathbf{1}(\eta\ge1)|G_{ii}G_{jj}^{(i)}|\Big|\mathbf{x}_{i}^{*}{\cal G}^{(ij)}\mathbf{x}_{j}\Big|\\
		& \prec\mathbf{1}(\eta\ge1)\frac{1}{M}\|\Sigma\|\|{\cal G}^{(ij)}\|_{F}\\
		& \le \mathbf{1}(\eta\ge1)\|\Sigma\|\bigg(\frac{\operatorname{Im}{\rm Tr}{\cal G}^{(ij)}}{M^{2}\eta}\bigg)^{1/2}\\
		& =\mathbf{1}(\eta\ge1)\|\Sigma\|\bigg(\frac{\operatorname{Im}{\rm Tr}G^{(ij)}}{M^{2}\eta}-\frac{N-2-M}{M^{2}|z|^{2}}\bigg)^{1/2}.
		\end{split}
		\end{equation*}
		Let $T$ be a subset of ${\cal I}$ such that $|T|\le C$ for all
		large $N$. From Lemma \ref{lem:trace_difference}, we know that
		\begin{equation}
		|{\rm Tr}G^{(T)}-{\rm Tr}G|\le C\eta^{-1}.\label{eq:10.22-3}
		\end{equation}
		It then follows from Proposition \ref{prop:Imm(z)} and (\ref{eq:10.22-3})
		that
		\begin{multline}
		\mathbf{1}(\eta\ge1)\frac{\operatorname{Im}{\rm Tr}{\cal G}^{(T)}}{M^{2}\eta}=\mathbf{1}(\eta\ge1)\bigg(\frac{\operatorname{Im}{\rm Tr}G^{(T)}}{M^{2}\eta}-\frac{N-|T|-M}{M^{2}|z|^{2}}\bigg)\\
		\le\mathbf{1}(\eta\ge1)\bigg(\frac{\operatorname{Im}{\rm Tr}G}{M^{2}\eta}+\frac{C\eta^{-1}}{M^{2}\eta}-\frac{N-|T|-M}{M^{2}|z|^{2}}\bigg)\prec\frac{\operatorname{Im} m+\Theta}{N\eta}=\Psi_{\Theta}^{2}.\label{eq:TrG/M2eta bound when eta>1}
		\end{multline}
		Consequently
		\[
		\mathbf{1}(\eta\ge1)\Lambda_{o}=\mathbf{1}(\eta\ge1)\max_{i,j\in{\cal I},i\ne j}|G_{ij}|\prec\mathbf{1}(\eta\ge1)\sqrt{\frac{\operatorname{Im} m+\Theta}{N\eta}}.
		\]
		For $Z_{i}$, using Lemma \ref{lem:large deviation}, (\ref{eq:TrG/M2eta bound when eta>1})
		and $\xi_{i}^2\prec1$ for $i\in{\cal I}$, we have
		\begin{eqnarray*}
			\mathbf{1}(\eta\ge1)Z_{i} & = & \mathbf{1}(\eta\ge1)\{z\mathbf{x}_{i}^{*}{\cal G}^{(i)}\mathbf{x}_{i}-z\frac{\xi_{i}^{2}}{M}{\rm Tr}({\cal G}^{(i)}\Sigma)\}
			\prec  \mathbf{1}(\eta\ge1)|z|\xi_{i}^{2}\frac{1}{M}\|{\cal G}^{(i)}\Sigma\|_{F}\\
			& \le & \mathbf{1}(\eta\ge1)|z|\xi_{i}^{2}\|\Sigma\|\frac{1}{M}\|{\cal G}^{(i)}\|_{F}
			\prec  \mathbf{1}(\eta\ge1)\sqrt{\frac{\operatorname{Im} m+\Theta}{N\eta}}.
		\end{eqnarray*}
		Hence (\ref{eq:Zi+Lambdao bound}) follows. Next,  we will show \eqref{eq:V+U}. Under $\Xi$, applying Lemmas  \ref{lem:resolvent}, \ref{lem:norm inequality},
		\ref{lem:large deviation}, \ref{lem:ward}, (\ref{eq:calGFbound}) and $\xi_i^2\prec 1$, we get, for
		any $i\in{\cal I}$,
		\[
		\begin{split}
		|{\cal U}_{i}|&=\frac{1}{M}\Big|{\rm Tr}[({\cal G}^{(i)}-{\cal G})\Sigma]\Big|=\frac{1}{M}\Big|\frac{\mathbf{x}_{i}^{*}{\cal G}^{(i)}\Sigma{\cal G}^{(i)}\mathbf{x}_{i}}{1+\mathbf{x}_{i}^{*}{\cal G}^{(i)}\mathbf{x}_{i}}\Big|\\
		&=\frac{1}{M}\Big|zG_{ii}\mathbf{x}_{i}^{*}{\cal G}^{(i)}\Sigma{\cal G}^{(i)}\mathbf{x}_{i}\Big|\prec\frac{1}{M}|zG_{ii}|\Big(\Big|\frac{1}{M}{\rm Tr}({\cal G}^{(i)}\Sigma{\cal G}^{(i)}\Sigma)\Big|+\frac{1}{M}\|{\cal G}^{(i)}\Sigma{\cal G}^{(i)}\Sigma\|_{F}\Big)\\
		&\le\frac{2}{M^{2}}|zG_{ii}|\|{\cal G}^{(i)}\Sigma\|_{F}^{2}\le\frac{2}{M^{2}}|zG_{ii}|\|{\cal G}^{(i)}\|_{F}^{2}\|\Sigma\|^{2}\\
		&\prec\Psi^2_{\Theta}.
		\end{split}
		\]
		Similarly, under $\Xi$,
		\begin{equation}\label{eq:V}
		\begin{split}
		&\mathbf{1}{(\Xi)}|\mathcal{V}|= \mathbf{1}{(\Xi)}\Big|\frac{1}{M}\left({\rm Tr}(-zm_N\Sigma-zI)^{-1}\Sigma-{\rm Tr}\mathcal{G}\Sigma\right)\Big|\\
		=&\mathbf{1}{(\Xi)}\frac{1}{M}\Big|{\rm Tr}\left(\sum_{i\in\mathcal{I}}\frac{(m_N\Sigma+I)^{-1}}{z(1+\mathbf{x}_i^{*}\mathcal{G}^{(i)}\mathbf{x}_i)}(\mathbf{x}_i\mathbf{x}_i^{*}\mathcal{G}^{(i)}\Sigma-\frac{1}{N}\Sigma\mathcal{G}\Sigma+\frac{1}{N}\Sigma\mathcal{G}^{(i)}\Sigma-\frac{1}{N}\Sigma\mathcal{G}^{(i)}\Sigma)\right) \Big|\\
		\prec&\mathbf{1}{(\Xi)}\frac{1}{M}\sum_{i\in\mathcal{I}}\frac{\xi^2_i}{M}\|\Sigma\|\|(m_N\Sigma+I)^{-1}\Sigma\mathcal{G}^{(i)}\|_{F}+\frac{q}{\sqrt{N}}+\Psi_{\Theta}^2\\
		\leqslant& \mathbf{1}{(\Xi)}\frac{1}{M}\sum_{i\in\mathcal{I}}\frac{1}{M}\|\Sigma\|\|(m_N\Sigma+I)^{-1}\|\|\Sigma\mathcal{G}^{(i)}\|_{F}+\frac{q}{\sqrt{N}}+\Psi_{\Theta}^2.
		\end{split}
		\end{equation}
		Since by assumption \eqref{2.8}
		\[
		|m\sigma_i+1|\geqslant\tau,
		\]
		we have
		\[
		\mathbf{1}{(\Xi)}|1+m_N\sigma|\geqslant\mathbf{1}{(\Xi)}(|1+m\sigma_i|-|m-m_N|\sigma_i)\geqslant\tau^{'}>0.
		\]
		Combining  \eqref{eq:V} we have
		\begin{equation}
		\begin{split}
		\mathbf{1}{(\Xi)}|\mathcal{V}|&\prec\frac{1}{M}\sum_{i\in\mathcal{I}}\frac{1}{M}\|\Sigma\mathcal{G}^{(i)}\|_{F}+\frac{q}{\sqrt{N}}+\Psi^2_{\Theta} \prec\Psi_{\Theta}.
		\end{split}
		\end{equation}
		For $\eta\geqslant1$, the procedure is similar and we omit the details. Then the lemma holds.
	\end{proof}
	\begin{rem}\label{}
		In the following  proof,we will use two relations several times,
		\begin{equation}
		\{\mathbf{1}(\eta\ge1)+\mathbf{1}(\Xi)\}\frac{1}{M}\|{\cal G}^{(i)}\|_{F}\prec\Psi_{\Theta}, \quad
		\{\mathbf{1}(\eta\ge1)+\mathbf{1}(\Xi)\}\frac{1}{M}\|{\cal G}\|_{F}\prec\Psi_{\Theta},\label{eq:lem3.5 eq:1}
		\end{equation}
		so we summarize them here.
	\end{rem}
	
	With the above results, we can further prove the next lemma.
	\begin{lem}
		\label{lem:Gmumu-Gnunu}Under the assumptions in Lemma \ref{lem:bounds of Z and Lambda_o}, one has
		\begin{equation}
		\{\mathbf{1}(\eta\ge1)+\mathbf{1}(\Xi)\}|G_{ii}-G_{jj}|\prec\Psi_{\Theta}+q\label{eq:Gmumu-Gnunu},
		\end{equation}
		uniformly for $i,j\in{\cal I}$ and $z\in\mathbf{D}$.
	\end{lem}
	\begin{proof}
		We observe from (\ref{eq:resolvent}) that
		\begin{equation}
		\begin{split}
		|G_{ii}-G_{jj}|&=\Big| G_{ii}G_{jj}\Big(\frac{1}{G_{jj}}-\frac{1}{G_{ii}}\Big)\Big| \\
		&\leqslant|G_{ii}G_{jj}||Z_{i}-Z_{j}|+\Big|G_{ii}G_{jj}z\bigg\{\frac{\xi_{i}^{2}}{M}{\rm Tr}({\cal G}^{(i)}\Sigma)-\frac{\xi_{j}^{2}}{M}{\rm Tr}({\cal G}^{(j)}\Sigma)\bigg\}\Big| \\
		&\leqslant|G_{ii}G_{jj}||Z_{i}-Z_{j}|+|G_{ii}G_{jj}||z|\Big( \frac{\xi_{i}^{2}}{M}|{\rm Tr}({\cal G}^{(i)}\Sigma)|+\frac{\xi_{j}^{2}}{M}|{\rm Tr}({\cal G}^{(j)}\Sigma)|\Big)\\
		&\prec|Z_{i}-Z_{j}|+\frac{\xi_i^2}{M}{\rm Tr}(|\mathcal{G}^{(i)}\Sigma-\mathcal{G}^{(j)}\Sigma|)+\frac{\xi_i^2-\xi_j^2}{M}{\rm Tr}(|\mathcal{G}^{(j)}\Sigma|)\\
		&\prec\Psi_{\Theta}+q+\Psi_{\Theta}^2,
		\label{eq:3.28}
		\end{split}
		\end{equation}
		where we used Condition \ref{cond:3} and the fact that under $\Xi$, $
		|\mathcal{G}^{(i)}_{kk}|\asymp\sigma_{k}.$
	\end{proof}
	
	Now we can complete the proof of Proposition \ref{prop:deformed weak local law}.
	\begin{proof}[Proof of Proposition \ref{prop:deformed weak local law}]
		We observe from (\ref{eq:Gmumu-Gnunu}) that
		\begin{eqnarray}
		&  & \{\mathbf{1}(\Xi)+\mathbf{1}(\eta\ge1)\}\Big\{\frac{1}{N}\sum_{i\in{\cal I}}\frac{1}{G_{ii}}-\frac{1}{m_{N}}\Big\}\nonumber \\
		& = & \{\mathbf{1}(\Xi)+\mathbf{1}(\eta\ge1)\}\frac{1}{N}\sum_{i\in{\cal I}}\Big(-\frac{G_{ii}-m_{N}}{m_{N}^{2}}+\frac{(G_{ii}-m_{N})^{2}}{G_{ii}m_{N}^2}\Big)\nonumber \\
		& = & \{\mathbf{1}(\Xi)+\mathbf{1}(\eta\ge1)\}\frac{1}{N}\sum_{i\in{\cal I}}\frac{(G_{ii}-m_{N})^{2}}{G_{ii}m_{N}^2}\nonumber \\
		& \prec & \Psi_{\Theta}^{2}+q^2.\label{eq:1/G-1/m-N}
		\end{eqnarray}
		It then follows from (\ref{eq:1/Gii}), Condition \ref{cond:3} and (\ref{eq:1/G-1/m-N})
		that
		\begin{multline}
		\{\mathbf{1}(\Xi)+\mathbf{1}(\eta\ge1)\}\frac{1}{m_{N}}=\{\mathbf{1}(\Xi)+\mathbf{1}(\eta\ge1)\}\frac{1}{N}\sum_{i\in{\cal I}}\frac{1}{G_{ii}}+O_{\prec}(\Psi_{\Theta}^{2})+O_{\prec}(q^2)\\
		\begin{aligned}= & \{\mathbf{1}(\Xi)+\mathbf{1}(\eta\ge1)\}\bigg[z\Big(-1+\frac{1}{N}\sum_{i\in{\cal I}}\xi_{i}^{2}{\cal U}_{i}+\frac{1}{N}{\cal V}\sum_{i\in{\cal I}}\xi_{i}^{2}\Big)\\
		& +\frac{1}{M}\sum_{i\in{\cal I}}\xi_{i}^{2}\frac{1}{N}{\rm Tr}\{(m_{N}\Sigma+I)^{-1}\Sigma\}-\frac{z}{N}\sum_{i\in{\cal I}}Z_{i}\bigg]+O_{\prec}(\Psi_{\Theta}^{2})+O_{\prec}(q^2)\\
		=& \{\mathbf{1}(\Xi)+\mathbf{1}(\eta\ge1)\}\Big(-z+\frac{1}{N}{\rm Tr}\{(m_{N}\Sigma+I)^{-1}\Sigma\}\Big)+O_{\prec}(\Psi_{\Theta})+O_{\prec}(q^2).
		\end{aligned}
		\label{eq:3.34}
		\end{multline}
		Since
		\[
		{\rm Tr}\{(m_{N}\Sigma+I)^{-1}\Sigma\}=\sum_{i\in\mathcal{I}}\frac{\sigma_i}{m_N\sigma_i+1},
		\]
		it follows from the definition of $f(x)$ in (\ref{def:f}) that
		\begin{equation}
		\{\mathbf{1}(\Xi)+\mathbf{1}(\eta\ge1)\}\{f(m_{N})-z\}\prec\Psi_{\Theta}+q^2.\label{eq:1(XI)(f(m)-z)}
		\end{equation}
		Applying Proposition \ref{prop:stability}, for any $\varepsilon>0$
		we have
		\begin{equation}
		\mathbf{1}(\eta\ge1)|m_{N}-m|\prec\frac{\Psi_{\Theta}+q^2}{\sqrt{\kappa+\eta}+\sqrt{N^{\varepsilon}(\Psi_{\Theta}+q^2})}\prec \sqrt{\Psi_{\Theta}+q^2}.\label{eq:3.36}
		\end{equation}
		Therefore, it follows from (\ref{eq:Gmumu-Gnunu}), (\ref{eq:3.36})
		and Lemma \ref{lem:bounds of Z and Lambda_o} that
		\begin{equation}
		\mathbf{1}(\eta\ge1)\Lambda(z)\le\mathbf{1}(\eta\ge1)\{\max_{i}|G_{ii}-m_{N}|+|m_{N}-m|+\Lambda_{o}\}\prec N^{-1/2}+q.\label{eq:3.38-1}
		\end{equation}
		The rest proof of Proposition \ref{prop:deformed weak local law}
		follows from a standard bootstrapping step which we summarize  into Appendix \ref{Appendix B}-\ref{Suppl:prop deformed wll} and omit further details here.
	\end{proof}

	Now we can prove Theorem \ref{thm:strong local law}. Note that for
	$i\in{\cal I}$,
	\begin{equation}
	Q_{i}\frac{1}{G_{ii}}=Q_{i}\{-z-z\frac{\xi_{i}^{2}}{M}{\rm Tr}({\cal G}^{(i)}\Sigma)-zZ_{i}\}=-zZ_{i},\label{eq:observation Q1/Gii}
	\end{equation}
	and we write
	\begin{equation}\label{eq:FA of V}
	\begin{split}
	\mathcal{V}&=\frac{1}{M}{\rm Tr}\Big(\sum_{i\in\mathcal{I}}\frac{(m_N\Sigma+I)^{-1}}{z(1+\mathbf{x}_{i}^{*}\mathcal{G}^{(i)}\mathbf{x}_{i})}(\mathbf{x}_{i}\mathbf{x}_{i}^{*}\mathcal{G}^{(i)}\Sigma-\frac{1}{N}\Sigma\mathcal{G}^{(i)}\Sigma+\frac{1}{N}\Sigma\mathcal{G}^{(i)}\Sigma-\frac{1}{N}\Sigma\mathcal{G}\Sigma)\Big)\\
	&=\frac{1}{M}\sum_{i\in\mathcal{I}}G_{ii}{\rm Tr}(\mathcal{V}_i)+\frac{1}{M}\sum_{i\in\mathcal{I}}G_{ii}\frac{1}{N}{\rm Tr}((m_N\Sigma+I)^{-1}\Sigma(\mathcal{G}^{(i)}-\mathcal{G})\Sigma),
	\end{split}
	\end{equation}
	where
	\begin{equation}\label{def:Vi}
	\mathcal{V}_i:=(m_N\Sigma+I)^{-1}(\mathbf{x}_{i}\mathbf{x}_{i}^{*}\mathcal{G}^{(i)}\Sigma-\frac{1}{N}\Sigma\mathcal{G}^{(i)}\Sigma).
	\end{equation}
	
	For the second term in (\ref{eq:FA of V}),
	\begin{equation}\label{eq:second term of V}
	\begin{split}
	&\Big|\frac{1}{M}\sum_{i\in\mathcal{I}}G_{ii}\frac{1}{N}{\rm Tr}((m_N\Sigma+I)^{-1}\Sigma(\mathcal{G}^{(i)}-\mathcal{G})\Sigma)\Big|\\
	=&\Big|\frac{1}{M}\sum_{i\in\mathcal{I}}G_{ii}\frac{1}{N}{\rm Tr}((m_N\Sigma+I)^{-1}\Sigma\frac{\mathcal{G}^{(i)}\mathbf{x}_i\mathbf{x}_i^{*}\mathcal{G}^{(i)}}{1+\mathbf{x}_i^{*}\mathcal{G}^{(i)}\mathbf{x}_i}\Sigma)\Big| \\
	\leqslant& \frac{1}{M}\sum_{i\in\mathcal{I}}|G_{ii}|^2\frac{1}{N}|\mathbf{x}_i^{*}\mathcal{G}^{(i)}(m_N\Sigma+I)^{-1}\mathcal{G}^{(i)}\mathbf{x}_i|,
	\end{split}
	\end{equation}
	which can be bounded by $\Psi_{\Theta}^2$ by Lemma \ref{lem:norm inequality}, Lemma \ref{lem:large deviation}, Lemma \ref{lem:crude bound on G and G inverse} and (\ref{eq:lem3.5 eq:1}).
	
	Furthermore, using the same methods one can easily verify that
	\begin{equation}\label{eq:Vi decomp}
	\begin{split}
	\frac{1}{M}\sum_{i\in\mathcal{I}}G_{ii}{\rm Tr}(\mathcal{V}_i)&= \frac{1}{M}\sum_{i\in\mathcal{I}}G_{ii}{\rm Tr}(\mathcal{V}_i^{(i)})+\frac{1}{M}\sum_{i\in\mathcal{I}}G_{ii}{\rm Tr}(\mathcal{V}_i-\mathcal{V}_i^{(i)})\\
	&\prec \frac{1}{M}\sum_{i\in\mathcal{I}}G_{ii}{\rm Tr}(\mathcal{V}_i^{(i)})+\Psi_{\Theta}^2,
	\end{split}
	\end{equation}
	where $\mathcal{V}_i^{(i)}:=(m_N^{(i)}\Sigma+I)^{-1}(\mathbf{x}_{i}\mathbf{x}_{i}^{*}\mathcal{G}^{(i)}\Sigma-\frac{1}{N}\Sigma\mathcal{G}^{(i)}\Sigma).$
	
	From Proposition \ref{prop:deformed weak local law}, we know that $\Xi$ is true with high probability, i.e. $1\prec\mathbf{1}(\Xi)$. So from
	now on, we can drop the factor $\mathbf{1}(\Xi)$ in all $\Xi$ dependent
	results without affecting their validity. To improve the deformed weak local
	law to the strong local law, a key input is Proposition \ref{prop:fluctuation averaging}
	below whose proof we postpone to Appendix  \ref{Appendix B}-\ref{Suppl:FA}.
	
	\begin{prop}[Fluctuation averaging]
		\label{prop:fluctuation averaging}Let $\nu\in[1/4,1]$ and $\tau^{\prime}$
		be defined in Proposition \ref{prop:deformed weak local law}. Denote $\Phi_{\nu}=\sqrt{\frac{\operatorname{Im} m+(N\eta)^{-\nu}+q}{N\eta}}.$
		Suppose moreover that $\Theta\prec(N\eta)^{-\nu}+q$ uniformly for $z\in\mathbf{D}^{e}(\tau,\tau^{\prime},N)$.
		Then we have
		\begin{equation}
		\frac{1}{N}\sum_{i\in{\cal I}}Q_{i}\frac{1}{G_{ii}}\prec\Phi_{\nu}^{2},\label{eq:fluctuation of Z}
		\end{equation}
		and
		\begin{equation}
		\frac{1}{N}\sum_{i\in{\cal I}}Q_{i}\mathscr{V}_{i}\prec \Phi_{\nu}^{2},\label{eq:fluctuation of V}
		\end{equation}
		uniformly for $z\in\mathbf{D}^{e}(\tau,\tau^{\prime},N)$, where $\mathscr{V}_{i}$ is defined as
		\begin{equation}
		\mathscr{V}_i:=\mathbf{x}_i^{*}\mathcal{G}^{(i)}\Sigma(m_N^{(i)}\Sigma+I)^{-1}\mathbf{x}_i.
		\end{equation}
	\end{prop}
	
	\begin{proof}[Proof of Theorem \ref{thm:strong local law}]
		Let $\varepsilon>0$ be an arbitrary small number. Suppose $\Theta\le N^{\varepsilon}(q^{1/2}+(N\eta)^{-\nu})$
		holds with high probability for some $\nu\in[1/4,1]$ uniformly for
		$z\in\mathbf{D}^{e}$. The idea is to update $\nu$ by applying Proposition
		\ref{prop:fluctuation averaging} iteratively.
		
		Let $\Phi_{\nu}$ be defined in Proposition \ref{prop:fluctuation averaging}. Given
		that $\Theta\le N^{\varepsilon}(q^{1/2}+(N\eta)^{-\nu})$ holds with high probability,
		it follows from  (\ref{eq:observation Q1/Gii}), (\ref{eq:second term of V}),
		Proposition \ref{prop:fluctuation averaging} and (\ref{eq:3.34})
		that
		\[
		|f(m_{N})-z|\le N^{\varepsilon}\{\Phi_{\nu}^2+q^2\}\leqslant N^{\varepsilon}\{q^2+\frac{1}{(N\eta)^{\nu+1}}+\frac{\operatorname{Im} m}{N\eta}\},
		\]
		holds with high probability uniformly for $z\in\mathbf{D}^{e}$.
		
		Then we observe from Lemma \ref{lem:basic property of m} and Proposition \ref{prop:stability} that
		\begin{equation}
		\begin{split}
		\Theta&\leqslant N^{\epsilon}\frac{\{q^2+\frac{1}{(N\eta)^{\nu+1}}+\frac{\operatorname{Im} m}{N\eta}\}}{\sqrt{\kappa+\eta}+\sqrt{\{q^2+\frac{1}{(N\eta)^{\nu+1}}+\frac{\operatorname{Im} m}{N\eta}\}}}\\
		&\leqslant CN^{\epsilon}\Big(\frac{\operatorname{Im} m}{N\eta\sqrt{\kappa+\eta}}+\sqrt{q^2+\frac{1}{(N\eta)^{\nu+1}}+\frac{\operatorname{Im} m}{N\eta}}\Big)\\
		&\leqslant CN^{\epsilon}\Big(\sqrt{\frac{\operatorname{Im} m}{N\eta}}+q+\frac{1}{(N\eta)^{(\nu+1)/2}}\Big)
		\end{split}
		\end{equation}
		holds with high probability uniformly for $z\in\mathbf{D}^{e}$. Then using Lemma \ref{lem:bounds of Z and Lambda_o} and Lemma \ref{lem:Gmumu-Gnunu}, it is easy to check
		\[
		\begin{split}
		\Lambda&\leqslant CN^{\epsilon}(\Psi_{\Theta}+q)+\Theta\leqslant CN^{\epsilon}\Big(\sqrt{\frac{\operatorname{Im} m}{N\eta}}+q+\frac{1}{(N\eta)^{(\nu+1)/2}}\Big).
		\end{split}
		\]
		
		One can see that after the self-improving arguments, the error bound of $\Lambda$ improves from $1/(N\eta)^{\nu}$ to $1/(N\eta)^{(\nu+1)/2}$. Hence implementing the auguments a finite number (depending
		only on $\varepsilon$) of times, we obtain that
		\begin{equation}\label{eq:lambda_result}
		\Lambda\leqslant CN^{\epsilon}(q+\frac{1}{N\eta}+\sqrt{\frac{\operatorname{Im} m}{N\eta}})
		\end{equation}
		holds with high probability uniformly for $z\in\mathbf{D}^{e}$. Applying \eqref{eq:lambda_result} in Lemma \ref{lem:basic property of m}, Proposition \ref{prop:fluctuation averaging} and Proposition \ref{prop:stability}, we conclude Theorem \ref{thm:strong local law}.
	\end{proof}
	
	\section{\label{sec:6}Proof of the edge universality with small support}
	Once the following Green function comparison Theorem \ref{thm:GreenFuncitonComparision} holds, Theorem \ref{thm:Edge universality with small support} will follow from  a standard procedure. We only prove Theorem \ref{thm:GreenFuncitonComparision} in this section while the complete proof of Theorem \ref{thm:Edge universality with small support} is put in Appendix \ref{Appendix C}.
	
	\begin{thm}[Green function comparison on the edge]
		\label{thm:GreenFuncitonComparision}Let $X^V$ and $X^W$ be
		defined  in Theorem \ref{thm:Edge universality with small support}. Let $F:\mathbb{R}\to\mathbb{R}$
		be a function whose derivatives satisfy
		\begin{equation}
		\sup_{x\in\mathbb{R}}|F^{(k)}(x)|(1+|x|)^{-C_{1}}\le C_{1},\qquad k=1,2,3,4,\label{eq:Fderivative}
		\end{equation}
		with some constants $C_{1}>0$. Then there exist $\varepsilon_{0}>0$,
		$N_{0}\in\mathbb{Z}_{+}$ depending on $C_{1}$ such that for any
		$\varepsilon<\varepsilon_{0}$ and $N\ge N_{0}$ and for any real
		numbers $E,E_{1}$ and $E_{2}$ satisfying
		\[
		|E-\lambda_{+}|,|E_{1}-\lambda_{+}|,|E_{2}-\lambda_{+}|\le N^{-2/3+\varepsilon}
		\]
		and $\eta=N^{-2/3-\varepsilon}$, we have
		\begin{equation}
		|\mathbb{E}F(N\eta\operatorname{Im} m_{N}^V(z))-\mathbb{E}F(N\eta\operatorname{Im} m_{N}^W(z))|\le CN^{-1/6+C_{\epsilon}},\qquad z=E+\imath\eta,\label{eq:E F(N eta Im m_N) bound}
		\end{equation}
		and
		\begin{equation}
		\Big|\mathbb{E}F\Big(\int_{E_{1}}^{E_{2}}N\operatorname{Im} m_{N}^V(y+\imath\eta){\rm d}y\Big)-\mathbb{E}F\Big(\int_{E_{1}}^{E_{2}}N\operatorname{Im} m_{N}^W(y+\imath\eta){\rm d}y\Big)\Big|\le CN^{-1/6+C_{\epsilon}},\label{eq:E F(integral) bound}
		\end{equation}
		where $m_{N}^V(z)=N^{-1}{\rm Tr}((X^{V})^{*}X^V-zI)^{-1}$, $C_{\epsilon}$ is a constant which tends to $0$ as $\epsilon \to 0$.
	\end{thm}
	
	\begin{proof}
		Let $\gamma\in\{1,\dots,N+1\}$ and set $X_{\gamma}$ to be the matrix
		whose first $\gamma-1$ columns are the same as those of $X^W$
		and the remaining $N-\gamma+1$ columns are the same as those of $X^V$.
		Then we note that since $X_{\gamma}$ and $X_{\gamma+1}$ only differ
		in the $\gamma$-th column,
		\[
		X_{\gamma}^{(\gamma)}=X_{\gamma+1}^{(\gamma)}.
		\]
		
		We define $m_{N,\gamma}(z)$ and $m_{N,\gamma+1}(z)$ to be the analogs of $m_{N}(z)$ with the matrix $X$ replaced by $X_{\gamma}$ and $X_{\gamma+1}$ respectively. Similarly, for $i\in{\cal I}$, define $m_{N,\gamma}^{(i)}(z)$ and $m_{N,\gamma+1}^{(i)}(z)$
		to be the analogs of $m_{N}^{(i)}(z)$ with the matrix $X^{(i)}$
		replaced by $X_{\gamma}^{(i)}$ and $X_{\gamma+1}^{(i)}$ respectively. Then we have
		\begin{eqnarray*}
			&  & \mathbb{E}^VF(N\eta\operatorname{Im} m^V_{N}(z))-\mathbb{E}^WF(N\eta\operatorname{Im} m^W_{N}(z))\\
			&=&  \sum_{\gamma=1}^{N}\Big\{\mathbb{E}F(N\eta\operatorname{Im} m_{N,\gamma}(z))-\mathbb{E}F(N\eta\operatorname{Im} m_{N,\gamma+1}(z))\Big\}.\\
		\end{eqnarray*}
		So (\ref{eq:E F(N eta Im m_N) bound}) follows from Lemma \ref{lem:GreenCompLem1} below.
		(\ref{eq:E F(integral) bound}) follows from an analogous argument.
		Hence we omit its proof.
	\end{proof}
	\begin{lem}
		\label{lem:GreenCompLem1}Let $F$ be a function satisfying (\ref{eq:Fderivative})
		and $z=E+\imath\eta$. If $|E-\lambda_{+}|\le N^{-2/3+\varepsilon}$
		and $N^{-2/3-\varepsilon}\le\eta\le N^{-2/3}$ for some $\varepsilon>0$,
		there exists some positive constant $C$ independent of $\varepsilon$
		such that
		\begin{equation}
		|\mathbb{E}F(N\eta\operatorname{Im} m_{N,\gamma}(z))-\mathbb{E}F(N\eta\operatorname{Im} m_{N,\gamma+1}(z))|\prec N^{-7/6+C\varepsilon},\label{eq:expectation comparison}
		\end{equation}
		uniformly for $\gamma\in\{1,\dots,N+1\}.$
	\end{lem}
	\begin{proof}
		Recall the relationship between the eigenvalues of $\mathcal{G}$ and $G$,
		\[
		m_N=\frac{1}{N}{\rm Tr}G=\frac{1}{N}{\rm Tr}\mathcal{G}-\frac{1-\phi}{z},
		\]
		with
		\[
		|\frac{1}{N}{\rm Tr}\mathcal{G}-\frac{1}{N}{\rm Tr}\mathcal{G}^{(\gamma)}|=\frac{zG_{\gamma\gamma}}{N}\mathbf{x}_{\gamma}^{*}(\mathcal{G}^{(\gamma)})^2\mathbf{x}_{\gamma}.
		\]
		Here we can assume $|1-\phi|$ to be $1$ after introducing a multiplicative constant. Then
		\[
		\mathbb{E}f(N\eta\operatorname{Im} m_{N,\gamma}(z))=\mathbb{E}F(N\eta\operatorname{Im}(m_{N,\gamma}^{(\gamma)}(z)-\frac{1}{Nz}+\frac{zG_{\gamma\gamma}}{N}\mathbf{x}_{\gamma}^{*}(\mathcal{G}^{(\gamma)})^2\mathbf{x}_{\gamma})).
		\]
		Denoting
		\begin{equation}
		y^V=\eta zG_{\gamma\gamma}\mathbf{x}_{\gamma}^{V^{*}}(\mathcal{G}^{(\gamma)})^2\mathbf{x}^V_{\gamma},
		\end{equation}
		by the Taylor expansion we obtain that
		\begin{equation}
		\begin{split}
		&\mathbb{E}F(N\eta\operatorname{Im} m_{N,\gamma}(z))=\mathbb{E}F(N\eta\operatorname{Im} m_{N,\gamma}^{(\gamma)}(z)-\operatorname{Im}\frac{\eta}{z}+\operatorname{Im} y^V)\\
		=&\mathbb{E}\Big[F(N\eta\operatorname{Im} m_{N,\gamma}^{(\gamma)}(z)-\frac{\eta^2}{|z|^2})\\
		&+\sum_{k=1}^3\frac{1}{k!}F^{(k)}(N\eta\operatorname{Im} m_{N,\gamma}^{(\gamma)}(z)-\frac{\eta^2}{|z|^2})(\operatorname{Im} y^V)^k+O_{\prec}(N^{-4/3+C_{\epsilon}})\Big],\\
		\end{split}
		\end{equation}
		where we used the estimation
		$
		|y^V|\prec N^{-1/3+C_{\epsilon}}.
		$
		
		Then the left-hand side of (\ref{eq:expectation comparison}) reads
		\begin{equation}\label{eq:expectation comparison II}
		\begin{split}
		&|\mathbb{E}F(N\eta\operatorname{Im} m_{N,\gamma}(z))-\mathbb{E}F(N\eta\operatorname{Im} m_{N,\gamma+1}(z))|\\
		=&|\sum_{k=1}^3\frac{1}{k!}F^{(k)}(N\eta\operatorname{Im} m_{N,\gamma}^{(\gamma)}(z)-\frac{\eta^2}{|z|^2})(\operatorname{Im} y^V)^k\\
		&-\sum_{k=1}^3\frac{1}{k!}F^{(k)}(N\eta\operatorname{Im} m_{N,\gamma+1}^{(\gamma)}(z)-\frac{\eta^2}{|z|^2})(\operatorname{Im} y^W)^k+O_{\prec}(N^{-4/3+C_{\epsilon}})|\\
		=&|\sum_{k=1}^3\frac{1}{k!}F^{(k)}(N\eta\operatorname{Im} m_{N,\gamma}^{(\gamma)}(z)-\frac{\eta^2}{|z|^2})\left((\operatorname{Im} y^V)^k-(\operatorname{Im} y^W)^k\right)+O_{\prec}(N^{-4/3+C_{\epsilon}})|.
		\end{split}
		\end{equation}
		
		One can check that by Theorem \ref{thm:strong local law} and Lemma \ref{lem:basic property of m} as well as the choice of $\eta$,
		\[
		N\eta\operatorname{Im} m_{N,\gamma}^{(\gamma)}(z)-\frac{\eta^2}{|z|^2}=N\eta\operatorname{Im} m(z)+O_{\prec}(1)\prec 1.
		\]
		Then  for $k=1,2,3,4$, there exists $c>0$ such that with high probability
		\begin{equation}
		F^{(k)}(N\eta\operatorname{Im} m_{N,\gamma}^{(\gamma)}(z)-\frac{\eta^2}{|z|^2})\le N^c.
		\end{equation}
		
		Now the proof of (\ref{eq:expectation comparison}) is reduced to showing
		\begin{equation}\label{eq:comparing y^k}
		|(y^V)^k-(y^W)^k|\prec N^{-7/6+C_{\epsilon}},
		\end{equation}
		for $k=1,2,3$. Let
		\begin{equation}
		B:=\frac{(m-G_{\gamma\gamma})^2}{m^2G_{\gamma\gamma}}=\frac{1}{G_{\gamma\gamma}}+\frac{G_{\gamma\gamma}-2m}{m^2},
		\end{equation}
		so by Lemma \ref{lem:crude bound on G and G inverse} and Theorem \ref{thm:strong local law} we have
		\begin{equation}
		|B|\prec\frac{1}{(N\eta)^2}\le N^{-2/3+\epsilon}.
		\end{equation}
		
		On the other hand, we may write
		\begin{equation}
		\begin{split}
		G_{\gamma\gamma}&=\frac{m^2/(2m-G_{\gamma\gamma})}{m^2/(2m-G_{\gamma\gamma})B+1}=\frac{m^2}{2m-G_{\gamma\gamma}}\sum_{k\ge0}(-\frac{m^2}{zm-G_{\gamma\gamma}}B)^k.
		\end{split}
		\end{equation}
		Consequently, we obtain
		\begin{equation}
		\begin{split}
		y&=\sum_{k\ge 0}\eta z\frac{m^2}{2m-G_{\gamma\gamma}}(-\frac{m^2}{2m-G_{\gamma\gamma}}B)^k\mathbf{x}_{\gamma}^{*}(\mathcal{G}^{(1)})^2\mathbf{x}_{\gamma}=\sum_{k\ge 0}y_k,
		\end{split}
		\end{equation}
		where
		\[
		y_k:=\eta z\frac{m^2}{2m-G_{\gamma\gamma}}(-\frac{m^2}{2m-G_{\gamma\gamma}}B)^k\mathbf{x}_{\gamma}^{*}(\mathcal{G}^{(1)})^2\mathbf{x}_{\gamma}.
		\]
		Then one can check that
		\begin{equation}
		|y_k|\prec N^{-2/3}N^{-2k/3+\epsilon}N^{1/3+2\epsilon}\le N^{-1/3-2k/3+C_{\epsilon}}.
		\end{equation}
		Therefore it suffices to prove
		\begin{equation}
		\Big|\sum_{k=1}^3\left((\sum_{j\ge0}y^V_j)^k-(\sum_{j\ge0}y^W_j)^k\right)\Big|\prec N^{-7/6+C_{\epsilon}}.
		\end{equation}
		
		We note that for $j\ge2$, $|y_j|$ is sufficiently small, hence it suffices  to consider $y_0, y_1$ for the following three cases. Now let $\mathbb{E}_{\gamma}$ be  the conditional expectation with respect to $\xi^V_{\gamma}$ and $\xi^W_{\gamma}$.
		\begin{itemize}
			\item [a)] $k=1$.\par
			It suffices to bound
			\[
			|\operatorname{Im} (y^V_0+y^V_1)-\operatorname{Im} (y^W_0+y^W_1)|.
			\]
			We observe that
			\begin{equation}
			\mathbb{E}_{\gamma}(y_0^V-y_0^W)=0;
			\end{equation}
			\begin{equation}
			\begin{split}
			\big|\mathbb{E}_{\gamma}(y_1^V-y_1^W)\big|&=|\eta z\frac{1}{C^2}\mathbb{E}_{\gamma}\Big((\frac{1}{G_{\gamma\gamma}}+C)\mathbf{x}_{\gamma}^{V^{*}}(\mathcal{G}^{(\gamma)})^2\mathbf{x}_{\gamma}^V-(\frac{1}{G_{\gamma\gamma}}+C)\mathbf{x}_{\gamma}^{W^{*}}(\mathcal{G}^{(\gamma)})^2\mathbf{x}_{\gamma}^W\Big)|\\
			&=|\frac{\eta z^2}{C^2}\mathbb{E}_{\gamma}\Big((\xi_{\gamma}^V)^4-(\xi_{\gamma}^W)^4\Big)(\Sigma^{1/2}\mathbf{u}_{\gamma}^{*}\mathcal{G}^{(\gamma)}\mathbf{u}_{\gamma}\Sigma\mathbf{u}_{\gamma}^{*}(\mathcal{G}^{(\gamma)})^2\mathbf{u}_{\gamma}\Sigma^{1/2})|\\
			&\prec|\frac{\eta z^2}{C^2}N^{-10/12+C_{\epsilon}}\Sigma^{1/2}\mathbf{u}_{\gamma}^{*}\mathcal{G}^{(\gamma)}\mathbf{u}_{\gamma}\Sigma\mathbf{u}_{\gamma}^{*}(\mathcal{G}^{(\gamma)})^2\mathbf{u}_{\gamma}\Sigma^{1/2}|\\
			&\prec N^{-7/6+C_{\epsilon}},
			\end{split}
			\end{equation}
			where we have used Condition \ref{cond:3} and the fact that second moments of $\xi^V$, $\xi^W$ match .
			\item [b)] $k=2$.\par
			In this case we only need to consider
			\[
			|(\operatorname{Im} y_0^V)^2-(\operatorname{Im} y_0^W)^2|.
			\]
			Similarly, we observe that
			\begin{equation}
			\begin{split}
			&\big|\mathbb{E}_{\gamma}\Big((y_0^V)^2-(y_0^W)^2\Big)\big|\\
			=&|\frac{\eta^2z^2}{C}\mathbb{E}_{\gamma}\Big((\xi_{\gamma}^V)^4-(\xi_{\gamma}^W)^4\Big)(\Sigma^{1/2}\mathbf{u}_{\gamma}^{*}(\mathcal{G}^{(\gamma)})^2\mathbf{u}_{\gamma}\Sigma\mathbf{u}_{\gamma}^{*}(\mathcal{G}^{(\gamma)})^2\mathbf{u}_{\gamma}\Sigma^{1/2})|\\
			\prec& N^{-3/2+C_{\epsilon}}.
			\end{split}
			\end{equation}
			\item[c)] $k=3$.\par
			In this case we need to bound
			\[
			|(\operatorname{Im} y_0^V)^3-(\operatorname{Im} y_0^W)^3|.
			\]
			We observe that
			\begin{equation}
			\begin{split}
			|\mathbb{E}_{\gamma}\Big((y_0^V)^3-(y_0^W)^3\Big)|&=|\frac{\eta^3z^3}{C}\mathbb{E}_{\gamma}\Big((\xi_{\gamma}^V)^6-(\xi_{\gamma}^W)^6\Big)(\Sigma^{1/2}\mathbf{u}_{\gamma}^{*}(\mathcal{G}^{(\gamma)})^2\mathbf{u}_{\gamma}\Sigma^{1/2})^3|\\
			&\prec N^{-11/6+C_{\epsilon}}.
			\end{split}
			\end{equation}
		\end{itemize}
		
		Finally, combining all the results, we see that (\ref{eq:expectation comparison}) holds. Thus we complete the proof of Lemma \ref{lem:GreenCompLem1}.
	\end{proof}

	\section{Proof of Theorem \ref{thm:Rigidity of eigenvalues with large support} and Theorem \ref{thm:Edge universality with large support}}\label{sec:8}
	\subsection{Proof of Theorem \ref{thm:Rigidity of eigenvalues with large support}}
	We need the next lemma to prove Theorem \ref{thm:Rigidity of eigenvalues with large support}.
	\begin{lem}\label{lem:match tilde_X with X}
		Suppose $\xi_i$'s satisfy the assumptions in Theorem \ref{thm:Rigidity of eigenvalues with large support}. Then there exists one matrix $\tilde{X}=(\tilde{x}_{ij})$, such that the elements $\tilde{\xi}_i$'s satisfy Condition \ref{cond:3} with $q=O(N^{-1/2}\log N)$, and the first four moments of $\xi_i$ and $\tilde{\xi}_i$ match for all $i$, that is
		\begin{equation}
		\mathbb{E}\xi_i^k=\mathbb{E}\tilde{\xi}_i^k,\qquad k=1,2,3,4.
		\end{equation}
	\end{lem}
	The proof of this lemma can be found in \cite{Yin2014}. We note that $\tilde{X}$ satisfies the conditions of Theorem \ref{thm:Rigidity of eigenvalues with large support}. Now we process to prove Theorem \ref{thm:Rigidity of eigenvalues with large support}.
	\begin{proof}
		Note that from Theorem \ref{thm:strong local law}, $\tilde{X}$ satisfies (\ref{eq:strong average law with large support}). We use the Green function comparison idea to show that (\ref{eq:strong average law with large support}) also holds for $X$. Since we have the trivial bound
		\[
		\max_{ij}G_{ij}\le C\eta^{-1}\le N,
		\]
		for any $X$ with $q\le N^{-c}$. Then by Lemma \ref{lem:equivalence}, it suffices to show that
		\begin{equation}
		\mathbb{E}|m_N-m|^p\prec (N\eta)^{-p},
		\end{equation}
		for $X$ with $q\le N^{-c}$.
		
		Firstly, recall the relationship (\ref{eq:relationship between m_w and m_N}). For simplification, we denote $m_M:=N^{-1}{\rm Tr}\mathcal{G}$ and $\underline{m}:=m+(1-\phi)z^{-1}$, so
		\[
		m_M=m_N+(1-\phi)z^{-1}.
		\]
		Then it is equivalent to showing that
		\begin{equation}\label{eq:m_M in thm1}
		\mathbb{E}|m_M-\underline{m}|^p\prec(N\eta)^{-p}.
		\end{equation}
		
		For $\gamma=0,\cdots,N$, let $X_{\gamma}$ be the matrix whose first $\gamma$ columns are the same as those of $X$ and the remaining $N-\gamma$ columns are the same as those of $\tilde{X}$ with entries $\sigma_i\tilde{\xi}_j\mathbf{u}_{ij}$, where $\tilde{\xi}_j$'s satisfy the assumptions in Lemma \ref{lem:match tilde_X with X}. Then $X_0=\tilde{X}$ and $X_N=X$. Denote $G_{\gamma}$, $\mathcal{G}_{\gamma}$ as the Green functions of $X_{\gamma}^*X_{\gamma}$ and $X_{\gamma}X_{\gamma}^*$ respectively, and $m_{M,\gamma}=N^{-1}{\rm Tr}\mathcal{G}_{\gamma}$.
		$\mathcal{G}_{\gamma}^{(\gamma)}$ and $m_{M,\gamma}^{(\gamma)}$ are defined similarly with $X_{\gamma}^{(\gamma)}$.
		
		The resolvent expansion gives
		\[
		\mathcal{G}_{\gamma}=\mathcal{G}_{\gamma}^{(\gamma)}-\mathcal{G}_{\gamma}\mathbf{x}_{\gamma}\mathbf{x}_{\gamma}^{*}\mathcal{G}_{\gamma}^{(\gamma)}.
		\]
		Consequently, by $m_{M,\gamma}=\frac{1}{N}{\rm Tr}\mathcal{G}_{\gamma}$, we may write
		\begin{equation}\label{eq:expansion of m_M}
		\begin{split}
		&m_{M,\gamma}-m_{M,\gamma}^{(\gamma)}
		=-\frac{1}{N}\xi_{\gamma}^2\mathbf{r}_{\gamma}^{*}\mathcal{G}_{\gamma}^{(\gamma)}\mathcal{G}_{\gamma}\mathbf{r}_{\gamma}.
		\end{split}
		\end{equation}
		Similarly,
		\[
		m_{M,\gamma-1}-m_{M,\gamma-1}^{(\gamma)}
		=-\frac{1}{N}\tilde\xi_{\gamma}^2\mathbf{r}_{\gamma}^{*}\mathcal{G}_{\gamma-1}^{(\gamma)}\mathcal{G}_{\gamma-1}\mathbf{r}_{\gamma}.
		\]

		We note that $|m_M-\underline{m}|^p=(m_M-\underline{m})^{p/2}(m_M^{*}-\underline{m}^{*})^{p/2}$ for any even integer $p>0$. In the following of this proof, we slightly abuse the notation by ignoring the conjugate $*$ in $m_M$ and $\underline{m}$ for simplicity. We shall see that this will not affect the validity of our result.
		
		When $\gamma=0$, from Theorem \ref{thm:strong local law} and the assumptions on $\tilde{X}$, it is clear that
		\begin{equation}\label{eq:induction at m_0}
		|m_{M,0}-\underline{m}|\prec(N\eta)^{-1},\qquad \mathbb{E}|m_{M,0}-\underline{m}|^p\prec(N\eta)^{-p}.
		\end{equation}
		
		The target is to show that $|\mathbb{E}(m_{M,N}-\underline{m})^p|\prec(N\eta)^p$. Actually, in the proof below, we use the deterministic form of the bound in (\ref{eq:m_M in thm1}), that is, we choose $\epsilon>0$ such that $|\mathbb{E}(m_{M,N}-\underline{m})^p|\le(N\eta)^pN^{\epsilon}$.
		
		Note that $\mathcal{G}_{\gamma}^{(\gamma)}=\mathcal{G}_{\gamma-1}^{(\gamma)}$, and
		\[
		\begin{split}
		\mathbf{r}_\gamma^*\mathcal{G}_{\gamma}^{(\gamma)}\mathcal{G}_{\gamma}\mathbf{r}_\gamma=\frac{\mathbf{r}_\gamma^*\mathcal{G}_{\gamma}^{(\gamma)}\mathcal{G}_{\gamma}^{(\gamma)}\mathbf{r}_\gamma}{1+\xi_{\gamma}^2\mathbf{r}_\gamma^*\mathcal{G}_{\gamma}^{(\gamma)}\mathbf{r}_\gamma}.
		\end{split}
		\]
		It's not hard to see that the local law also holds for $G_{\gamma}$, then by large deviations bounds
		\[
		\begin{split} |\mathbf{r}^{*}_{\gamma}\mathcal{G}_{\gamma}^{(\gamma)}\mathbf{r}_{\gamma}|\le&\sigma_1^2|\mathbf{u}_{\gamma}^{*}\mathcal{G}_{\gamma}^{(\gamma)}\mathbf{u}_{\gamma}|\prec |\frac{1}{M}{\rm Tr}(\mathcal{G}_{\gamma}^{(\gamma)})|+|\frac{1}{M}{\rm Tr}(\mathcal{G}_{\gamma}^{(\gamma)})|^2\le C,\\
		|\frac{1}{N}\mathbf{r}^{*}_{\gamma}\mathcal{G}_{\gamma}^{(\gamma)}\mathcal{G}_{\gamma}^{(\gamma)}\mathbf{r}_{\gamma}|\le&\frac{\sigma_1^2}{N}|\mathbf{u}_{\gamma}^{*}\mathcal{G}_{\gamma}^{(\gamma)}\mathcal{G}_{\gamma}^{(\gamma)}\mathbf{u}_{\gamma}|\prec \frac{1}{N}\bigg(|\frac{1}{M}{\rm Tr}(\mathcal{G}_{\gamma}^{(\gamma)}\mathcal{G}_{\gamma}^{(\gamma)})|+|\frac{1}{M}{\rm Tr}(\mathcal{G}_{\gamma}^{(\gamma)}\mathcal{G}_{\gamma}^{(\gamma)})|^2\bigg)\\
		\le &\frac{1}{N}\frac{1}{M}\|\mathcal{G}_{\gamma}^{(\gamma)}\|_F^2\prec \frac{1}{N\eta},
		\end{split}
		\]
		where we use
		\[
		\frac{1}{N} \frac{1}{M}\|\mathcal{G}_{\gamma}^{(\gamma)}\|_{F}^2
		=\frac{1}{N} M(\frac{\operatorname{Im} {\rm Tr}\mathcal{G}_{\gamma}^{(\gamma)}}{M^2\eta})=\frac{1}{N} M(\frac{N\operatorname{Im} m+N\Theta }{M^2\eta}-\frac{N-M}{M^2|z|^2})
		\prec \frac{q+\sqrt{\eta}}{N\eta}\le \frac{1}{N\eta}.
		\]
		Using Taylor's expansion
		\[
		\frac{1}{1+\xi_\gamma^2\mathbf{r}_\gamma^*\mathcal{G}_{\gamma}^{(\gamma)}\mathbf{r}_\gamma}=\sum_{k\ge 0}\bigg(\frac{1}{1+\phi\mathbf{r}_\gamma^*\mathcal{G}_{\gamma}^{(\gamma)}\mathbf{r}_\gamma}\bigg)^{k+1}\frac{\big(-(\xi_\gamma^2-\phi)\mathbf{r}_\gamma^*\mathcal{G}_{\gamma}^{(\gamma)}\mathbf{r}_\gamma\big)^k}{k!}.
		\]
		and the fact that $|1+\phi\mathbf{r}_\gamma^*\mathcal{G}_{\gamma}^{(\gamma)}\mathbf{r}_\gamma|^{-1}\le C$, $\xi_{\gamma}^2=\xi_\gamma^2-\phi+\phi$, the RHS of (\ref{eq:expansion of m_M})  can be written as
		\[
		\sum_{k\ge 0}(\xi_\gamma^2-\phi)^kA_{k,\gamma},
		\]
		where $A_{k,\gamma}$ is independent of $\xi_\gamma$ and for any $k\ge 0$,
		\[
		|\mathbb{E}(A_{k,\gamma})^p|\le (N\eta)^{-p}N^\epsilon.
		\]
		Similarly, we can write
		\begin{equation}\label{eq:decomp}
		m_{M,\gamma-1}-m_{M,\gamma-1}^{(\gamma)}=\sum_{k\ge 0}(\tilde \xi_{\gamma}^2-\phi)^kA_{k,\gamma-1},\quad \text{ while } 	A_{k,\gamma-1}=A_{k,\gamma}.
		\end{equation}

		Hence, we can write
		\begin{equation}\label{eq:gamma}
		\begin{split}
		&\mathbb{E}(m_{M,\gamma}-\underline{m})^p\\
		=&\mathbb{E}(m_{M,\gamma}^{(\gamma)}-\underline{m})^p+\mathbb{E}\sum_{k=1}^p\binom{p}{k}(m_{M,\gamma}^{(\gamma)}-\underline{m})^{p-k}\bigg(	\sum_{k_1\ge 0}(\xi_\gamma^2-\phi)^{k_1}A_{k_1,\gamma}\bigg)^k,
		\end{split}
		\end{equation}
		and similarly,
		\begin{equation}\label{eq:gamma-1}
		\begin{split}
		&\mathbb{E}(m_{M,\gamma-1}-\underline{m})^p\\
		=&\mathbb{E}(m_{M,\gamma-1}^{(\gamma)}-\underline{m})^p+\mathbb{E}\sum_{k=1}^p\binom{p}{k}(m_{M,\gamma-1}^{(\gamma)}-\underline{m})^{p-k}\bigg(	\sum_{k_1\ge 0}(\tilde\xi_\gamma^2-\phi)^{k_1}A_{k_1,\gamma}\bigg)^k.
		\end{split}
		\end{equation}
		We claim the fact that $m_{M,\gamma-1}^{(\gamma)}=m_{M,\gamma}^{(\gamma)}$, $m_{M,\gamma}^{(\gamma)}$ is independent of $\xi_\gamma$ and $\tilde\xi_\gamma$,  $\mathbb{E}(\xi_\gamma^2-\phi)^k=\mathbb{E}(\tilde \xi_\gamma^2-\phi)^k$ for $k\le 2$, and for any $k\ge 3$,
		\[
		\mathbb{E}(\xi_\gamma^2-\phi)^k\le N^{-1}\log Nq^{k-2},\quad 	\mathbb{E}(\tilde\xi_\gamma^2-\phi)^k\le N^{-1}\log Nq^{k-2}.
		\]
		Then, comparing (\ref{eq:gamma}) and (\ref{eq:gamma-1}), we infer from the Cauchy-Schwartz inequality  that
		\begin{equation}\label{eq:ep1}
		\begin{split}
		&|\mathbb{E}(m_{M,\gamma}-\underline{m})^p|
		\le|\mathbb{E}(m_{M,\gamma-1}-\underline{m})^p|+\sum_{k=1}^p\binom{p}{k}\frac {(N\eta)^{-k}} {N^{1+c-\epsilon}}|\mathbb{E}(m_{M,\gamma-1}^{(\gamma)}-\underline{m})^{2(p-k)}|^{1/2}.
		\end{split}
		\end{equation}
		Moreover, we know that
		\begin{equation}\label{eq:ep2}
		\begin{split}
		&|\mathbb{E}(m_{M,\gamma-1}^{(\gamma)}-\underline{m})^{2(p-k)}|=|\mathbb{E}(m_{M,\gamma-1}^{(\gamma)}-m_{M,\gamma-1}+m_{M,\gamma-1}-\underline{m})^{2(p-k)}|\\
		=&\sum_{l=0}^{2(p-k)}\binom{2(p-k)}{l}|\mathbb{E}((m_{M,\gamma-1}^{(\gamma)}-m_{M,\gamma-1})^l(m_{M,\gamma-1}-\underline{m}))^{2(p-k)-l}|\\
		\le &\sum_{l=0}^{2(p-k)}\binom{2(p-k)}{l}\bigg(|\mathbb{E}(m_{M,\gamma-1}^{(\gamma)}-m_{M,\gamma-1})^{2l}|\bigg)^{1/2}\bigg(|\mathbb{E}(m_{M,\gamma-1}-\underline{m})^{4(p-k)-2l}|\bigg)^{1/2}\\
		\le &C_pN^{c_p\epsilon}\sum_{l=0}^{2(p-k)}(N\eta)^{-l}\bigg(|\mathbb{E}(m_{M,\gamma-1}-\underline{m})^{4(p-k)-2l}|\bigg)^{1/2}
		\end{split}
		\end{equation}
		for some constants $C_p$ and $c_p$, where the last inequality is by (\ref{eq:decomp}).

		We then use (\ref{eq:ep1}) and (\ref{eq:ep2}) to complete the induction. For $\gamma=0$, we already know that
		\[
		|\mathbb{E}(m_{M,0}-\underline{m})^p|\le (N\eta)^{-p}N^{\epsilon},\quad |\mathbb{E}(m_{M,0}^{(1)}-\underline{m})^p|\le (N\eta)^{-p}N^{\epsilon}.
		\]
		Then by (\ref{eq:ep1}) it's easy to see that for $\gamma=1$,
		\[
		|\mathbb{E}(m_{M,1}-\underline{m})^p|\le\bigg(1+\frac{1}{N^{1+c/2}}\bigg) (N\eta)^{-p}N^{\epsilon}.
		\]
		Now assume that for some $\gamma\ge 1$, there exists constant $a>0$ such that  $|\mathbb{E}(m_{M,\gamma-1}-\underline{m})^p|\le (1+N^{-1-c/2})^a(N\eta)^{-p}N^{\epsilon}$ for any fixed $p$. By (\ref{eq:ep2}),
		\begin{equation}\label{eq:induction}
		|\mathbb{E}(m_{M,\gamma-1}^{(\gamma)}-\underline{m})^{2(p-k)}|\le C_p\bigg(1+\frac{1}{N^{1+c/2}}\bigg)^a(N\eta)^{-2(p-k)}N^{c_p\epsilon}
		\end{equation}
		for some constants $C_p$ and $c_p$. Note that $\epsilon$ is arbitrary small, so plug (\ref{eq:induction}) into (\ref{eq:ep1}) to obtain
		\[
		|\mathbb{E}(m_{M,\gamma}-\underline{m})^p|\le\bigg(1+\frac{1}{N^{1+c/2}}\bigg)^{a+1} (N\eta)^{-p}N^{\epsilon}.
		\]
		Then by induction,
		\[
		|\mathbb{E}(m_{M,N}-\underline{m})^p|\le\bigg(1+\frac{1}{N^{1+c/2}}\bigg)^{N} (N\eta)^{-p}N^{\epsilon}\le (N\eta)^{-p}N^{2\epsilon},
		\]
		for arbitrary small $\epsilon>0$.

		A similar but more complicated procedure can lead to
		\[
		\mathbb{E}|m_{M,N}-\underline{m}|^p\le (N\eta)^{-p}N^{\epsilon},
		\]
		and the theorem follows from Chebyshev's inequality.
		The other conclusions in Theorem \ref{thm:Rigidity of eigenvalues with large support} can be obtained by the standard procedure used in the proof of Theorem \ref{thm:rigidity}. So we omit  details.

	\end{proof}
	
	\subsection{Proof of Theorem \ref{thm:Edge universality with large support}}
	We note that $\tilde{X}$ in Lemma \ref{lem:match tilde_X with X} satisfies the desired edge universality according to Theorem \ref{thm:Edge universality with small support}. Thus if we can prove the following lemma, then Theorem \ref{thm:Edge universality with large support} follows immediately.
	\begin{lem}\label{lem:Green comparsion for tilde_X}
		Let $X$ and $\tilde{X}$ be two matrices  in Lemma \ref{lem:match tilde_X with X}. Then there exist constants $\epsilon,\delta>0$ such that, for any $s\in\mathbb{R}$
		\begin{equation}
		\begin{split}
		\mathbb{P}^{\tilde{X}}(N^{2/3}&(\lambda_1-\lambda_{+})\le s-N^{-\epsilon})-N^{-\delta}\le\mathbb{P}^{X}(N^{2/3}(\lambda_1-\lambda_{+})\le s)\\
		&\le\mathbb{P}^{\tilde{X}}(N^{2/3}(\lambda_1-\lambda_{+})\le s+N^{-\epsilon})+N^{-\delta}
		\end{split}
		\end{equation}
		where $\mathbb{P}^{X}$ and $\mathbb{P}^{\tilde{X}}$ are the laws of $X$ and $\tilde{X}$, respectively.
	\end{lem}
	
	Most of the proof of Lemma \ref{lem:Green comparsion for tilde_X} is the same as the one of Theorem \ref{thm:Edge universality with small support}. We only write down the Green function comparison part, which is slightly different from before but simpler since we have the first four moments matching at this time.
	\begin{thm}\label{thm:Green comparsion of tilde_X}
		Let $X$ and $\tilde{X}$ be two matrices  in Lemma \ref{lem:match tilde_X with X}. Suppose $F:\mathbb{R}\rightarrow\mathbb{R}$ is a function whose derivatives satisfy
		\[
		\sup_{x\in\mathbb{R}}|F^{(l)}(x)|(1+|x|)^{-C_{2}}\le C_{2},\qquad l=1,2,3
		\]
		with some constant $C_{2}>0$. Then for any sufficiently small constant $\epsilon>0$ and  for any real
		numbers $E,E_{1}$ and $E_{2}$ satisfying
		\[
		|E-\lambda_{+}|,|E_{1}-\lambda_{+}|,|E_{2}-\lambda_{+}|\le N^{-2/3+\epsilon}
		\]
		and $\eta=N^{-2/3-\epsilon}$, we have
		\begin{equation}
		|\mathbb{E}F(N\eta\operatorname{Im} m_{N}(z))-\mathbb{E}F(N\eta\operatorname{Im} \tilde{m}_{N}(z))|\le N^{-c_1+C_{\epsilon}},\qquad z=E+\imath\eta,
		\end{equation}
		and
		\begin{equation}
		\Big|\mathbb{E}F\Big(\int_{E_{1}}^{E_{2}}N\operatorname{Im} m_{N}(y+\imath\eta){\rm d}y\Big)-\mathbb{E}F\Big(\int_{E_{1}}^{E_{2}}N\operatorname{Im} \tilde{m}_{N}(y+\imath\eta){\rm d}y\Big)\Big|\le N^{-c_1+C_{\epsilon}},
		\end{equation}
		where $c_1$ is a positive constant and $C_\epsilon\rightarrow 0$ as $\epsilon\rightarrow 0$.
	\end{thm}
	\begin{proof}
		We only prove the first inequality and the second one follows from similar arguments. The beginning part is the same as before. We split
		\begin{eqnarray*}
			&  & \mathbb{E}F(N\eta\operatorname{Im} m_{N}(z))-\mathbb{E}F(N\eta\operatorname{Im}
			\tilde{m}_{N}(z))\\
			& = & \sum_{\gamma=1}^{N}\Big\{\mathbb{E}F(N\eta\operatorname{Im} m_{N,\gamma}(z))-\mathbb{E}F(N\eta\operatorname{Im} m_{N,\gamma-1}(z))\Big\}.
		\end{eqnarray*}
		We will prove that
		\begin{equation}\label{eq:8.12}
		|\mathbb{E}F(N\eta\operatorname{Im} m_{N,\gamma}(z))-\mathbb{E}F(N\eta\operatorname{Im} m_{N,\gamma-1}(z))|\prec N^{-1-c+C\epsilon}.
		\end{equation}
		
		Use the resolvent expansion that $\mathcal{G}_{\gamma}=\mathcal{G}_{\gamma}^{(\gamma)}-\mathcal{G}_{\gamma}\mathbf{x}_\gamma\mathbf{x}_\gamma^*\mathcal{G}_{\gamma}^{(\gamma)}$ and the relationship $m_{M,\gamma}=m_{N}+(1-\phi)z^{-1}$, we have
		\[
		N\eta\operatorname{Im} m_{N,\gamma}(z)=N\eta\operatorname{Im} m^{(\gamma)}_{M,\gamma}(z)+(1-\phi)N^{-1/3-\epsilon}-\operatorname{Im}\frac{\eta\mathbf{x}_\gamma^*\mathcal{G}_{\gamma}^{(\gamma)}\mathcal{G}_{\gamma}^{(\gamma)}\mathbf{x}_\gamma}{1+\mathbf{x}_\gamma^*\mathcal{G}_{\gamma}^{(\gamma)}\mathbf{x}_\gamma}.
		\]
		We further expand the last term (ignoring $\operatorname{Im}$) of the last identity to
		\[
		\begin{split}
		&\eta(\xi_\gamma^2-\phi+\phi)\mathbf{r}_\gamma^*\mathcal{G}_{\gamma}^{(\gamma)}\mathcal{G}_{\gamma}^{(\gamma)}\mathbf{r}_\gamma\sum_{k\ge 0}\bigg(\frac{1}{1+\phi\mathbf{r}_\gamma^*\mathcal{G}_{\gamma}^{(\gamma)}\mathbf{r}_\gamma}\bigg)^{k+1}\bigg(\frac{\Big(-(\xi_\gamma^2-\phi)\mathbf{r}_\gamma^*\mathcal{G}_{\gamma}^{(\gamma)}\mathbf{r}_\gamma\Big)^k}{k!}\bigg)\\
		=&
		\frac{\eta\phi\mathbf{r}_\gamma^*\mathcal{G}_{\gamma}^{(\gamma)}\mathcal{G}_{\gamma}^{(\gamma)}\mathbf{r}_\gamma}{1+\phi\mathbf{r}_\gamma^*\mathcal{G}_{\gamma}^{(\gamma)}\mathbf{r}_\gamma}-\frac{\eta\phi\mathbf{r}_\gamma^*\mathcal{G}_{\gamma}^{(\gamma)}\mathcal{G}_{\gamma}^{(\gamma)}\mathbf{r}_\gamma(\xi_\gamma^2-\phi)\mathbf{r}_\gamma^*\mathcal{G}_{\gamma}^{(\gamma)}\mathbf{r}_\gamma}{(1+\phi\mathbf{r}_\gamma^*\mathcal{G}_{\gamma}^{(\gamma)}\mathbf{r}_\gamma)^2}+\frac{\eta(\xi_\gamma^2-\phi)\mathbf{r}_\gamma^*\mathcal{G}_{\gamma}^{(\gamma)}\mathcal{G}_{\gamma}^{(\gamma)}\mathbf{r}_\gamma}{1+\phi\mathbf{r}_\gamma^*\mathcal{G}_{\gamma}^{(\gamma)}\mathbf{r}_\gamma}+R_\gamma\\
		:=&A_\gamma+B_\gamma+C_\gamma+R_{\gamma},
		\end{split}
		\]
		where $A_{\gamma}$ is independent of $\xi_\gamma$, $\mathbb{E}B_{\gamma}=\mathbb{E}C_{\gamma}=0$. We have already known that
		\[
		\eta\mathbf{r}_\gamma^*\mathcal{G}_{\gamma}^{(\gamma)}\mathcal{G}_{\gamma}^{(\gamma)}\mathbf{r}_\gamma\prec q,\quad \mathbf{r}_\gamma^*\mathcal{G}_{\gamma}^{(\gamma)}\mathbf{r}_\gamma\le C, \quad |1+\phi\mathbf{r}_\gamma^*\mathcal{G}_{\gamma}^{(\gamma)}\mathbf{r}_\gamma|^{-1}\le C,
		\]
		\[\quad\mathbb{E}(\xi_\gamma^2-\phi)^{2k}\le N^{-1}\log N,k\ge 1.
		\]
		Hence, $\mathbb{E}|R_\gamma|^i\le q\times N^{-1}\log N \times N^{\epsilon}\le N^{-1-c+2\epsilon}$, $i=1,2$. Then
		\begin{equation} \label{diff}
		\begin{split}
		&F(N\eta\operatorname{Im} m_{N,\gamma}(z))-F\Big(N\eta\operatorname{Im} m^{(\gamma)}_{M,\gamma}(z)+(1-\phi)N^{-1/3-\epsilon}-\operatorname{Im} A_\gamma\Big)\\
		=&-F^{(1)}\Big(N\eta\operatorname{Im} m^{(\gamma)}_{M,\gamma}(z)+(1-\phi)N^{-1/3-\epsilon}-\operatorname{Im} A_\gamma\Big) \times \operatorname{Im}(B_\gamma+C_\gamma+R_\gamma)\\
		&+\frac{1}{2}F^{(2)}(\psi) \operatorname{Im}^2(B_\gamma+C_\gamma+R_\gamma),
		\end{split}
		\end{equation}
		where $\psi$ is some number between $N\eta m_{N,\gamma}(z)$ and $N\eta\operatorname{Im} m^{(\gamma)}_{M,\gamma}(z)+(1-\phi)N^{-1/3-\epsilon}-\operatorname{Im} A_\gamma$. By the local law and large deviation bounds,  $A_{\gamma}\prec q+\sqrt{\eta}\rightarrow 0$. Furthermore, we observe that from Theorem \ref{thm:Rigidity of eigenvalues with large support},
		\[
		N\eta\operatorname{Im} m_{N,\gamma}(z)\prec N\eta(\operatorname{Im} m(z)+(N\eta)^{-1})\le C,
		\]
		which implies
		\begin{equation}
		\Big|F^{(1)}\Big(N\eta\operatorname{Im} m^{(\gamma)}_{M,\gamma}(z)+(1-\phi)N^{-1/3-\epsilon}-\operatorname{Im} A_{\gamma}\Big)\Big|\prec 1,\quad |F^{(2)}(\psi)|\prec 1.
		\end{equation}
		Therefore,
		\[
		\begin{split}
		&\bigg|\mathbb{E}\left(F(N\eta\operatorname{Im} m_{N,\gamma}(z))-F(N\eta\operatorname{Im} m^{(\gamma)}_{M,\gamma}(z)+(1-\phi)N^{-1/3-\epsilon}-\operatorname{Im} A_\gamma)\right)\bigg|\\
		\le &C(\mathbb{E}|R_\gamma|+\mathbb{E}|R_\gamma|^2+\mathbb{E}|B_\gamma|^2+\mathbb{E}|C_\gamma|^2)\\
		\le &q\times N^{-1}\log N\times N^{\epsilon}\le N^{-1-c+2\epsilon}.
		\end{split}
		\]
		
		Similarly, we can prove that
		\[
		\bigg|\mathbb{E}\left(F(N\eta\operatorname{Im} m_{N,\gamma-1}(z))-F(N\eta\operatorname{Im} m^{(\gamma)}_{M,\gamma-1}(z)+(1-\phi)N^{-1/3-\epsilon}-\operatorname{Im} A_{\gamma-1})\right)\bigg|\le \frac 1{N^{1+c-2\epsilon}}.
		\]
		Note that $m^{(\gamma)}_{M,\gamma-1}(z)=m^{(\gamma)}_{M,\gamma}(z)$ and $A_{\gamma-1}=A_{\gamma}$, which conclude (\ref{eq:8.12}) and the Theorem holds.

	\end{proof}
	
	\section{Proof of Theorem \ref{thm:universality}}\label{sec:9}
	Suppose the matrix $X$ satisfies Condition \ref{cond:1} and Condition \ref{cond:2}. We can write the sample covariance matrix as
	\begin{equation}
	\mathcal{W}=XX^{*}=\sum_{i=1}^{N}\xi_i^2\mathbf{r}_i\mathbf{r}_i^{*},
	\end{equation}
	where  $\mathbf{r}_i=\Sigma^{1/2}\mathbf{u}_i$.
	
	For any fixed $\epsilon>0$, define
	\begin{equation}
	\alpha_N:=\mathbb{P}(|\hat{\xi}^2_i-M|>N^{1-\epsilon}).
	\end{equation}
	
	Using Condition \ref{cond:2}, we can see that for any $\delta>0$ and large enough $N$,
	\begin{equation}
	\alpha_N\le\delta N^{-1+2\epsilon}.
	\end{equation}
	
	Let $\rho(x)$ be the distribution of $\xi_i^2$. Then we define independent random variables $\zeta_i^s$, $\zeta_i^l$ and $c_i$, $1\le i\le N$ in the following ways:
	\begin{itemize}
		\item [1.] $\zeta_i^s$ has distribution density $\rho_s(x)$, where
		\begin{equation}
		\rho_s(x):=\mathbf{1}\Big(|x-\phi|\le N^{-\epsilon}\Big)\frac{\rho(x)}{1-\alpha_N};
		\end{equation}
		\item [2.] $\zeta_i^l$ has distribution density $\rho_l(x)$, where
		\begin{equation}
		\rho_l(x):=\mathbf{1}\Big(|x-\phi|>N^{-\epsilon}\Big)\frac{\rho(x)}{\alpha_N};
		\end{equation}
		\item [3.] $c_i$ is a Bernoulli $0-1$ random variable with $\mathbb{P}(c_i=1)=\alpha_N$ and $\mathbb{P}(c_i=0)=1-\alpha_N$.
	\end{itemize}
	
	It is easy to check
	\begin{equation}
	\begin{split}
	\xi_i^2&\overset{d}{=}\zeta_i^s(1-c_i)+\zeta_i^lc_i,
	\end{split}
	\end{equation}
	therefore we may write
	\begin{equation}\label{eq:split of W}
	\mathcal{W}=\sum_{i=1}^{N}\xi_i^2\mathbf{r}_i\mathbf{r}_i^{*}=\sum_{i=1}^{N}\Big(\zeta_i^s(1-c_i)+\zeta_i^lc_i\Big)\mathbf{r}_i\mathbf{r}_i^{*}
	\end{equation}
	We observe that
	\begin{equation}
	\mathbb{E}|\zeta_i^s-\phi|^2=O(N^{-1}\log N),
	\end{equation}
	so  $\zeta_i^s$ satisfies the assumptions in Theorem \ref{thm:Edge universality with large support}. We conclude that for the matrix
	\[
	\tilde{\mathcal{W}}:=\sum_{i=1}^{N}\zeta_i^s\mathbf{r}_i\mathbf{r}_i^{*},
	\]
	there exist constants $\epsilon,\delta>0$ such that for any $s\in\mathbb{R}$,
	\begin{equation}\label{eq:first main result}
	\begin{split}
	\mathbb{P}^{G}(N^{2/3}&(\lambda_1-\lambda_{+})\le s-N^{-\epsilon})-N^{-\delta}\le\mathbb{P}^{\tilde{W}}(N^{2/3}(\lambda_1-\lambda_{+})\le s)\\
	&\le\mathbb{P}^{G}(N^{2/3}(\lambda_1-\lambda_{+})\le s+N^{-\epsilon})+N^{-\delta},
	\end{split}
	\end{equation}
	where $\mathbb{P}^{G}$ denotes the law for a Gaussian covariance matrix and $\mathbb{P}^{\tilde{W}}$ denotes the law for $\tilde{\mathcal{W}}$.
	
	Now we write the right-hand side of (\ref{eq:split of W}) as
	\[
	\begin{split}
	\mathcal{W}&=\sum_{i=1}^{N}(\zeta_i^s+(\zeta_i^l-\zeta_i^s)c_i)\mathbf{r}_i\mathbf{r}_i^{*}:=\sum_{i=1}^{N}(\zeta_i^s+R_ic_i)\mathbf{r}_i\mathbf{r}_i^{*},
	\end{split}
	\]
	where $R_i:=\zeta_i^l-\zeta_i^s$. We aim to show that  the $R_ic_i$ terms have negligible effects on $\lambda_1$. Define the corresponding matrix as
	\[
	R^c:=\sum_{i=1}^{N}R_ic_i\mathbf{r}_i\mathbf{r}_i^{*}.
	\]
	Note that $c_i$ is independent of $\zeta_i^s$ and $\zeta_i^l$. In order to understand the spectral behavior of this matrix, we first introduce the following event
	\[
	A:=\{\sharp\{i:c_i=1\}\le N^{5\epsilon}\}.
	\]
	Since $c_i$'s are independent and identically distributed Bernoulli random variables, by Bernstein's inequality it is easy to check
	\begin{equation}\label{eq:estimation of probability for c_i}
	\begin{split}
	\mathbb{P}(A)&\ge 1-\exp(-N^{\epsilon}).
	\end{split}
	\end{equation}
	Without loss of generality, we will assume that $c_i=0$ for $i>N^{5\epsilon}$ and $c_i=1$ for $i\le N^{5\epsilon}$.
	On the other hand, by Condition \ref{cond:2}, we have
	\begin{equation}
	\begin{split}
	\mathbb{P}(|R_i|\ge\omega)&\le\mathbb{P}(|\zeta_{i}^l|\ge\frac{\omega}{2})=\mathbb{P}(|\hat{\xi}_{i}^2-M|\ge\omega N)=o(N^{-2}),
	\end{split}
	\end{equation}
	for any fixed constant $\omega>0$. Hence, the event
	\[
	A\bigcap\{\max_{i}|R_i|\le\omega\}
	\]
	happens with probability approaching to 1.  Hereafter, we will focus on this event. Define
	\[
	\mathcal{W}_t(\lambda)=\lambda I-\bigg(\tilde{\mathcal{W}}+t\sum_{i=1}^{N^{5\epsilon}}R_i\mathbf{r}_i\mathbf{r}_i^*\bigg),\quad t\in[0,1].
	\]
	In fact, by taking $\omega$ sufficiently small, the eigenvalues of $\tilde{\mathcal{W}}+t\sum_{i=1}^{N^{5\epsilon}}R_i\mathbf{r}_i\mathbf{r}_i^*$ are continuous in $t$.  Next, we aim to prove that for  $\lambda=\mu:=\lambda_1(\tilde{\mathcal{W}})\pm N^{-3/4}$,
	\begin{equation}\label{eq:pr}
	\mathbb{P}\Big(\det(\mathcal{W}_t(\mu))\ne 0,\forall t\in[0,1]\Big)=1-o(1).
	\end{equation}
	If (\ref{eq:pr}) holds, by continuity we know that the largest eigenvalue of $\tilde{\mathcal{W}}+t\sum_{i=1}^{N^{5\epsilon}}R_i\mathbf{r}_i\mathbf{r}_i^*$ will not cross the boundary $\lambda_1(\tilde{\mathcal{W}})\pm N^{-3/4}$. Hence  $\lambda_1(\mathcal{W})$ is sticking to $\lambda_1(\tilde{\mathcal{W}})$ with a rate smaller than $N^{-3/4}$, which concludes the theorem.
	
	Now we prove (\ref{eq:pr}). We know that the eigenvalues of GOE are separated at the scale of $N^{-2/3}$, so by (\ref{eq:first main result}),
	\[
	\mathbb{P}(|\lambda_k(\tilde{\mathcal{W}})-\mu|\ge N^{-3/4})=1-o(1).
	\]
	Therefore, $\mu$ is not an eigenvalue of $\tilde{\mathcal{W}}$, and
	\[
	\det({\mathcal{W}}_t(\mu))=\det(\mu-\tilde{\mathcal{W}})\det\bigg(1-t\sum_{i=1}^{N^{5\epsilon}}R_i\mathbf{r}_i\mathbf{r}_i^*\mathcal{G}^s(\mu)\bigg),
	\]
	where $\mathcal{G}^s(z)=(\tilde{\mathcal{W}}-z)^{-1}$ is the Green function. Hereafter, we ignore the superscript $s$ in $\mathcal{G}^s(z)$ for simplicity.  Let $z=\lambda_++\imath N^{-1+\delta}$ for some $\delta>0$. Then
	\begin{equation}\label{eq:dec}
	\begin{split}
	1-t\sum_{i=1}^{N^{5\epsilon}}R_i\mathbf{r}_i\mathbf{r}_i^*\mathcal{G}(\mu)=1-t\sum_{i=1}^{N^{5\epsilon}}R_i\mathbf{r}_i\mathbf{r}_i^*\big(\mathcal{G}(\mu)-\mathcal{G}(z)\big)-t\sum_{i=1}^{N^{5\epsilon}}R_i\mathbf{r}_i\mathbf{r}_i^*\mathcal{G}(z).
	\end{split}
	\end{equation}
	Note that for each $i$,
	\[
	\mathbf{r}_i\mathbf{r}_i^*\mathcal{G}(z)=\mathbf{r}_i\mathbf{r}_i^*\mathcal{G}^{(i)}(z)+\frac{\zeta_i^s\mathbf{r}_i\mathbf{r}_i^*\mathcal{G}^{(i)}(z)\mathbf{r}_i\mathbf{r}_i^*\mathcal{G}^{(i)}(z)}{1+\zeta_i^s\mathbf{r}_i^*\mathcal{G}^{(i)}(z)\mathbf{r}_i},
	\]
	while
	\[
	\begin{split}
	&\bigg|\mathbf{r}_i^*\mathcal{G}^{(i)}(z)\mathbf{r}_i-\frac{1}{M}\underline m\text{Tr}\Sigma\bigg|
	\\
	\le&\bigg|\mathbf{r}_i^*\mathcal{G}^{(i)}(z)\mathbf{r}_i-\frac{1}{M}\text{Tr}\mathcal{G}^{(i)}(z)\Sigma\bigg|+\bigg|\frac{1}{M}\text{Tr}(\mathcal{G}^{(i)}(z)-\underline{m})\Sigma\bigg|\prec N^{-1/6+\epsilon}.
	\end{split}
	\]
	Therefore, we can replace $\mathbf{r}_i^*\mathcal{G}^{(i)}(z)\mathbf{r}_i$ with $M^{-1}\underline{m}\text{Tr}\Sigma$ and write
	\begin{equation}\label{eq:second}
	\bigg\|t\sum_{i=1}^{N^{5\epsilon}}R_i\mathbf{r}_i\mathbf{r}_i^*\mathcal{G}(z)\bigg\|\le \max|R_i|\bigg\|\sum_{i=1}^{N^{5\epsilon}}\mathbf{r}_i\mathbf{r}_i^*\bigg\|\big\|\mathcal{G}(z)\big\|\le C\omega\|\sum_i\mathbf{r}_i\mathbf{r}_i^*\|\le \frac{1}{10},
	\end{equation}
	where we have used the fact that $w$ can be sufficiently small and $\|\mathcal{G}(z)\|\le C$ by Theorem \ref{thm:rigidity}. Then, it remains  to consider the second surm in (\ref{eq:dec}).
	
	Let $\beta_\alpha$ be the eigenvector of $\tilde{\mathcal{W}}$ corresponding to the $\alpha$-th eigenvalue $\lambda_\alpha$. Note that for any $\lambda_\alpha>\tau$ and $z^*=\lambda_\alpha+\imath N^{-1+\delta}$, we have $|\mathbf{r}_i^*\mathcal{G}(z^*)\mathbf{r}_i|\le C$ with high probability. Moreover,
	\[
	|\operatorname{Im} \mathbf{r}_i^*\mathcal{G}(z^*)\mathbf{r}_i|=(\operatorname{Im} z^*)\sum_{j}\frac{<\mathbf{r}_i,\beta_j>^2}{|\lambda_j-z^*|^2}\ge \frac{\operatorname{Im} z^*}{|\lambda_\alpha-z^*|^2}<\mathbf{r}_i,\beta_\alpha>^2=(\operatorname{Im} z^*)^{-1}<\mathbf{r}_i,\beta_\alpha>^2.
	\]
	Since $\delta$ can be arbitrary small, we have $<\mathbf{r}_i,\beta_\alpha>^2\prec N^{-1}$ for any $\alpha$ satisfying $\lambda_{\alpha}>\tau$.  Let $\alpha^*$ be the largest $\alpha$ satisfying this condition, so by the eigenvalue rigidity we have $\alpha^*\asymp kN$ for some constant $k$. The eigenvalue rigidity also implies that $\lambda_\alpha-\lambda_+\asymp (\alpha/N)^{2/3}$ for any $\alpha\ge N^{\epsilon}$.
	
	Recall $z=\lambda_++\imath N^{-2/3}$, so for each $i$,
	\[
	\begin{split}
	&\bigg|\mathbf{r}_i^*\big(\mathcal{G}(\mu)-\mathcal{G}(z)\big)\mathbf{r}_i\bigg|=\sum_\alpha<\mathbf{r}_i,\beta_\alpha>^2\bigg|\frac{1}{\mu-\lambda_\alpha}-\frac{1}{z-\lambda_\alpha}\bigg|\\
	\le&\sum_\alpha<\mathbf{r}_i,\beta_\alpha>^2\bigg(\frac{\eta}{(\lambda_+-\lambda_\alpha)^2+\eta^2}+\frac{(1+o(1))\eta^2}{|\lambda_\alpha-\mu||(\lambda_+-\lambda_\alpha)^2+\eta^2|}\bigg).
	\end{split}
	\]
	Firstly,
	\[
	\begin{split}
	&\sum_{\alpha\le N^{\epsilon}}<\mathbf{r}_i,\beta_\alpha>^2\bigg(\frac{\eta}{(\lambda_+-\lambda_\alpha)^2+\eta^2}+\frac{(1+o(1))\eta^2}{|\lambda_\alpha-\mu||(\lambda_+-\lambda_\alpha)^2+\eta^2|}\bigg)\\
	\prec&N^{\epsilon}N^{-1}(N^{2/3+\epsilon}+N^{3/4+\epsilon})\le N^{-1/4+2\epsilon}.
	\end{split}
	\]
	Secondly,
	\[
	\begin{split}
	&\sum_{ N^{\epsilon}\le \alpha\le \alpha^*}<\mathbf{r}_i,\beta_\alpha>^2\bigg(\frac{\eta}{(\lambda_+-\lambda_\alpha)^2+\eta^2}+\frac{(1+o(1))\eta^2}{|\lambda_\alpha-\mu||(\lambda_+-\lambda_\alpha)^2+\eta^2|}\bigg)\\
	\prec&\sum_{N^{\epsilon}\le \alpha\le \alpha^*}N^{-5/3+\epsilon}\bigg(\frac{\alpha}{N}\bigg)^{-4/3}\le N^{-1/3+\epsilon}\sum_{1\le \alpha\le kN}\alpha^{-4/3}\le N^{-1/3+2\epsilon}.
	\end{split}
	\]
	Lastly,
	\[
	\begin{split}
	&\sum_{ \alpha> \alpha^*}<\mathbf{r}_i,\beta_\alpha>^2\bigg(\frac{\eta}{(\lambda_+-\lambda_\alpha)^2+\eta^2}+\frac{(1+o(1))\eta^2}{|\lambda_\alpha-\mu||(\lambda_+-\lambda_\alpha)^2+\eta^2|}\bigg)\\
	\le &\frac{\eta}{(\lambda_+-\tau)^2}\sum_{\alpha}<\mathbf{r}_i,\beta_\alpha>^2\le\frac{\eta}{(\lambda_+-\tau)^2}\|\mathbf{r}_i\|^2\le N^{-2/3+\epsilon}.
	\end{split}
	\]
	Therefore, we have
	\begin{equation}\label{eq:first}
	\bigg\|t\sum_{i=1}^{N^{5\epsilon}}R_i\mathbf{r}_i\mathbf{r}_i^*\big(\mathcal{G}(\mu)-\mathcal{G}(z)\big)\bigg\|\prec N^{-1/4+7\epsilon}.
	\end{equation}
	Combining (\ref{eq:dec}), (\ref{eq:second}) and (\ref{eq:first}),  with probability approaching to 1 we have
	\[
	\det\bigg(1-t\sum_{i=1}^{N^{5\epsilon}}R_i\mathbf{r}_i\mathbf{r}_i^*\mathcal{G}(z)\bigg)\ne 0,\forall t\in[0,1],
	\]
	which concludes the theorem.

	\qedhere

	\newpage

	\mbox{}\par
	\begin{appendix}\label{Suppl1}
		
		\setcounter{section}{0}
		\renewcommand\thesection{\Roman{section}}
		\renewcommand\thesubsection{\roman{subsection}}

		In the following appendices, we provide the proofs of some lemmas and results which are omitted in the main text.

		\section{\label{Appendix A}Proof of results in Section \ref{sec:Preliminary-results}.}
		\subsection{Proof of Lemma \ref{lem:equivalence}}
		\begin{proof}
			If $\mathbb{E}X_{N}^{p}\prec\Phi_{N}^{p}$, then for any $\varepsilon>0$
			we get from Markov's inequality that
			\[
			\mathbb{P}(|X_{N}|>N^{\varepsilon}\Phi)\le\frac{\mathbb{E}|X_{N}|^{p}}{N^{\varepsilon p}\Phi^{p}}\le\frac{1}{N^{\varepsilon(p-1)}}.
			\]
			
			Choosing $p$ large enough (depending on $\varepsilon$) proves the
			``$\Leftarrow$'' part. Conversely, if $X_{N}\prec\Phi_{N}$, then
			for any $D>0$ we get
			\begin{multline*}
			|\mathbb{E}X_{N}|\le\mathbb{E}|X_{N}|=\mathbb{E}|X_{N}|\mathbf{1}(|X_{N}|\le N^{\varepsilon}\Phi_{N})+\mathbb{E}|X_{N}|\mathbf{1}(|X_{N}|>N^{\varepsilon}\Phi_{N})\\
			\le N^{\varepsilon}\Phi_{N}+\sqrt{\mathbb{E}|X_{N}|^{2}}\sqrt{\mathbb{P}(|X_{N}|>N^{\varepsilon}\Phi_{N})}\le N^{\varepsilon}\Phi_{N}+N^{C_{2}/2-D/2}.
			\end{multline*}
			
			Using $\Phi_{N}\ge N^{-C}$ and choosing $D$ large enough, we obtain
			the ``$\Rightarrow$'' part for $p=1$. The same implication for
			arbitrary $p$ follows from the fact that $X_{N}\prec\Phi_{N}$ implies
			$X_{N}^{p}\prec\Phi_{N}^{p}$ for any fixed $p$.
		\end{proof}

		\subsection{Proof of large deviation bounds in Lemma \ref{lem:large deviation}}
		
		\begin{proof}
			Let ${\cal F}_{k}=\sigma\{u_{1},\dots,u_{k}\}$ be the $\sigma$-algebra
			generated by $u_{1},\dots,u_{k}$. In particular, ${\cal F}_{0}$
			is the trivial $\sigma$-algebra, i.e., $\mathbb{E}(\cdot|{\cal F}_{0})$
			is the unconditional expectation. For $k=1,\dots,M-1$, we see by
			symmetry that conditioned on ${\cal F}_{k}$, $(u_{k+1},\dots,u_{M})^{\prime}$
			follows the uniform distribution on the $(M-k)$-dimensional sphere
			with radius $\sqrt{1-\sum_{i=1}^{k}u_{i}^{2}}$, namely,
			\begin{equation}
			(u_{k+1},\dots,u_{M})^{\prime}|{\cal F}_{k}\sim U\Big(\big(1-\sum_{i=1}^{k}u_{i}^{2}\big)^{1/2}\mathbb{S}^{M-k}\Big).\label{eq:conditional spherical distribution}
			\end{equation}
			
			Define the martingale difference sequence
			\begin{eqnarray*}
				\mathbf{s}_{k} & = & \sum_{i=1}^{M}b_{i}\{\mathbb{E}(u_{i}|{\cal F}_{k})-\mathbb{E}(u_{i}|{\cal F}_{k-1})\}.
			\end{eqnarray*}
			A direct observation from (\ref{eq:conditional spherical distribution})
			is that $\mathbf{s}_{k}=b_{k}u_{k}$. Then it follows from Theorem
			\ref{thm:mix_moment} in Appendix \ref{Appendix D}  and the Burkholder inequality \cite{burkholder1973}
			that for any positive integer $q$, there exists a constant $C_{q}>0$
			such that
			\begin{eqnarray*}
				\mathbb{E}|\mathbf{b}^{*}\mathbf{u}|^{2q} & = & \mathbb{E}\Big|\sum_{k=1}^{M}\mathbf{s}_{k}\Big|^{2q}
				\le  C_{q}\mathbb{E}\Big(\sum_{k=1}^{M}|\mathbf{s}_{k}|^{2}\Big)^{q}
				\le  C_{q}\mathbb{E}\Big(\sum_{k=1}^{M}|b_{k}|^{2}u_{k}^{2}\Big)^{q}\\
				& = & C_{q}\Big\{\sum_{1\le k_{1}\cdots k_{q}\le M}|b_{k_{1}}|^{2}\cdots|b_{k_{q}}|^{2}\mathbb{E}(u_{k_{1}}^{2}\cdots u_{k_{q}}^{2})\Big\}\\
				& \le & \frac{C_{q}}{M^{q}}\sum_{1\le k_{1}\cdots k_{q}\le M}|b_{k_{1}}|^{2}\cdots|b_{k_{q}}|^{2}=\frac{C_{q}}{M^{q}}\Big(\sum_{k=1}^{M}|b_{k}|^{2}\Big)^{q}=C_{q}\Big(\frac{\|\mathbf{b}\|^{2}}{M}\Big)^{q}.
			\end{eqnarray*}
			Then (\ref{eq:10.1-1}) follows from that for any $q\in\mathbb{Z}_{+}$,
			\[
			\mathbb{P}(|\mathbf{b}^{*}\mathbf{u}|>M^{\varepsilon}\sqrt{\frac{\|\mathbf{b}\|^{2}}{M}})\le\frac{\mathbb{E}|\mathbf{b}^{*}\mathbf{u}|^{2q}}{M^{2\varepsilon q}\Big(\frac{\|\mathbf{b}\|^{2}}{M}\Big)^{q}}\le\frac{C_{q}}{M^{2\varepsilon q}}.
			\]
			
			To show (\ref{eq:10.2}), we first show that
			\begin{equation}
			\Big|\sum_{k=1}^{M}a_{kk}(u_{k}^{2}-\frac{1}{M})\Big|\prec\frac{1}{M}\sqrt{\sum_{k=1}^{M}|a_{kk}|^{2}}.\label{eq:10.5}
			\end{equation}
			We construct the martingale difference sequence as
			\[
			\mathcal{N}_k:=\sum_{i=1}^Ma_{ii}\bigg(\mathbb{E}(u_{i}^2|\mathcal{F}_{k})-\mathbb{E}(u_{i}^2|\mathcal{F}_{k-1})\bigg).
			\]
			Note that $\mathcal{F}_{M}=\mathcal{F}_{M-1}$, and
			\[
			\mathbb{E}(u_i^2|\mathcal{F}_k)=\frac{1}{M-k}(1-\sum_{l=1}^ku_l^2), \forall i\ge k+1.
			\]
			Therefore,
			\[
			\begin{split}
			\mathcal{N}_k=&a_{kk}u_{k}^2+\sum_{i=k+1}^{M}a_{ii}\mathbb{E}(u_i^2|\mathcal{F}_k)-\sum_{i=k}^{M}a_{ii}\mathbb{E}(u_i^2|\mathcal{F}_{k-1})\\
			=&\bigg(a_{kk}-\sum_{i=k+1}^M\frac{a_{ii}}{M-k}\bigg)\bigg(u_{k}^2-\frac{1}{M-k+1}(1-\sum_{l=1}^{k-1}u_l^2)\bigg)\\
			=&\bigg(a_{kk}-\sum_{i=k+1}^M\frac{a_{ii}}{M-k}\bigg)\bigg(u_{k}^2-M^{-1}-\frac{1}{M-k+1}\sum_{l=k}^{M}(u_l^2-M^{-1})\bigg)\\
			=&\bigg(a_{kk}-\sum_{i=k+1}^M\frac{a_{ii}}{M-k}\bigg)\bigg(u_{k}^2-M^{-1}+\frac{1}{M-k+1}\sum_{l=1}^{k-1}(u_l^2-M^{-1})\bigg)\\
			:=&A_k\nu_k.
			\end{split}
			\]
			Use the Burkholder inequality to obtain that
			\[
			\begin{split}
			&\mathbb{E}{|\sum_{k=1}^Ma_{kk}(u_k^2-\frac{1}{M})|^{2q}}=\mathbb{E}{|\sum_{k=1}^M\mathcal{N}_k|^{2q}}\le C_q\mathbb{E}\bigg(\sum_k\mathcal{N}_k^2\bigg)^q\\
			\le & C_q\sum_{1\le k_1,\ldots,k_q\le M}A_{k_1}^2\cdots A_{k_q}^2\mathbb{E}\bigg(\nu_{k_1}^2\cdots\nu_{k_q}^2\bigg)\\
			\le &\frac{C_q}{M^{2q}}\bigg(\sum_kA_k^2\bigg)^q\le \frac{C_q\log M}{M^{2q}}\Big(\sum_ka_{kk}^2\Big)^q,
			\end{split}
			\]
			where the third line is by Theorem \ref{thm:mix_moment}.  So (\ref{eq:10.5}) holds.
			
			In order to complete the proof of (\ref{eq:10.2}), we will prove that
			\begin{equation}
			\Big|\sum_{j< k}^{M}a_{jk}u_{j}u_{k}\Big|\prec\frac{1}{M}\sqrt{\sum_{j\ne k}|a_{jk}|^{2}}.\label{eq:10.3}
			\end{equation}
			We begin with the martingale difference sequence
			\[
			\mathcal{Y}_l:=\sum_{j< k}a_{jk}\bigg(\mathbb{E}(u_{j}u_{k}|\mathcal{F}_l)-\mathbb{E}(u_{j}u_{k}|\mathcal{F}_{l-1})\bigg).
			\]
			Note that
			\[
			\begin{split}
			\sum_{j< k}\mathbb{E}(u_{j}u_{k}|\mathcal{F}_l)=&\sum_{j< k\le l}\mathbb{E}(u_{j}u_{k}|\mathcal{F}_l)+\sum_{j\le l< k}\mathbb{E}(u_{j}u_{k}|\mathcal{F}_l)+\sum_{l<j< k}\mathbb{E}(u_{j}u_{k}|\mathcal{F}_l)\\
			=&\sum_{j< k\le l}\mathbb{E}(u_{j}u_{k}|\mathcal{F}_l)=\sum_{j< k\le l}u_{j}u_{k}.
			\end{split}
			\]
			Then,
			\[
			\mathcal{Y}_l=\sum_{j< l}a_{jl}u_{j}u_{l}.
			\]
			By the Burkholder inequality,
			\[
			\begin{split}
			&\mathbb{E}|\sum_{j< k}^{M}a_{jk}u_{j}u_{k}|^{2q}\le C_q\mathbb{E}\bigg(\sum_l\mathcal{Y}_l^2\bigg)^q\le C_q\mathbb{E}\bigg(\sum_lu_l^2(\sum_{j<l}a_{jl}u_j)^2\bigg)^q\\
			\le &C_q\sum_{k_1}\sum_{i_1<k_1}\sum_{j_1<k_1}\ldots\sum_{k_q}\sum_{i_q<k_q}\sum_{j_q<k_q}a_{i_1k_1}a_{j_1k_1}\ldots a_{i_qk_q}a_{j_qk_q}\mathbb{E}\bigg(u_{k_1}^2u_{i_1}^2u_{j_1}^2\ldots u_{k_q}^2u_{i_q}^2u_{j_q}^2\bigg)\\
			\le &\frac{C_q}{M^{3q}}\bigg(\sum_k(\sum_{j<k}a_{jk})^2\bigg)^q\\
			\le &\frac{C_q}{M^{2q}}(\sum_{j\ne k}a_{jk}^2)^q,
			\end{split}
			\]
			which concludes (\ref{eq:10.3}).
			Therefore, (\ref{eq:10.2})
			follows from
			\begin{eqnarray*}
				|\mathbf{u}^{*}A\mathbf{u}-\frac{1}{M}{\rm tr}A| & = & |\sum_{k=1}^{M}a_{kk}(u_{k}^{2}-\frac{1}{M})+\sum_{j\ne k}a_{jk}u_{j}u_{k}|\\
				& \prec & \frac{1}{\sqrt{2}M}\sqrt{\sum_{k=1}^{M}|a_{kk}|^{2}}+\frac{1}{\sqrt{2}M}\sqrt{\sum_{j\ne k}|a_{jk}|^{2}}\\
				& \le & \frac{1}{M}\sqrt{\sum_{k=1}^{M}|a_{kk}|^{2}+\sum_{j\ne k}|a_{jk}|^{2}}\\
				& = & \frac{1}{M}\|A\|_{F}.
			\end{eqnarray*}
			
			Denote the $i$th column of $A$ as $A_{\cdot i}$. Finally (\ref{eq:10.4-2})
			follows from (\ref{eq:10.1-1})  conditioned on $\mathbf{u}$,
			\[
			|\mathbf{u}^{*}A\tilde{\mathbf{u}}|\prec\sqrt{\frac{\|\mathbf{u}^{*}A\|^{2}}{M}},
			\]
			and
			\[
			\|\mathbf{u}^{*}A\|^{2}=\sum_{i=1}^{M}|\mathbf{u}^{*}A_{\cdot i}|^{2}\prec\sum_{i=1}^{M}\frac{\|A_{\cdot i}\|^{2}}{M}=\frac{1}{M}\sum_{i,j}a_{ij}^{2}=\frac{1}{M}\|A\|_{F}^{2},
			\]
			which concludes the lemma.
		\end{proof}
		
		\section{\label{Appendix B}Proof of the results in Section \ref{sec:5}.}
		\subsection{Proof of Lemma \ref{lem:10.5}}
		\begin{proof}
			We recall that
			\[
			{\cal G}=\big(\sum_{i\in{\cal I}}\mathbf{x}_{i}\mathbf{x}_{i}^{*}-zI\big)^{-1}.
			\]
			It follows from the resolvent identity that
			\begin{equation}
			{\cal G}-(-zm_{N}\Sigma-zI)^{-1}=(-zm_{N}\Sigma-zI)^{-1}\big(-zm_{N}\Sigma-\sum_{i\in{\cal I}}\mathbf{x}_{i}\mathbf{x}_{i}^{*}\big){\cal G}.\label{eq:10.21-1}
			\end{equation}
			Using the Sherman-Morrison formula (see, e.g., (2.2) of \cite{BaiandSilverstein1995} or Lemma \ref{lem:Sherman-Morrison formula} in Appendix \ref{Appendix D},
			we have
			\begin{equation}
			\mathbf{x}_{i}\mathbf{x}_{i}^{*}{\cal G}=\frac{\mathbf{x}_{i}\mathbf{x}_{i}^{*}{\cal G}^{(i)}}{1+\mathbf{x}_{i}^{*}{\cal G}^{(i)}\mathbf{x}_{i}}.\label{eq:10.22-1}
			\end{equation}
			Using (\ref{eq:10.21-1}), (\ref{eq:10.22-1}) and (\ref{eq:resolvent}), we have
			\begin{eqnarray*}
				{\cal G}-(-zm_{N}\Sigma-zI)^{-1} & = & \frac{1}{N}\sum_{i\in{\cal I}}\frac{(-zm_{N}\Sigma-zI)^{-1}\Sigma{\cal G}}{1+\mathbf{x}_{i}^{*}{\cal G}^{(i)}\mathbf{x}_{i}}\\
				&  & -\sum_{i\in{\cal I}}\frac{(-zm_{N}\Sigma-zI)^{-1}\mathbf{x}_{i}\mathbf{x}_{i}^{*}{\cal G}^{(i)}}{1+\mathbf{x}_{i}^{*}{\cal G}^{(i)}\mathbf{x}_{i}}\\
				& = & \sum_{i\in{\cal I}}\frac{(m_{N}\Sigma+I)^{-1}}{z(1+\mathbf{x}_{i}^{*}{\cal G}^{(i)}\mathbf{x}_{i})}\big(\mathbf{x}_{i}\mathbf{x}_{i}^{*}{\cal G}^{(i)}-\frac{1}{N}\Sigma{\cal G}\big),
			\end{eqnarray*}
			which concludes the lemma.
		\end{proof}
		
		\subsection{Proof of Lemma \ref{lem:ward}}
		\begin{proof}
			Let $\tilde{\lambda}_{k}$ and $\tilde{\mathbf{v}}_{k}$ be the $k$-th
			largest eigenvalue of ${\cal W}^{(T)}$ and the eigenvector corresponding
			to $\tilde{\lambda}_{k}$ respectively for $k=1,\dots,M$. Denote
			$\tilde{\mathbf{v}}_{k}(i)$ as the $i$-th entry of $\tilde{\mathbf{v}}_{k}$.
			We observe that
			\begin{multline*}
			\sum_{i,j\in\{1,\dots,M\}}|{\cal G}_{ij}^{(T)}|^{2}=\sum_{i\in\{1,\dots,M\}}\big({\cal G}^{(T)}({\cal G}^{(T)})^{*}\big)_{ii}\\
			=\sum_{i=1}^{M}\Big\{\Big(\sum_{k_{1}=1}^{M}\frac{\tilde{\mathbf{v}}_{k_{1}}(i)\tilde{\mathbf{v}}_{k_{1}}^{*}}{\tilde{\lambda}_{k_{1}}-z}\Big)\Big(\sum_{k_{1}=1}^{M}\frac{\tilde{\mathbf{v}}_{k_{2}}\tilde{\mathbf{v}}_{k_{2}}^{*}(i)}{\tilde{\lambda}_{k_{2}}-z^{*}}\Big)\Big\}\\
			=\sum_{i=1}^{M}\sum_{k=1}^{M}\frac{\tilde{\mathbf{v}}_{k}(i)\tilde{\mathbf{v}}_{k}^{*}(i)}{|\tilde{\lambda}_{k}-z|^{2}}=\sum_{k=1}^{M}\eta^{-1}\operatorname{Im}\Big(\frac{1}{\tilde{\lambda}_{k}-z}\Big)=\eta^{-1}\operatorname{Im}{\rm Tr}{\cal G}^{(T)},
			\end{multline*}
			which concludes the lemma.
		\end{proof}

		\subsection{Proof of Lemma \ref{lem:trace_difference}}
		\begin{proof}
			The first inequality follows from Theorem A.6 of \cite{BaiandSilverstein2010}. To show the second and the third inequalities, we observe that
			\[
			{\rm Tr}({\cal G}^{(i)}-{\cal G})=\frac{N-1-M}{z}-\frac{N-M}{z}+{\rm Tr}(G^{(i)}-G)=-\frac{1}{z}+{\rm Tr}(G^{(i)}-G),
			\]
			so the results follow.
		\end{proof}
		
		\subsection{Proof of Proposition \ref{prop:Imm(z)}}
		\begin{proof}
			Following Lemma 1 of \cite{Jing2010}, taking imaginary part
			and multiplying $(\operatorname{Im} m)^{-1}$ on both sides of $z=f(m)$, we get
			\[
			\frac{1}{|m|^{2}}-\phi\int\frac{x\pi({\rm d}x)}{|1+xm|^{2}}=\frac{\operatorname{Im} z}{\operatorname{Im} m}>0.
			\]
			Hence
			\begin{equation}
			\frac{1}{|m|^{2}}>\phi\int\frac{x\pi({\rm d}x)}{|1+xm|^{2}}.\label{eq:64}
			\end{equation}
			
			By (\ref{eq:64}) and the Cauchy-Schwartz inequality, we obtain from $z=f(m)$
			that
			\begin{eqnarray}
			|m| & = & \left|-\frac{1}{z}+\frac{\phi}{z}-\frac{\phi}{z}\int\frac{\pi({\rm d}x)}{1+xm}\right|\nonumber \\
			& < & \frac{|1-\phi|}{|z|}+\frac{\phi}{|z|}\left(\int\frac{x^{2}\pi({\rm d}x)}{|1+xm|^{2}}\right)^{1/2}\left(\int x^{-2}\pi({\rm d}x)\right)^{1/2}\nonumber \\
			& < & \frac{|1-\phi|}{|z|}+\frac{\sqrt{\phi}}{|z||m|}\left(\int x^{-2}\pi({\rm d}x)\right)^{1/2}.\label{eq:65}
			\end{eqnarray}
			This implies that
			\[
			|m|^{2}-\frac{|1-\phi||m|}{|z|}-\frac{\sqrt{\phi}}{|z|}\left(\int x^{-2}\pi({\rm d}x)\right)^{1/2}<0.
			\]
			Some basic calculations yield
			\[
			|m|\le\frac{|1-\phi|+\sqrt{|z||1-\phi|+4|z|\sqrt{\phi}\left(\int x^{-2}\pi({\rm d}x)\right)^{1/2}}}{2|z|}.
			\]
			Since ${\rm supp}(\pi)$ is uniformly bounded away from $0$ for
			all $N$,  $\int x^{-2}\pi({\rm d}x)$ is uniformly bounded.  Then from (\ref{eq:65}), we have
			\[
			\sup_{|z|\in[\tau,\tau^{-1}]}\sup_{N}|m|\le C,
			\]
			for some constant $C>0$.
			
			Suppose $\inf_{|z|\in[\tau,\tau^{-1}]}\inf_{N}|m|=0$. Then we can
			choose a sequence $\{z_{N}\}_{N=1}^{\infty}\subset\{\tau\le|z|\le\tau^{-1}\}$
			such that $m(z_{N})\to0$ as $N\to\infty$.  From $z=f(m)$, we have for all $N$
			\[
			z_{N}m(z_{N})+1=\phi\int\frac{xm(z_{N})\pi({\rm d}x)}{1+xm(z_{N})},
			\]
			which implies that with probability $1$, as $N\to\infty$,
			\[
			\phi\int\frac{xm(z_{N})\pi({\rm d}x)}{1+m(z_{N})}\to1.
			\]
			However one can see
			\begin{eqnarray*}
				\left|\phi\int\frac{xm(z_{N})\pi({\rm d}x)}{1+xm(z_{N})}\right| & \le & |m(z_{N})|\phi\int\frac{x\pi({\rm d}x)}{|1+xm(z_{N})|}\to0,
			\end{eqnarray*}
			which is a contradiction. Therefore we have
			\[
			\inf_{|z|\in[\tau,\tau^{-1}]}\inf_{N}|m|>0.
			\]
			Finally,
			
			\[
			\operatorname{Im} m=\frac{\operatorname{Im} z}{\frac{1}{|m|^{2}}-\phi\int\frac{x\pi({\rm d}x)}{|1+xm|^{2}}}\ge|m^{2}|\operatorname{Im} z\ge C^{-1}\eta,
			\]
			which concludes the lemma.
		\end{proof}
		
		\subsection{Proof of Lemma \ref{lem:crude bound on G and G inverse}}
		\begin{proof}
			Given the event $\Xi$, $G_{ii}$ is within $\log^{-1}N$ distance
			to $m$ uniformly for $i\in{\cal I}$. Since $|m|\asymp1$, it then
			follows that
			\[
			\mathbf{1}(\Xi)G_{ii}\asymp1.
			\]
			
			Next, it follows from Lemma \ref{lem:resolvent} and the definition of $\Xi$ that			
			\begin{eqnarray*}
				\mathbf{1}(\Xi)|G_{ij}^{(k)}| & \le & |G_{ij}|+|\frac{G_{ik}G_{kj}}{G_{kk}}|\le\delta_{ij}|m|+\log^{-1}N+\frac{\log^{-2}N}{|m|-\log^{-1}N},\\
				\mathbf{1}(\Xi)|G_{ij}^{(k)}| & \ge & |G_{ij}|-|\frac{G_{ik}G_{kj}}{G_{kk}}|\ge\delta_{ij}|m|-\log^{-1}N-\frac{\log^{-2}N}{|m|-\log^{-1}N},
			\end{eqnarray*}
			which implies that given $\Xi$, $G_{ij}^{(k)}$ is within $2\log^{-1}N$
			distance to $m$ uniformly for all $i,j,k\in{\cal I}$ such that $i,j\ne k$.
			
			Applying this argument inductively, we conclude that for any index
			set $T$ such that $|T|\le C_{1}$, given $\Xi$, there exists a constant
			$C_{2}>0$ such that
			\begin{equation}
			\mathbf{1}(\Xi)|G_{ij}^{(T)}-\delta_{ij}m|\le C_{2}\log^{-1}N,\label{eq:3.16}
			\end{equation}
			uniformly for $i,j\in{\cal I}\backslash T$. Consequently, it follows from Proposition \ref{prop:Imm(z)} that there
			exists some constant $C>0$ such that
			\begin{equation}
			\mathbf{1}(\Xi)|G_{ij}^{(T)}|+\mathbf{1}(\Xi)\Big|\frac{1}{G_{ii}^{(T)}}\Big|\le C.\label{eq:1(Xi)Gij < 1}
			\end{equation}
			
			Let $G^{(T)}=V(L^{(T)}-zI)^{-1}V^{*}$ be the eigen-decomposition
			of $G^{(T)}$ where $V=(\mathbf{v}_{1},\dots,\mathbf{v}_{N})$ with
			orthonormal columns $\mathbf{v}_{1},\dots,\mathbf{v}_{N}$ is an orthogonal
			matrix and $L^{(T)}={\rm diag}(\lambda_{1}^{(T)},\cdots,\lambda_{N-|T|}^{(T)})$.
			Then we see that for any $i,j\in{\cal I}$,
			\begin{equation}
			|G_{ij}^{(T)}|\le\|G^{(T)}\|=\sup_{\|\mathbf{w}\|=1}\Big|\sum_{k=1}^{N-|T|}\frac{\mathbf{w}^{*}\mathbf{v}_{k}\mathbf{v}_{k}^{*}\mathbf{w}}{\lambda_{k}^{(T)}-z}\Big|\le\sup_{\|\mathbf{w}\|=1}\sum_{k=1}^{N-|T|}\frac{\mathbf{w}^{*}\mathbf{v}_{k}\mathbf{v}_{k}^{*}\mathbf{w}}{\eta}=\eta^{-1}.\label{eq:Gij<eta^-1}
			\end{equation}
			
			Therefore the desired result follows from (\ref{eq:1(Xi)Gij < 1})
			and (\ref{eq:Gij<eta^-1}).
		\end{proof}
		
		\subsection{Complement of the proof of Proposition \ref{prop:deformed weak local law}}\label{Suppl:prop deformed wll}
		\begin{proof}
			Let $\omega_1, \omega_2\in\mathbb{C}^{+}$. Some basic calculations yield that
			\begin{equation}
			|G_{ij}(w_{1})-G_{ij}(w_{2})|\le(\operatorname{Im} w_{1})^{-1}(\operatorname{Im} w_{2})^{-1}|w_{1}-w_{2}|,\qquad i,j\in{\cal I}.\label{eq:lipschitz}
			\end{equation}
			
			Let $z\equiv E+\imath\eta\in\mathbf{D}^{e}$. We construct a lattice
			as follows. Let $z_{0}=E+\imath$. Fix $\varepsilon\in(0,\tau/8)$.
			For $k=0,1,2,\dots,N^{5}-N^{4+\tau}$, define
			\begin{eqnarray*}
				\eta_{k} & = & 1-kN^{-5},\qquad z_{k}=E+\imath\eta_{k},\\
				\delta_{k} & = & (N\eta_{k})^{-1/2}+q^2,\qquad\Xi_{k}=\{\Lambda(z_{k})\le N^{\varepsilon}\sqrt{\delta_{k}}\}.
			\end{eqnarray*}
			
			Let $C>0$ be a fixed constant. We show by induction for $k=1,\dots,N^{5}-N^{4+\tau}$
			that if the two events
			\begin{equation}
			\Theta(z_{k-1})\le\frac{CN^{\varepsilon}\delta_{k-1}}{\sqrt{\kappa+\eta_{k-1}}+\sqrt{N^{\varepsilon}\delta_{k-1}}},\qquad\mathbf{1}(\Xi_{k-1})=1,\label{eq:induction hypothesis}
			\end{equation}
			hold with high probability, then
			\[
			\Theta(z_{k})\le\frac{N^{\varepsilon}\delta_{k}}{\sqrt{\kappa+\eta_{k}}+\sqrt{N^{\varepsilon}\delta_{k}}},\qquad\mathbf{1}(\Xi_{k})=1,
			\]
			hold with high probability.
			
			It is clear that (\ref{eq:induction hypothesis}) for $k=1$ follows
			from (\ref{eq:3.36}) and (\ref{eq:3.38-1}). We verify that if $\mathbf{1}(\Xi_{k-1})=1$, 	then $\Lambda(z_{k})\le\log^{-1}N,$
			$k=1,\dots,N^{5}-N^{4+\tau}$. Using the Lipschitz condition (\ref{eq:lipschitz}),
			we have
			\begin{eqnarray*}
				\mathbf{1}(\Xi_{k-1})\Lambda(z_{k}) & \le & \mathbf{1}(\Xi_{k-1})|\Lambda(z_{k})-\Lambda(z_{k-1})|+\mathbf{1}(\Xi_{k-1})\Lambda(z_{k-1})\\
				& \le & \max_{i,j}|G_{ij}(z_{k})-G_{ij}(z_{k-1})|+N^{\varepsilon}\sqrt{\delta_{k-1}}\\
				& \le & |z_{k}-z_{k-1}|\eta_{k}^{-1}\eta_{k-1}^{-1}+N^{\varepsilon}[(N\eta_{k-1})^{-1/4}+q]\\
				& \le & N^{-3-2\tau}+N^{\epsilon}[N^{-\tau/4}+q]\\
				& \le & \log^{-1}N.
			\end{eqnarray*}
			
			Let $D>0$ be an arbitrarily large number. Therefore, by (\ref{eq:Gmumu-Gnunu}),
			(\ref{eq:Zi+Lambdao bound}) and (\ref{eq:1(XI)(f(m)-z)}), we can
			choose $N_{0}\in\mathbb{Z}_{+}$ such that as $N\ge N_{0}$,
			\begin{multline}
			\sup_{k\in\{1,\dots,N^{5}-N^{4+\tau}\}}\mathbb{P}\Big(\mathbf{1}(\Xi_{k-1})(\Lambda_{o}(z_{k})+\max_{i\in{\cal I}}|G_{ii}(z_{k})-m_{N}(z_{k})|)>\frac{1}{2}N^{\varepsilon}\delta_{k}\Big)\\
			\le N^{-D},\label{eq:sup_k Lambdao + max Gii-mN}
			\end{multline}
			and
			\begin{equation}
			\sup_{k\in\{1,\dots,N^{5}-N^{4+\tau}\}}\mathbb{P}\Big(\mathbf{1}(\Xi_{k-1})|f\big(m_{N}(z_{k})\big)-z_{k}|>\delta_{k}\Big)\le N^{-D}.\label{eq:sup_k probability}
			\end{equation}
			Then, applying Proposition \ref{prop:stability}, we obtain from the
			induction hypothesis (\ref{eq:induction hypothesis}) and (\ref{eq:sup_k probability})
			that
			\begin{multline}
			\mathbb{P}\Big(\mathbf{1}(\Xi_{k-1})\Theta(z_{k})>CN^{\varepsilon/2}\sqrt{\delta_{k}}\Big)\\
			\le\mathbb{P}\Big(\mathbf{1}(\Xi_{k-1})\Theta(z_{k})>\frac{CN^{\varepsilon}\delta_{k}}{\sqrt{\kappa+\eta_{k}}+\sqrt{N^{\varepsilon}\delta_{k}}}\Big)\le N^{-D}.\label{eq:1(Xi_k-1)Theta(z_k) bound}
			\end{multline}
			
			Using (\ref{eq:sup_k Lambdao + max Gii-mN}), (\ref{eq:1(Xi_k-1)Theta(z_k) bound})
			and the fact that $\delta_{k}<\sqrt{\delta_{k}}$ for all $k=1,\dots,N^{5}-N^{4+\tau}$,
			we get that as $N\ge N_{0}$,
			\begin{multline*}
			\mathbb{P}(\Xi_{k-1}\cap\Xi_{k}^{c})\\
			\le\mathbb{P}\Big(\mathbf{1}(\Xi_{k-1})\big(\max_{i\in{\cal I}}|G_{ii}(z_{k})-m_{N}(z_{k})|+\Theta(z_{k})+\Lambda_{o}(z_{k})\big)>N^{\varepsilon}\sqrt{\delta_{k}}\Big)\\
			\le\mathbb{P}\Big(\mathbf{1}(\Xi_{k-1})\big(\max_{i\in{\cal I}}|G_{ii}(z_{k})-m_{N}(z_{k})|+\Lambda_{o}(z_{k})>\frac{1}{2}N^{\varepsilon}\sqrt{\delta_{k}}\Big)\\
			+\mathbb{P}\Big(\mathbf{1}(\Xi_{k-1})\Theta(z_{k})>\frac{1}{2}N^{\varepsilon}\sqrt{\delta_{k}}\Big)\le2N^{-D}.
			\end{multline*}
			Then we see that for any $k\in\{1,\dots,N^{5}-N^{4+\tau}\}$, as $N\ge N_{0}$,
			\[
			\mathbb{P}(\Xi_{k}^{c})=1-\mathbb{P}(\Xi_{k})=\sum_{i=1}^{k}\mathbb{P}(\Xi_{i-1}\cap\Xi_{i}^{c})+\mathbb{P}(\Xi_{0}^{c})\le2N^{5-D}.
			\]
			This shows that $1\prec\mathbf{1}(\Xi_{k})$ or equivalently $\Lambda(z_{k})\prec\sqrt{\delta_{k}}$ uniformly for all $k\in\{0,\dots,N^{5}-N^{4+\tau}\}$.
			
			Finally, by choosing $\hat{k}\in\{1,\dots,N^{5}-N^{4+\tau}\}$ such
			that $-\imath(z-z_{\hat{k}})\le N^{-5}$, we have
			\begin{multline*}
			\Lambda(z)\le|\Lambda(z)-\Lambda(z_{\hat{k}})|+\Lambda(z_{\hat{k}})\le\max_{i,j}|G_{ij}(z)-G_{ij}(z_{\hat{k}})|+\Lambda(z_{\hat{k}})\\
			\le N^{-3-2\tau}+\Lambda(z_{\hat{k}})\prec(N\eta)^{-1/4}+q.
			\end{multline*}
			The proof of Proposition \ref{prop:deformed weak local law} is now complete.
		\end{proof}

		\subsection{Proof of Proposition \ref{prop:fluctuation averaging}.}\label{Suppl:FA}
		\begin{proof}
			We omit the proof of (\ref{eq:fluctuation of Z}), since it is similar
			to that of (\ref{eq:fluctuation of V}) (actually it is also simpler
			than (\ref{eq:fluctuation of V}) since we only need to expand the
			$G$ terms using the third identity of Lemma \ref{lem:resolvent}).
			In the following, we give the proof of (\ref{eq:fluctuation of V}).
			
			For simplicity of notation, denote $\Sigma_{0}=\Sigma(m_N^{(i)}\Sigma+I)^{-1}$
			and write $\mathscr{V}_{i}=\mathbf{x}_{i}^{*}{\cal G}^{(i)}\Sigma_{0}\mathbf{x}_{i}$.
			In the following, we bound the quantity
			\[
			|\frac{1}{N}\sum_{i=1}^{N}Q_{i}\mathscr{V}_{i}|
			\]
			Let $p$ be an even integer. Denote $V_{i_{s}}:=Q_{i_{s}}\mathscr{V}_{i_{s}}$
			for $s\le p/2$ and $V_{i_{s}}:=Q_{i_{s}}\mathscr{V}_{i_{s}}^{*}$
			for $s>p/2$. We bound $\mathbb{E}\big|\frac{1}{N}\sum_{i_{s}=1}^{N}V_{i}\big|^{p}$.
			
			We see that
			\[
			\mathbb{E}\Big|\frac{1}{N}\sum_{i=1}^{N}Q_{i}\mathscr{V}_{i}\Big|^{p}=\frac{1}{N^{p}}\sum_{i_{1},\dots,i_{p}}\mathbb{E}\prod_{s=1}^{p}V_{i_{s}}=\frac{1}{N^{p}}\sum_{i_{1},\dots,i_{p}}\mathbb{E}\prod_{s=1}^{p}\Big(\prod_{r=1}^{p}(P_{i_{r}}+Q_{i_{r}})V_{i_{s}}\Big).
			\]
			
			Introducing the notation $\mathbf{i}=(i_{1},\dots,i_{p})$, $[\mathbf{i}]=\{i_{1},\dots,i_{p}\}$,
			$P_{A}=\prod_{i\in A}P_{i}$ and $Q_{A}=\prod_{i\in A}Q_{i}$ for
			some index set $A$, we get
			\begin{equation}
			\mathbb{E}\Big|\frac{1}{N}\sum_{i=1}^{N}Q_{i}\mathscr{V}_{i}\Big|^{p}=\frac{1}{N^{p}}\sum_{\mathbf{i}}\sum_{A_{1},\dots,A_{p}\subset[\mathbf{i}]}\mathbb{E}\prod_{s=1}^{p}\Big(P_{A_{s}^{c}}Q_{A_{s}}V_{i_{s}}\Big).\label{eq:p power}
			\end{equation}
			
			By definition of $V_{i}$, we have that $V_{i_{s}}=Q_{i_{s}}V_{i_{s}}$
			and $P_{i_{s}}V_{i_{s}}=0$, which imply that $P_{A_{s}^{c}}V_{i_{s}}=0$
			if $i_{s}\notin A_{s}$. Hence we may restrict the summation to $A_{s}$
			satisfying
			\begin{equation}
			i_{s}\in A_{s}\label{eq:mu in A}
			\end{equation}
			for all $s$. Moreover, we see that if $i_{s}\in\cap_{q\ne s}A_{q}^{c}$
			for some $s$, say $s=1$, then $P_{A_{q}^{c}}Q_{A_{q}}V_{i_{q}}$
			is $X^{(s)}$-measurable for each $q=2,\dots,p$. Thus, we have
			\begin{multline}
			\mathbb{E}\prod_{s=1}^{p}(P_{A_{s}^{c}}Q_{A_{s}}V_{i_{s}})=\mathbb{E}(P_{A_{1}^{c}}Q_{A_{1}}Q_{i_{1}}V_{i_{1}})\prod_{s=2}^{p}(P_{A_{s}^{c}}Q_{A_{s}}V_{i_{s}})\\
			=\mathbb{E}Q_{i_{1}}\Big\{(P_{A_{1}^{c}}Q_{A_{1}}V_{i_{1}})\prod_{s=2}^{p}(P_{A_{s}^{c}}Q_{A_{s}}V_{i_{s}})\Big\}=0.\label{eq:mu in bigcup Aq}
			\end{multline}
			
			(\ref{eq:mu in A}) and (\ref{eq:mu in bigcup Aq}) show that each
			index $i_{s}$ must belong to at least two different sets: $A_{s}$
			and $A_{q}$ for some $q\ne s$. Hence
			\begin{equation}
			\sum_{s=1}^{p}|A_{s}|\ge2|[\mathbf{i}]|.\label{eq:sum A ge 2mu}
			\end{equation}
			
			In the following, a crucial step is to show that for $i\in A$
			\begin{equation}
			|Q_{A}V_{i}|\prec\Phi_{\nu}^{|A|}.\label{eq:QW}
			\end{equation}
			
			When $|A|=1$ (corresponding to the case $A=\{i\}$), it follows straightforward
			from Lemma \ref{lem:large deviation}. Suppose $|A|\ge2$. For ease
			of presentation, we assume without loss of generality that $i=1$
			and $A=\{1,2,\dots,\nu\}$ for some $\nu\ge2$. Before we proceed, we note the
			following equality that for any $i\ne j$ and $T\subset{\cal I}$
			with $i,j\notin T$,
			\begin{eqnarray}
			\mathbf{x}_{i}^{*}{\cal G}^{(iT)} & = & \mathbf{x}_{i}^{*}({\cal G}^{(ijT)}-\frac{{\cal G}^{(ijT)}\mathbf{x}_{j}\mathbf{x}_{j}^{*}{\cal G}^{(ijT)}}{1+\mathbf{x}_{j}^{*}{\cal G}^{(ijT)}\mathbf{x}_{j}})\nonumber \\
			& = & \mathbf{x}_{i}^{*}{\cal G}^{(ijT)}+zG_{jj}^{(iT)}\mathbf{x}_{i}{\cal G}^{(ijT)}\mathbf{x}_{j}\mathbf{x}_{j}^{*}{\cal G}^{(ijT)}\nonumber \\
			& = & \mathbf{x}_{i}^{*}{\cal G}^{(ijT)}+\frac{G_{ij}^{(T)}}{G_{ii}^{(T)}}\mathbf{x}_{j}^{*}{\cal G}^{(ijT)},\label{eq:xG expansion}
			\end{eqnarray}
			and
			\begin{equation}
			\Sigma_{0}^{(T)}=\Sigma_{0}^{(iT)}+\frac{1}{N}\sum_{j\in{\cal I}\backslash(\{i\}\cup T)}\frac{G_{ji}^{(T)}G_{ij}^{(T)}}{G_{ii}^{(T)}}\Sigma_{0}^{(iT)}\Sigma\Sigma_{0}^{(T)}.\label{eq:resolvent of Sigma0}
			\end{equation}
			
			We show an example of the expansion. It follows from
			Lemma \ref{lem:Sherman-Morrison formula} and Lemma \ref{lem:resolvent}
			that
			\begin{eqnarray*}
				Q_{2}\mathscr{V}_{1} & = & Q_{2}(\mathbf{x}_{1}^{*}{\cal G}^{(1)}\Sigma_{0}^{(1)}\mathbf{x}_{1})\\
				& = & Q_{2}(\mathbf{x}_{1}^{*}{\cal G}^{(1)}\Sigma_{0}^{(12)}\mathbf{x}_{1}+\Big(\frac{1}{N}\sum_{j\notin\{1,2\}}\frac{G_{j2}^{(1)}G_{2j}^{(1)}}{G_{22}^{(1)}}\Big)\mathbf{x}_{1}^{*}{\cal G}^{(1)}\Sigma_{0}^{(12)}\Sigma\Sigma_{0}^{(1)}\mathbf{x}_{1})\\
				& = & Q_{2}(\mathbf{x}_{1}^{*}{\cal G}^{(12)}\Sigma_{0}^{(12)}\mathbf{x}_{1})+Q_{2}(\frac{G_{12}}{G_{11}}\mathbf{x}_{2}^{*}{\cal G}^{(12)}\Sigma_{0}^{(12)}\mathbf{x}_{1})\\
				&  & +Q_{2}(\Big(\frac{1}{N}\sum_{j\notin\{1,2\}}\frac{G_{j2}^{(1)}G_{2j}^{(1)}}{G_{22}^{(1)}}\Big)\mathbf{x}_{1}^{*}{\cal G}^{(1)}\Sigma_{0}^{(12)}\Sigma\Sigma_{0}^{(1)}\mathbf{x}_{1})\\
				& = & Q_{2}(\frac{G_{12}}{G_{11}}\mathbf{x}_{2}^{*}{\cal G}^{(12)}\Sigma_{0}^{(12)}\mathbf{x}_{1})+Q_{2}(\Big(\frac{1}{N}\sum_{j\notin\{1,2\}}\frac{G_{j2}^{(1)}G_{2j}^{(1)}}{G_{22}^{(1)}}\Big)\mathbf{x}_{1}^{*}{\cal G}^{(1)}\Sigma_{0}^{(12)}\Sigma\Sigma_{0}^{(1)}\mathbf{x}_{1}).
			\end{eqnarray*}
			
			We note that
			\[
			\Big|\frac{G_{12}}{G_{11}}\mathbf{x}_{2}^{*}{\cal G}^{(12)}\Sigma_{0}^{(12)}\mathbf{x}_{1}\Big|\prec\frac{\Phi_{\nu}}{M}\|{\cal G}^{(12)}\Sigma_{0}^{(12)}\|_{F}\le\frac{\Phi_{\nu}}{M}\|{\cal G}^{(12)}\|_{F}\|\Sigma_{0}^{(12)}\|\prec\Phi_{\nu}^{2},
			\]
			
			and
			\begin{eqnarray*}
				&  & Q_{1}(\Big(\frac{1}{N}\sum_{j\notin\{1,2\}}\frac{G_{j2}^{(1)}G_{2j}^{(1)}}{G_{22}^{(1)}}\Big)\mathbf{x}_{1}^{*}{\cal G}^{(1)}\Sigma_{0}^{(12)}\Sigma\Sigma_{0}^{(1)}\mathbf{x}_{1})\\
				& = & \Big(\frac{1}{N}\sum_{j\notin\{1,2\}}\frac{G_{j2}^{(1)}G_{2j}^{(1)}}{G_{22}^{(1)}}\Big)Q_{1}(\mathbf{x}_{1}^{*}{\cal G}^{(1)}\Sigma_{0}^{(12)}\Sigma\Sigma_{0}^{(1)}\mathbf{x}_{1})\\
				& \prec & \frac{\Phi_{\nu}^{2}}{M}\|{\cal G}^{(1)}\Sigma_{0}^{(12)}\Sigma\Sigma_{0}^{(1)}\|_{F}\\
				& \prec & \Phi_{\nu}^{3}.
			\end{eqnarray*}
			
			We see that
			
			\begin{eqnarray}
			&  & Q_{3}Q_{2}\mathscr{V}_{1}\nonumber \\
			& = & Q_{3}Q_{2}(\mathbf{x}_{1}^{*}{\cal G}^{(1)}\Sigma_{0}^{(1)}\mathbf{x}_{1})\nonumber \\
			& = & Q_{3}Q_{2}(\mathbf{x}_{1}^{*}{\cal G}^{(1)}\Sigma_{0}^{(12)}\mathbf{x}_{1}+\Big(\frac{1}{N}\sum_{j\notin\{1,2\}}\frac{G_{j2}^{(1)}G_{2j}^{(1)}}{G_{22}^{(1)}}\Big)\mathbf{x}_{1}^{*}{\cal G}^{(1)}\Sigma_{0}^{(12)}\Sigma\Sigma_{0}^{(1)}\mathbf{x}_{1})\nonumber \\
			& = & Q_{3}Q_{2}(\mathbf{x}_{1}^{*}{\cal G}^{(1)}\Sigma_{0}^{(123)}\mathbf{x}_{1}+\Big(\frac{1}{N}\sum_{j\notin\{1,2,3\}}\frac{G_{j3}^{(12)}G_{3j}^{(12)}}{G_{33}^{(12)}}\Big)\mathbf{x}_{1}^{*}{\cal G}^{(1)}\Sigma_{0}^{(123)}\Sigma\Sigma_{0}^{(12)}\mathbf{x}_{1}\nonumber \\
			&  & +\Big(\frac{1}{N}\sum_{j\notin\{1,2\}}\frac{G_{j2}^{(1)}G_{2j}^{(1)}}{G_{22}^{(1)}}\Big)\mathbf{x}_{1}^{*}{\cal G}^{(1)}\Sigma_{0}^{(12)}\Sigma\Sigma_{0}^{(1)}\mathbf{x}_{1})\nonumber \\
			& = & Q_{3}Q_{2}(\mathbf{x}_{1}^{*}{\cal G}^{(12)}\Sigma_{0}^{(123)}\mathbf{x}_{1}+\frac{G_{12}}{G_{11}}\mathbf{x}_{2}^{*}{\cal G}^{(12)}\Sigma_{0}^{(123)}\mathbf{x}_{1}\nonumber \\
			&  & +\Big(\frac{1}{N}\sum_{j\notin\{1,2,3\}}\frac{G_{j3}^{(12)}G_{3j}^{(12)}}{G_{33}^{(12)}}\Big)\mathbf{x}_{1}^{*}{\cal G}^{(1)}\Sigma_{0}^{(123)}\Sigma\Sigma_{0}^{(12)}\mathbf{x}_{1}\nonumber \\
			&  & +\Big(\frac{1}{N}\sum_{j\notin\{1,2\}}\frac{G_{j2}^{(1)}G_{2j}^{(1)}}{G_{22}^{(1)}}\Big)\mathbf{x}_{1}^{*}{\cal G}^{(1)}\Sigma_{0}^{(12)}\Sigma\Sigma_{0}^{(1)}\mathbf{x}_{1})\nonumber \\
			& = & Q_{3}Q_{2}(\frac{G_{12}}{G_{11}}\mathbf{x}_{2}^{*}{\cal G}^{(123)}\Sigma_{0}^{(123)}\mathbf{x}_{1}+\frac{G_{12}G_{23}^{(1)}}{G_{11}G_{22}^{(1)}}\mathbf{x}_{3}^{*}{\cal G}^{(123)}\Sigma_{0}^{(123)}\mathbf{x}_{1}\nonumber \\
			&  & +\Big(\frac{1}{N}\sum_{j\notin\{1,2,3\}}\frac{G_{j3}^{(12)}G_{3j}^{(12)}}{G_{33}^{(12)}}\Big)\mathbf{x}_{1}^{*}{\cal G}^{(1)}\Sigma_{0}^{(123)}\Sigma\Sigma_{0}^{(12)}\mathbf{x}_{1}\nonumber \\
			&  & +\Big(\frac{1}{N}\sum_{j\notin\{1,2\}}\frac{G_{j2}^{(1)}G_{2j}^{(1)}}{G_{22}^{(1)}}\Big)\mathbf{x}_{1}^{*}{\cal G}^{(1)}\Sigma_{0}^{(12)}\Sigma\Sigma_{0}^{(1)}\mathbf{x}_{1})\nonumber \\
			& = & Q_{3}Q_{2}(\Big(\frac{G_{12}^{(3)}}{G_{11}^{(3)}}+\frac{G_{13}G_{32}}{G_{33}G_{11}^{(3)}}-\frac{G_{12}^{(3)}G_{13}G_{31}}{G_{11}G_{11}^{(3)}G_{33}}-\frac{G_{13}^{2}G_{32}G_{31}}{G_{11}G_{11}^{(3)}G_{33}^{2}}\Big)\mathbf{x}_{2}^{*}{\cal G}^{(123)}\Sigma_{0}^{(123)}\mathbf{x}_{1}\nonumber \\
			&  & +\frac{G_{12}G_{23}^{(1)}}{G_{11}G_{22}^{(1)}}\mathbf{x}_{3}^{*}{\cal G}^{(123)}\Sigma_{0}^{(123)}\mathbf{x}_{1}\nonumber \\
			&  & +\Big(\frac{1}{N}\sum_{j\notin\{1,2,3\}}\frac{G_{j3}^{(12)}G_{3j}^{(12)}}{G_{33}^{(12)}}\Big)\mathbf{x}_{1}^{*}{\cal G}^{(1)}\Sigma_{0}^{(123)}\Sigma\Sigma_{0}^{(12)}\mathbf{x}_{1}\nonumber \\
			&  & +\Big(\frac{1}{N}\sum_{j\notin\{1,2\}}\frac{G_{j2}^{(1)}G_{2j}^{(1)}}{G_{22}^{(1)}}\Big)\mathbf{x}_{1}^{*}{\cal G}^{(1)}\Sigma_{0}^{(12)}\Sigma\Sigma_{0}^{(1)}\mathbf{x}_{1})\nonumber \\
			& = & Q_{3}Q_{2}(\Big(\frac{G_{13}G_{32}}{G_{33}G_{11}^{(3)}}-\frac{G_{12}^{(3)}G_{13}G_{31}}{G_{11}G_{11}^{(3)}G_{33}}-\frac{G_{13}^{2}G_{32}G_{31}}{G_{11}G_{11}^{(3)}G_{33}^{2}}\Big)\mathbf{x}_{2}^{*}{\cal G}^{(123)}\Sigma_{0}^{(123)}\mathbf{x}_{1}\nonumber \\
			&  & +\frac{G_{12}G_{23}^{(1)}}{G_{11}G_{22}^{(1)}}\mathbf{x}_{3}^{*}{\cal G}^{(123)}\Sigma_{0}^{(123)}\mathbf{x}_{1}\nonumber \\
			&  & +\Big(\frac{1}{N}\sum_{j\notin\{1,2,3\}}\frac{G_{j3}^{(12)}G_{3j}^{(12)}}{G_{33}^{(12)}}\Big)\mathbf{x}_{1}^{*}{\cal G}^{(1)}\Sigma_{0}^{(123)}\Sigma\Sigma_{0}^{(12)}\mathbf{x}_{1}\nonumber \\
			&  & +\Big(\frac{1}{N}\sum_{j\notin\{1,2\}}\frac{G_{j2}^{(1)}G_{2j}^{(1)}}{G_{22}^{(1)}}\Big)\mathbf{x}_{1}^{*}{\cal G}^{(1)}\Sigma_{0}^{(12)}\Sigma\Sigma_{0}^{(1)}\mathbf{x}_{1}).\label{eq:phi3}
			\end{eqnarray}
			
			We observe that the first two terms in (\ref{eq:phi3}) are $\prec\Phi_{\nu}^{3}$, and we continue the expansion procedures for $\mathcal{G}^{(T)}$ and $\Sigma_0^{(T)}$ for which $(T)$ is not maximally expanded. Thus, the third term can be written as
			\[
			\begin{split}
			&Q_3Q_2(\Big(\frac{1}{N}\sum_{j\notin\{1,2,3\}}\frac{G_{j3}^{(12)}G_{3j}^{(12)}}{G_{33}^{(12)}}\Big)(\mathbf{x}_1^{*}\mathcal{G}^{(12)}+\frac{G_{12}}{G_{11}}\mathbf{x}_2^{*}\mathcal{G}^{(12)})\Sigma_{0}^{(123)}\Sigma\Sigma_{0}^{(12)}\mathbf{x}_{1})\\
			&=Q_3Q_2(\Big(\frac{1}{N}\sum_{j\notin\{1,2,3\}}\frac{G_{j3}^{(12)}G_{3j}^{(12)}}{G_{33}^{(12)}}\Big)\frac{G_{12}}{G_{11}}\mathbf{x}_2^{*}\mathcal{G}^{(12)})\Sigma_{0}^{(123)}\Sigma\Sigma_{0}^{(12)}\mathbf{x}_{1})\prec\Phi_{\nu}^3.
			\end{split}
			\]
			Similar results can be obtained for the fourth term. Actually, one can observe that the first term of $Q_AV_i$ is the leading term, so we can only clarify the bounds for the first term.
			
			We see that
			\begin{eqnarray*}
				&  & Q_{4}Q_{3}Q_{2}Q_{1}\mathscr{V}_{1}\\
				& = & Q_{1}Q_{4}Q_{3}Q_{2}(\mathbf{x}_{1}^{*}{\cal G}^{(1)}\Sigma_{0}^{(123)}\mathbf{x}_{1}+\Big(\frac{1}{N}\sum_{j\notin\{1,2,3\}}\frac{G_{j3}^{(12)}G_{3j}^{(12)}}{G_{33}^{(12)}}\Big)\mathbf{x}_{1}^{*}{\cal G}^{(1)}\Sigma_{0}^{(123)}\Sigma\Sigma_{0}^{(12)}\mathbf{x}_{1}\\
				&  & +\Big(\frac{1}{N}\sum_{j\notin\{1,2\}}\frac{G_{j2}^{(1)}G_{2j}^{(1)}}{G_{22}^{(1)}}\Big)\mathbf{x}_{1}^{*}{\cal G}^{(1)}\Sigma_{0}^{(12)}\Sigma\Sigma_{0}^{(1)}\mathbf{x}_{1})\\
				& = & Q_{1}Q_{4}Q_{3}Q_{2}(\mathbf{x}_{1}^{*}{\cal G}^{(1)}\Sigma_{0}^{(1234)}\mathbf{x}_{1}\\
				&  & +\Big(\frac{1}{N}\sum_{j\notin\{1,2,3,4\}}\frac{G_{j4}^{(123)}G_{4j}^{(123)}}{G_{44}^{(123)}}\Big)\mathbf{x}_{1}^{*}{\cal G}^{(1)}\Sigma_{0}^{(1234)}\Sigma\Sigma_{0}^{(123)}\mathbf{x}_{1}\\
				&  & +\Big(\frac{1}{N}\sum_{j\notin\{1,2,3\}}\frac{G_{j3}^{(12)}G_{3j}^{(12)}}{G_{33}^{(12)}}\Big)\mathbf{x}_{1}^{*}{\cal G}^{(1)}\Sigma_{0}^{(123)}\Sigma\Sigma_{0}^{(12)}\mathbf{x}_{1}\\
				&  & +\Big(\frac{1}{N}\sum_{j\notin\{1,2\}}\frac{G_{j2}^{(1)}G_{2j}^{(1)}}{G_{22}^{(1)}}\Big)\mathbf{x}_{1}^{*}{\cal G}^{(1)}\Sigma_{0}^{(12)}\Sigma\Sigma_{0}^{(1)}\mathbf{x}_{1})\\
				& = & Q_{1}Q_{4}Q_{3}Q_{2}(\mathbf{x}_{1}^{*}{\cal G}^{(12)}\Sigma_{0}^{(1234)}\mathbf{x}_{1}+\frac{G_{12}}{G_{11}}\mathbf{x}_{2}^{*}{\cal G}^{(12)}\Sigma_{0}^{(1234)}\mathbf{x}_{1}),
			\end{eqnarray*}
			whose first term can be expanded as
			\begin{eqnarray*}
				&  & Q_{1}Q_{4}Q_{3}Q_{2}(\mathbf{x}_{1}^{*}{\cal G}^{(1)}\Sigma_{0}^{(1234)}\mathbf{x}_{1})\\
				& = & Q_{1}Q_{4}Q_{3}Q_{2}(\mathbf{x}_{1}^{*}{\cal G}^{(12)}\Sigma_{0}^{(1234)}\mathbf{x}_{1}+\frac{G_{12}}{G_{11}}\mathbf{x}_{2}^{*}{\cal G}^{(12)}\Sigma_{0}^{(1234)}\mathbf{x}_{1})\\
				& = & Q_{1}Q_{4}Q_{3}Q_{2}(\frac{G_{12}}{G_{11}}\mathbf{x}_{2}^{*}{\cal G}^{(123)}\Sigma_{0}^{(1234)}\mathbf{x}_{1}+\frac{G_{12}G_{23}^{(1)}}{G_{11}G_{22}^{(1)}}\mathbf{x}_{3}^{*}{\cal G}^{(123)}\Sigma_{0}^{(1234)}\mathbf{x}_{1})\\
				& = & Q_{1}Q_{4}Q_{3}Q_{2}(\Big(\frac{G_{12}^{(3)}}{G_{11}^{(3)}}+\frac{G_{13}G_{32}}{G_{33}G_{11}^{(3)}}-\frac{G_{12}^{(3)}G_{13}G_{31}}{G_{11}G_{11}^{(3)}G_{33}}-\frac{G_{13}^{2}G_{32}G_{31}}{G_{11}G_{11}^{(3)}G_{33}^{2}}\Big)\mathbf{x}_{2}^{*}{\cal G}^{(123)}\Sigma_{0}^{(1234)}\mathbf{x}_{1}\\
				&  & +\frac{G_{12}G_{23}^{(1)}}{G_{11}G_{22}^{(1)}}\mathbf{x}_{3}^{*}{\cal G}^{(1234)}\Sigma_{0}^{(1234)}\mathbf{x}_{1}+\frac{G_{12}G_{23}^{(1)}G_{34}^{(12)}}{G_{11}G_{22}^{(1)}G_{33}^{(12)}}\mathbf{x}_{4}^{*}{\cal G}^{(1234)}\Sigma_{0}^{(1234)}\mathbf{x}_{1})\\
				& = & Q_{1}Q_{4}Q_{3}Q_{2}(\Big(\frac{G_{13}G_{32}}{G_{33}G_{11}^{(3)}}-\frac{G_{12}^{(3)}G_{13}G_{31}}{G_{11}G_{11}^{(3)}G_{33}}-\frac{G_{13}^{2}G_{32}G_{31}}{G_{11}G_{11}^{(3)}G_{33}^{2}}\Big)\mathbf{x}_{2}^{*}{\cal G}^{(123)}\Sigma_{0}^{(1234)}\mathbf{x}_{1}\\
				&  & +\frac{G_{12}G_{23}^{(1)}}{G_{11}G_{22}^{(1)}}\mathbf{x}_{3}^{*}{\cal G}^{(1234)}\Sigma_{0}^{(1234)}\mathbf{x}_{1}+\frac{G_{12}G_{23}^{(1)}G_{34}^{(12)}}{G_{11}G_{22}^{(1)}G_{33}^{(12)}}\mathbf{x}_{4}^{*}{\cal G}^{(1234)}\Sigma_{0}^{(1234)}\mathbf{x}_{1})\\
				& = & Q_{1}Q_{4}Q_{3}Q_{2}(\Big(\frac{G_{13}G_{32}}{G_{33}G_{11}^{(3)}}-\frac{G_{12}^{(3)}G_{13}G_{31}}{G_{11}G_{11}^{(3)}G_{33}}-\frac{G_{13}^{2}G_{32}G_{31}}{G_{11}G_{11}^{(3)}G_{33}^{2}}\Big)\mathbf{x}_{2}^{*}{\cal G}^{(1234)}\Sigma_{0}^{(1234)}\mathbf{x}_{1}\\
				&  & +\Big(\frac{G_{13}G_{32}}{G_{33}G_{11}^{(3)}}-\frac{G_{12}^{(3)}G_{13}G_{31}}{G_{11}G_{11}^{(3)}G_{33}}-\frac{G_{13}^{2}G_{32}G_{31}}{G_{11}G_{11}^{(3)}G_{33}^{2}}\Big)\frac{G_{24}^{(13)}}{G_{44}^{(13)}}\mathbf{x}_{4}^{*}{\cal G}^{(1234)}\Sigma_{0}^{(1234)}\mathbf{x}_{1}\\
				&  & +\frac{G_{12}G_{23}^{(1)}}{G_{11}G_{22}^{(1)}}\mathbf{x}_{3}^{*}{\cal G}^{(1234)}\Sigma_{0}^{(1234)}\mathbf{x}_{1}+\frac{G_{12}G_{23}^{(1)}G_{34}^{(12)}}{G_{11}G_{22}^{(1)}G_{33}^{(12)}}\mathbf{x}_{4}^{*}{\cal G}^{(1234)}\Sigma_{0}^{(1234)}\mathbf{x}_{1}).
			\end{eqnarray*}
			We see that the term
			\[			Q_{1}Q_{4}Q_{3}Q_{2}(\Big(\frac{G_{13}G_{32}}{G_{33}G_{11}^{(3)}}-\frac{G_{12}^{(3)}G_{13}G_{31}}{G_{11}G_{11}^{(3)}G_{33}}-\frac{G_{13}^{2}G_{32}G_{31}}{G_{11}G_{11}^{(3)}G_{33}^{2}}\Big)\mathbf{x}_{2}^{*}{\cal G}^{(1234)}\Sigma_{0}^{(1234)}\mathbf{x}_{1})
			\]
			can be bounded by carrying out the following expansion
			\[
			\frac{G_{13}G_{32}}{G_{33}G_{11}^{(3)}}=\Big(G_{13}^{(4)}+\frac{G_{14}G_{43}}{G_{44}}\Big)\Big(G_{32}^{(4)}+\frac{G_{34}G_{42}}{G_{44}}\Big)\Big(\frac{1}{G_{33}^{(4)}}-\frac{G_{34}G_{43}}{G_{33}G_{33}^{(4)}G_{44}}\Big)\Big(\frac{1}{G_{11}^{(34)}}-\frac{G_{14}^{(3)}G_{41}^{(3)}}{G_{11}^{(3)}G_{11}^{(34)}G_{44}^{(3)}}\Big).
			\]
			
			Thus the first term resulting from the expansion yields that
			\[
			Q_{1}Q_{4}Q_{3}Q_{2}(\frac{G_{13}^{(4)}G_{32}^{(4)}}{G_{33}^{(4)}G_{11}^{(34)}}\mathbf{x}_{2}^{*}{\cal G}^{(1234)}\Sigma_{0}^{(1234)}\mathbf{x}_{1})=0,
			\]
			and the remaining terms all contain three off-diagonal $G$ terms
			in the numerators.
			
			Consequently,
			\[
			Q_{1}Q_{4}Q_{3}Q_{2}(\Big(\frac{G_{13}G_{32}}{G_{33}G_{11}^{(3)}}-\frac{G_{12}^{(3)}G_{13}G_{31}}{G_{11}G_{11}^{(3)}G_{33}}-\frac{G_{13}^{2}G_{32}G_{31}}{G_{11}G_{11}^{(3)}G_{33}^{2}}\Big)\mathbf{x}_{2}^{*}{\cal G}^{(1234)}\Sigma_{0}^{(1234)}\mathbf{x}_{1})\prec\Phi_{\nu}^{4}.
			\]
			
			Similarly, the term
			\[
			Q_{1}Q_{4}Q_{3}Q_{2}(\frac{G_{12}G_{23}^{(1)}}{G_{11}G_{22}^{(1)}}\mathbf{x}_{3}^{*}{\cal G}^{(1234)}\Sigma_{0}^{(1234)}\mathbf{x}_{1})
			\]
			can be bounded by $\Phi_{\nu}^{4}$ via expanding the $G$ terms into
			those with $(4)$ added to the superscripts.
			
			So the remaining term  is
			\[
			Q_{1}Q_{4}Q_{3}Q_{2}(\Big(\frac{1}{N}\sum_{j\notin\{1,2\}}\frac{G_{j2}^{(1)}G_{2j}^{(1)}}{G_{22}^{(1)}}\Big)\mathbf{x}_{1}^{*}{\cal G}^{(12)}\Sigma_{0}^{(12)}\Sigma\Sigma_{0}^{(12)}\mathbf{x}_{1}).
			\]
			
			We note that
			\begin{eqnarray*}
				&  & Q_{1}Q_{4}Q_{3}Q_{2}(\Big(\frac{1}{N}\sum_{j\notin\{1,2\}}\frac{G_{j2}^{(1)}G_{2j}^{(1)}}{G_{22}^{(1)}}\Big)\mathbf{x}_{1}^{*}{\cal G}^{(12)}\Sigma_{0}^{(12)}\Sigma\Sigma_{0}^{(12)}\mathbf{x}_{1})\\
				& = & Q_{1}Q_{4}Q_{3}Q_{2}(\Big(\frac{1}{N}\sum_{j\notin\{1,2\}}\frac{G_{j2}^{(1)}G_{2j}^{(1)}}{G_{22}^{(1)}}\Big)\mathbf{x}_{1}^{*}{\cal G}^{(123)}\Sigma_{0}^{(123)}\Sigma\Sigma_{0}^{(123)}\mathbf{x}_{1})+O_{\prec}(\Phi_{\nu}^{4})\\
				& = & Q_{1}Q_{4}Q_{3}Q_{2}(\Big(\frac{1}{N}\sum_{j\notin\{1,2,3\}}\frac{G_{j2}^{(13)}G_{2j}^{(13)}}{G_{22}^{(13)}}+\mathcal{T}\Big)\mathbf{x}_{1}^{*}{\cal G}^{(123)}\Sigma_{0}^{(123)}\Sigma\Sigma_{0}^{(123)}\mathbf{x}_{1})+O_{\prec}(\Phi_{\nu}^{4})\\
				& \prec & \Phi_{\nu}^{4},
			\end{eqnarray*}
			where $\mathcal{T}$ is the term with at least three off-diagonal $G$ terms
			in the numerator which can be bounded by $\Phi_{\nu}^{4}$ via expanding the $G$ terms.
			
			Now we summarise the expansion steps as follows. Consider
			\[
			Q_A\mathscr{V}_{1}:=Q_{\nu}\cdots Q_{2}Q_{1}\mathscr{V}_{1}=Q_{\nu}\cdots Q_{2}Q_{1}\mathbf{x}_1^{*}\mathcal{G}^{(1)}\Sigma_0^{(1)}\mathbf{x}_1
			\]
			\begin{itemize}
				\item [a)] We first expand the term $\mathbf{x}_1^{*}\mathcal{G}^{(1)}$ to $\mathbf{x}_i^{*}\mathcal{G}^{(A)}$ from the smallest index to the largest index. We will get a sequence of monomials with the maximally expanded $\mathcal{A}_i:=\mathbf{x}_i^{*}\mathcal{G}^{(A)}$ where $i\in A$. The coefficients of $\mathcal{A}_i$'s are of the pattern
				\[
				\frac{\prod_{\{a,b\}\in P_i}G_{ab}^{(T_{ab})}}{\prod_{\{a,b\}\in P_i}G_{aa}^{(T_{ab})}},
				\]
				where $P_i$ is an ordered paired pattern subset of $\{1,\cdots,i\}\subset A$, for example, $P_i=\{\{1,3\},\{3,4\},\{4,i\}\}$ and $T_{ab}=\{1,\cdots,b\}\setminus\{a,b\}$ or $T_{ab}=\emptyset$. And these coefficients can be handled by the resolvent extension just like the same procedures in $(1/N)\sum_{i=1}Q_i1/G_{ii}$. One can observe that the only remaining terms after taking $Q_A$ (we imprecisely ignore the term $\Sigma_0$ here) are those monomials whose lower indexes in the numerator contain all elements of $A$. Since there are only off-diagonal entries in the numerator,  by the paired pattern $P_i$ we remark that for those monomials, the number of off-diagonal entries is $|A|$.
				\item [b)] We expand $\Sigma_0$ to match the upper index with $\mathcal{G}^{(T)}$ which ensures that several undesired terms vanish after taking conditional expectation. Actually, we use the same strategy as in step a), adding the upper index of $\Sigma_0$ from the smallest value to the largest value of $A$. We remark that except the leading term with coefficient $1$, other terms give us more off-diagonal entries as coefficients.
				\item [c)] We expand the Green function of coefficients after steps a) and b)  to a maximal extent, or have at least $|A|$ off-diagonal entries.
			\end{itemize}

			Then it follows that for $\nu\geqslant2$
			\[
			Q_AV_i=Q_{\nu}\cdots Q_{2}Q_{1}\mathscr{V}_{i}\prec\Phi_{\nu}^{\nu+1}.
			\]
			
			This completes the proof of (\ref{eq:QW}). By (\ref{eq:sum A ge 2mu})
			and (\ref{eq:QW}), it follows from (\ref{eq:p power}) that there
			exists some constant $C_{p}>0$ depending on $p$ only such that
			\begin{multline}
			\mathbb{E}\Big|\frac{1}{N}\sum_{i=1}^{N}Q_{i}\mathscr{V}_{i}\Big|^{p}=\frac{1}{N^{p}}\sum_{\mathbf{i}}\sum_{A_{1},\dots,A_{p}\subset[\mathbf{i}]}\mathbb{E}\prod_{s=1}^{p}\Big(P_{A_{s}^{c}}Q_{A_{s}}V_{i_{s}}\Big)\\
			\prec\frac{C_{p}}{N^{p}}\sum_{\mathbf{i}}\Phi_{\nu}^{2|[\boldsymbol{i}]|}=\frac{C_{p}}{N^{p}}\sum_{s=1}^{p}\Phi_{\nu}^{2s}\sum_{\mathbf{i}}\mathbf{1}(|\mathbf{i}|=s)\\
			\le C_{p}\sum_{s=1}^{p}\Phi_{\nu}^{2s}N^{s-p}\le C_{p}(\Phi_{\nu}+N^{-1/2})^{2p}\le C_{p}\Phi_{\nu}^{2p},\label{eq:E average QtildeV power p}
			\end{multline}
			where the second last step follows from the elementary inequality
			$a^{n}b^{m}\le(a+b)^{n+m}$ for positive $a,b$ and the last step
			follows from the fact that $CN^{-1/2}\le\Phi_{\nu}$. (\ref{eq:E average QtildeV power p})
			shows that $\frac{1}{N}\sum_{i=1}^{N}Q_{i}\mathscr{V}_{i}\prec\Phi_{\nu}^{2}$,
			which concludes (\ref{eq:fluctuation of V}).
		\end{proof}

		\section{Proof of Theorem \ref{thm:rigidity}.}\label{Appendix C}
		Firstly we show (\ref{eq:no eigenvalues outside sepectrum}).
		\begin{proof}[Proof of (\ref{eq:no eigenvalues outside sepectrum})] Using Proposition \ref{prop:fluctuation averaging}, we get from
			(\ref{eq:1(XI)(f(m)-z)}) that
			\[
			|f(m_{N})-z|\prec N^{\varepsilon}\{q^2+\frac{1}{(N\eta)^2}+\frac{\operatorname{Im} m}{N\eta}\},
			\]
			uniformly for $z\in\mathbf{D}^{e}$. Note here we assume $q<N^{-1/3}$. Then, we obtain from Proposition \ref{prop:stability} that, for any
			$\varepsilon,D>0$, as $N$ is sufficiently large,
			\begin{multline*}
			\sup_{z\in\mathbf{D}^{e}}\mathbb{P}\Big(|m_{N}-m|>\frac{N^{\varepsilon}}{\sqrt{\kappa+\eta}}(\frac{\operatorname{Im} m}{N\eta}+\frac{1}{(N\eta)^{2}}+q^2)\Big)\\
			\le\sup_{z\in\mathbf{D}^{e}}\mathbb{P}\Big(|m_{N}-m|>\frac{N^{\varepsilon}(\frac{\operatorname{Im} m}{N\eta}+\frac{1}{(N\eta)^{2}}+q^2)}{\sqrt{\kappa+\eta}+\sqrt{N^{\varepsilon/2}(\frac{\operatorname{Im} m}{N\eta}+\frac{1}{(N\eta)^{2}}+q^2)}}\Big)\le N^{-D},
			\end{multline*}
			so uniformly for $z\in\mathbf{D}^{e}$,
			\begin{equation}
			|m_{N}-m|\prec\frac{1}{\sqrt{\kappa+\eta}}(\frac{\operatorname{Im} m}{N\eta}+\frac{1}{(N\eta)^{2}}+q^2).\label{eq:improved bound of m-m}
			\end{equation}
			
			Denote
			\[
			\tilde{\Psi}(z)=\frac{1}{\sqrt{\kappa+\eta}}(\frac{\operatorname{Im} m}{N\eta}+\frac{1}{(N\eta)^{2}}+q^2).
			\]
			Next, we show that
			\begin{equation}
			\lambda_{1}=\lambda_{+}+O_{\prec}(N^{-2/3}+q^2).\label{eq:rigidity of largest eigenvalue}
			\end{equation}
			We know from Lemma \ref{lem:YY<C w.h.p} there exists some constant
			$C>0$ such that $\lambda_{1}\le C$ with high probability. Therefore,
			it remains to show that for any fixed $\varepsilon>0$, there is no
			eigenvalue of $W$ in the interval
			\begin{equation}
			\mathbf{I}:=[\lambda_{+}+N^{-2/3+4\varepsilon}+N^{4\epsilon}q^2,C]\label{eq:definition of I}
			\end{equation}
			with high probability. The idea of the proof is to choose, for each
			$E\in\mathbf{I}$, a scale $\eta(E)$ such that $\operatorname{Im} m_{N}(E+\imath\eta(E))\le\frac{N^{-\varepsilon}}{N\eta(E)}$
			with high probability. First, we need a simultaneous version of (\ref{eq:improved bound of m-m}),
			i.e.
			\begin{equation}
			\bigcap_{z\in\mathbf{D}^{e}}\Big\{|m_{N}(z)-m(z)|\le\tilde{\Psi}(z)\Big\}\text{ holds with high probability}.\label{eq:improved bound}
			\end{equation}
			
			It suffices to show that
			\begin{equation}
			\sup_{z\in\mathbf{D}^{e}}\frac{|m_{N}(z)-m(z)|}{\tilde{\Psi}(z)}\prec1.\label{eq:uniform local law}
			\end{equation}
			
			Let for $i=1,2$, $z_{i}\equiv E_{i}+\imath\eta_{i}\in\mathbf{D}^{e}$
			and $\kappa_{i}=|E_{i}-\lambda_{+}|$. Elementary calculation yields
			that there exists some constant $C_{1}>0$ such that
			\begin{eqnarray}
			|m_{N}(z_{1})-m_{N}(z_{2})| & \le & C_{1}N^{2}|z_{1}-z_{2}|,\nonumber \\
			|m(z_{1})-m(z_{2})| & \le & C_{1}N^{2}|z_{1}-z_{2}|,\nonumber \\
			|\tilde{\Psi}(z_{1})-\tilde{\Psi}(z_{2})| & \le & C_{1}N^{5/2}|z_{1}-z_{2}|,\nonumber \\
			\inf_{z\in\mathbf{D}^{e}}\tilde{\Psi}(z) & \ge & (C_{1}N)^{-1}.\label{eq:Lipschitz inequalities}
			\end{eqnarray}
			
			We define the $N^{-4}$-net $\hat{\mathbf{D}}^{e}=(N^{-4}\mathbb{Z}^{2})\cap\mathbf{D}^{e}$.
			Hence, $|\hat{\mathbf{D}}^{e}|\le CN^{8}$ and for any $z\in\mathbf{D}^{e}$,
			there exists a $w\in\hat{\mathbf{D}}^{e}$ such that $|z-w|\le2N^{-4}$.
			Then using a simple union bound and (\ref{eq:Lipschitz inequalities}),
			we can deduce (\ref{eq:uniform local law}).
			
			Let $\varepsilon$ be as in (\ref{eq:definition of I}). Then we observe
			from (\ref{eq:improved bound}) and Lemma \ref{lem:basic property of m}
			that
			\begin{equation}
			\bigcap_{z\in\mathbf{D}^{e},E\ge\lambda_{+}}\Big\{|m_{N}(z)-m(z)|\le N^{\varepsilon}\Big(\frac{\eta}{\kappa}\frac{1}{N\eta}+\frac{1}{\sqrt{\kappa}}(\frac{1}{(N\eta)^{2}}+q^2)\Big)\Big\}\label{eq:proof of rigidity bound 1}
			\end{equation}
			holds with high probability.
			
			For each $E\in\mathbf{I}$, we define
			\begin{equation}
			\eta(E)=q^{-1}N^{-1}\kappa(E)^{1/2},\qquad z(E)=E+\imath\eta(E).\label{eq:choice of eta(E)}
			\end{equation}
			
			Using Lemma \ref{lem:basic property of m}, we find that there exists
			a constant $C_{2}>0$ such that for all $E\in\mathbf{I}$
			\begin{equation}
			\operatorname{Im} m(z(E))\le\frac{C_{2}\eta(E)}{\sqrt{\kappa(E)}}\le\frac{C_{2}N^{-\varepsilon}}{N\eta(E)}.\label{eq:rigidity Im m bound}
			\end{equation}
			
			With the choice $\eta(E)$ in (\ref{eq:choice of eta(E)}), we obtain
			from (\ref{eq:proof of rigidity bound 1}) that
			\begin{equation}
			\bigcap_{E\in\mathbf{I}}\Big\{|m_{N}(z)-m(z)|\le\frac{2N^{-\varepsilon}}{N\eta(E)}\Big\}\text{ holds with high probability}.\label{eq:rigidity mN-m bound}
			\end{equation}
			
			From (\ref{eq:rigidity Im m bound}) and (\ref{eq:rigidity mN-m bound})
			we conclude that
			\begin{equation}
			\bigcap_{E\in\mathbf{I}}\Big\{\operatorname{Im} m_{N}(z)\le\frac{(2+C_{2})N^{-\varepsilon}}{N\eta(E)}\Big\}\text{ holds with high probability}.\label{eq:proof of rigidity intersection bounds of Imm}
			\end{equation}
			
			Now suppose that there is an eigenvalue, say $\lambda_{i}$ of $W$
			in $\mathbf{I}$. Then we find that
			\[
			\operatorname{Im} m_{N}(z(\lambda_{i}))=\frac{1}{N}\sum_{j}\frac{\eta(\lambda_{i})}{(\lambda_{j}-\lambda_{i})^{2}+\eta(\lambda_{i})^{2}}\ge\frac{1}{N\eta(\lambda_{i})},
			\]
			which contradicts with the inequality in (\ref{eq:proof of rigidity intersection bounds of Imm}).
			Therefore, we conclude that with high probability, there is no eigenvalue
			in $\mathbf{I}$. Since $\varepsilon>0$ in (\ref{eq:definition of I})
			is arbitrary, (\ref{eq:rigidity of largest eigenvalue}) follows.
		\end{proof}

		Now we show (\ref{eq:local law on small scale}). First we show the following lemma.
		\begin{lem}
			\label{lem:integral of rhoDelta bound}Let $a_{1},a_{2}$ be two numbers
			with $a_{1}\le a_{2}$ and $|a_{1}|+|a_{2}|=O(1)$. For any $E_{1},E_{2}\in[a_{1},a_{2}]$
			and $\eta=N^{-1}$, let $\psi(\lambda):=\psi_{E_{1},E_{2},\eta}(\lambda)$
			be a $C^{2}(\mathbb{R})$ function such that $\psi(x)=1$ for $x\in[E_{1}+\eta,E_{2}-\eta]$,
			$\psi(x)=0$ for $x\in\mathbb{R}\backslash[E_{1},E_{2}]$ and the
			first two derivatives of $\psi$ satisfy $|\psi^{(1)}(x)|\le C\eta^{-1}$,
			$|\psi^{(2)}(x)|\le C\eta^{-2}$ for all $x\in\mathbb{R}$. Let $\varrho^{\Delta}$
			be a signed measure on the real line and $m^{\Delta}$ be the Stieltjes
			transform of $\varrho^{\Delta}$. Suppose, for some positive number
			$c_{N}$ depending on $N$, we have
			\begin{equation}
			|m^{\Delta}(x+\imath y)|\le Cc_{N}(\frac{1}{Ny}+\frac{q^2}{\sqrt{\kappa+y}})\qquad\forall y<1,x\in[a_{1},a_{2}],\label{eq:mDelta bound}
			\end{equation}
			then
			\begin{equation}
			\begin{split}
			\Big|\int\psi(\lambda)\varrho^{\Delta}(d\lambda)\Big|&\leqslant c_N(\frac{1}{N}+\frac{q^2}{\sqrt{\kappa+1/2}})\\
			&\leqslant c_N(\frac{1}{N}+q^2\sqrt{E_2-E_1+\eta})\\
			&\leqslant c_N(\frac{1}{N}+q^3+q^2\sqrt{\kappa_{E_1}}).
			\end{split}
			\end{equation}
		\end{lem}
		
		\begin{proof}
			By (\ref{eq:no eigenvalues outside sepectrum}), it suffices to show the case where $E_2=\lambda_{+}+N^{\epsilon}(q^2+N^{-2/3})$.
			Let $\chi(y)$ be a smooth cutoff function with support $[-1,1]$
			such that $\chi(y)=1$ for $|y|\le1/2$ and $\chi(y)$ has bounded
			derivatives otherwise.
			Define:
			\[
			\varrho^{\Delta}(x):=\rho_{W}(x)-\varrho(x),\qquad m^{\Delta}(z):=m_N(z)-m(z).
			\]
			Using the Helffer-Sj\"{o}strand formula (setting
			$\chi(x+\imath y)=\chi(y)$ in Proposition C.1 of \cite{locallawlecturenotes}),
			we get that
			\[
			\psi(\lambda)=\frac{\imath}{2\pi}\int_{\mathbb{R}^{2}}\frac{y\psi^{(2)}(x)\chi(y)+\{\psi(x)+\imath y\psi^{(1)}(x)\}\chi^{(1)}(y)}{\lambda-x-\imath y}{\rm d}x{\rm d}y.
			\]
			
			Integrating with respect to $\varrho^{\Delta}$ and using the fact
			that $\psi$ and $\chi$ are real, we obtain that
			\begin{eqnarray}
			\Big|\int\psi(\lambda)\varrho^{\Delta}(d\lambda)\Big| & = & \Big|\frac{\imath}{2\pi}\int_{\mathbb{R}^{2}}[y\psi^{(2)}(x)\chi(y)+\{\psi(x)+\imath y\psi^{(1)}(x)\}\chi^{(1)}(y)]m^{\Delta}(x+\imath y){\rm d}x{\rm d}y\Big|\nonumber \\
			& \le & C\int_{\mathbb{R}^{2}}\{|\psi(x)|+|y||\psi^{(1)}(x)|\}|\chi^{(1)}(y)||m^{\Delta}(x+\imath y)|{\rm d}x{\rm d}y\nonumber \\
			&  & +C\Big|\int_{|y|\le\eta}\int_{\mathbb{R}}y\psi^{(2)}(x)\chi(y)\operatorname{Im} m^{\Delta}(x+\imath y){\rm d}x{\rm d}y\Big|\nonumber \\
			&  & +C\Big|\int_{|y|>\eta}\int_{\mathbb{R}}y\psi^{(2)}(x)\chi(y)\operatorname{Im} m^{\Delta}(x+\imath y){\rm d}x{\rm d}y\Big|.\label{eq:f(lambda)rho(dlambda)}
			\end{eqnarray}
			
			With (\ref{eq:mDelta bound}), the first term in (\ref{eq:f(lambda)rho(dlambda)})
			can be estimated as
			\begin{eqnarray*}
				&  & C\int_{\mathbb{R}^{2}}\{|\psi(x)|+|y||\psi^{(1)}(x)|\}|\chi^{(1)}(y)||m^{\Delta}(x+\imath y)|{\rm d}x{\rm d}y\\
				& = & C\int_{[-1,1]\backslash[-1/2,1/2]}\int_{E_{1}}^{E_{2}}|\psi(x)||\chi^{(1)}(y)||m^{\Delta}(x+\imath y)|{\rm d}x{\rm d}y\\
				&  & +C\int_{[-1,1]\backslash[-1/2,1/2]}\int_{[E_{1},E_{2}]\backslash[E_{1}+\eta,E_{2}-\eta]}|y||\psi^{(1)}(x)||\chi^{(1)}(y)||m^{\Delta}(x+\imath y)|{\rm d}x{\rm d}y\\
				&  & \le Cc_N(\frac{1}{N}+\frac{q^2}{\sqrt{\kappa+1/2}}).
			\end{eqnarray*}
			
			The second term in (\ref{eq:f(lambda)rho(dlambda)}) can be estimated
			as
			\begin{eqnarray*}
				&  & C\Big|\int_{|y|\le\eta}\int_{\mathbb{R}}y\psi^{(2)}(x)\chi(y)\operatorname{Im} m^{\Delta}(x+\imath y){\rm d}x{\rm d}y\Big|\\
				& = & C\Big|\int_{|y|\le\eta}\int_{[E_{1},E_{2}]\backslash[E_{1}+\eta,E_{2}-\eta]}y\psi^{(2)}(x)\chi(y)\operatorname{Im} m^{\Delta}(x+\imath y){\rm d}x{\rm d}y\Big|\\
				& \le & \frac{Cc_{N}}{N}.
			\end{eqnarray*}
			
			For the third term in (\ref{eq:f(lambda)rho(dlambda)}), we note that
			\begin{equation*}
			\frac{\partial}{\partial x}\operatorname{Im} m^{\Delta}(x+\imath y)=\operatorname{Im}\Big\{\frac{\partial}{\partial x}m^{\Delta}(x+\imath y)\Big\}
			=\operatorname{Im}\Big\{-\imath\frac{\partial}{\partial y}m^{\Delta}(x+\imath y)\Big\}=-\frac{\partial}{\partial y}\operatorname{Re} m^{\Delta}(x+\imath y).
			\end{equation*}
			
			Then it follows from integration by parts first with respect to $x$
			then $y$ that
			\begin{eqnarray*}
				&  & C\Big|\int_{|y|>\eta}\int_{\mathbb{R}}y\psi^{(2)}(x)\chi(y)\operatorname{Im} m^{\Delta}(x+\imath y){\rm d}x{\rm d}y\Big|\\
				& = & C\Big|\int_{|y|>\eta}\int_{\mathbb{R}}y\psi^{(1)}(x)\chi(y)\frac{\partial}{\partial x}\operatorname{Im} m^{\Delta}(x+\imath y){\rm d}x{\rm d}y\Big|\\
				& = & C\Big|-\int_{\mathbb{R}}\int_{|y|>\eta}y\psi^{(1)}(x)\chi(y)\frac{\partial}{\partial y}\operatorname{Re} m^{\Delta}(x+\imath y){\rm d}y{\rm d}x\Big|\\
				& = & C\Big|-\int_{\mathbb{R}}\Big[y\psi^{(1)}(x)\chi(y)\operatorname{Re} m^{\Delta}(x+\imath y)\Big]_{\eta}^{\infty}{\rm d}x\\
				&  & -\int_{\mathbb{R}}\Big[y\psi^{(1)}(x)\chi(y)\operatorname{Re} m^{\Delta}(x+\imath y)\Big]_{-\infty}^{-\eta}{\rm d}x\\
				&  & +C\int_{\mathbb{R}}\int_{|y|\ge\eta}\psi^{(1)}(x)\{\chi(y)+y\chi^{(1)}(y)\}\operatorname{Re} m^{\Delta}(x+\imath y){\rm d}y{\rm d}x\Big|\\
				& = & C\Big|2\int_{\mathbb{R}}\eta\psi^{(1)}(x)\operatorname{Re} m^{\Delta}(x+\imath\eta){\rm d}x\\
				&  & +C\int_{\mathbb{R}}\int_{|y|>\eta}\psi^{(1)}(x)\{\chi(y)+y\chi^{(1)}(y)\}\operatorname{Re} m^{\Delta}(x+\imath y){\rm d}y{\rm d}x\Big|\\
				& \le & C\int_{\mathbb{R}}\eta|\psi^{(1)}(x)||\operatorname{Re} m^{\Delta}(x+\imath\eta)|{\rm d}x\\
				&  & +\frac{C}{\eta}\int_{\eta<|y|\le1}\int_{[E_{1},E_{2}]\backslash[E_{1}+\eta,E_{2}-\eta]}|\operatorname{Re} m^{\Delta}(x+\imath y)|{\rm d}x{\rm d}y\\
				&  & +C\int_{\eta<|y|\le1}\int_{[E_{1},E_{2}]\backslash[E_{1}+\eta,E_{2}-\eta]}|y\psi^{(1)}(x)\chi^{(1)}(y)\operatorname{Re} m^{\Delta}(x+\imath y)|{\rm d}x{\rm d}y\\
				& \le & C\int_{[E_{1},E_{2}]\backslash[E_{1}+\eta,E_{2}-\eta]}\frac{c_{N}}{N\eta}{\rm d}x+C\int_{\eta<|y|\le1}c_N(\frac{1}{N|y|}+\frac{q^2}{\sqrt{\kappa+|y|}}){\rm d}y\\
				&  & +C\int_{\eta<|y|\le1}\int_{[E_{1},E_{2}]\backslash[E_{1}+\eta,E_{2}-\eta]}\frac{c_N}{\eta}(\frac{1}{N}+\frac{q^2}{\sqrt{\kappa+1/2}}){\rm d}x{\rm d}y\\
				& \le & \frac{Cc_{N}q^2}{\sqrt{\kappa+1/2}}+\frac{Cc_{N}}{N}|\log\eta|\leqslant Cc_N(\frac{\log N}{N}+\frac{q^2}{\sqrt{\kappa+1/2}}).
			\end{eqnarray*}
			
			Then we summarise that
			\begin{equation}
			\begin{split}
			\Big|\int\psi(\lambda)\varrho^{\Delta}(d\lambda)\Big|&\leqslant c_N(\frac{1}{N}+\frac{q^2}{\sqrt{\kappa+1/2}})\\
			&\leqslant c_N(\frac{1}{N}+q^2\sqrt{E_2-E_1+\eta})\\
			&\leqslant c_N(\frac{1}{N}+q^3+q^2\sqrt{\kappa_{E_1}})
			\end{split}
			\end{equation}
		\end{proof}
		
		Next, let $\varrho^{\Delta}$ be the signed measure $\varrho_{N}-\varrho$.
		If $y\ge y_{0}=N^{-1+\tau}$, the condition (\ref{eq:mDelta bound})
		in Lemma \ref{lem:integral of rhoDelta bound} holds for the difference
		$m^{\Delta}=m_{N}-m$ and $c_{N}=N^{\varepsilon}$ for any small $\varepsilon>0$
		with high probability due to Theorem \ref{thm:strong local law}.
		For $y\le y_{0}$, set $z=x+\imath y$, $z_{0}=x+\imath y_{0}$ and
		estimate
		\begin{equation}
		|m_{N}(z)-m(z)|\le|m_{N}(z_{0})-m(z_{0})|+\int_{y}^{y_{0}}\Big|\frac{\partial}{\partial\eta}\{m_{N}(x+\imath\eta)-m(x+\imath\eta)\}\Big|d\eta.\label{eq:bound of m_N-m by int_y^y0}
		\end{equation}
		
		Note that
		\begin{multline*}
		\Big|\frac{\partial}{\partial\eta}m_{N}(x+i\eta)\Big|=\Big|\frac{\partial}{\partial\eta}\int\frac{1}{\lambda-x-\imath\eta}\varrho_{N}({\rm d}\lambda)\Big|\\
		\le\int\frac{1}{|\lambda-x-\imath\eta|^{2}}\varrho_{N}({\rm d}\lambda)=\eta^{-1}\operatorname{Im} m_{N}(x+\imath\eta).
		\end{multline*}
		
		The same bound applies to $|\frac{\partial}{\partial\eta}m(x+\imath\eta)|$
		with $m_{N}$ replaced by $m$.
		
		Then using Theorem \ref{thm:strong local law} and the fact that the
		functions $y\to y\operatorname{Im} m_{N}(x+\imath y)$ and $y\to y\operatorname{Im} m(x+\imath y)$
		are both monotone increasing for any $y>0$ since both are Stieltjes
		transforms of a positive measure, we obtain that
		\begin{eqnarray*}
			\int_{y}^{y_{0}}\Big|\frac{\partial}{\partial\eta}\{m_{N}(x+\imath\eta)-m(x+\imath\eta)\}\Big|d\eta & \le & \int_{y}^{y_{0}}\frac{1}{\eta}\{\operatorname{Im} m_{N}(x+\imath\eta)+\operatorname{Im} m(x+\imath\eta)\}d\eta\\
			& \le & y_{0}\{\operatorname{Im} m_{N}(z_{0})+\operatorname{Im} m(z_{0})\}\int_{y}^{y_{0}}\frac{1}{\eta^{2}}d\eta\\
			& = & y_{0}\{\operatorname{Im} m_{N}(z_{0})+\operatorname{Im} m(z_{0})\}(\frac{1}{y}-\frac{1}{y_{0}})\\
			& = & \{\operatorname{Im} m_{N}(z_{0})+\operatorname{Im} m(z_{0})\}\frac{y_{0}-y}{y}\\
			& \prec & 2\operatorname{Im} m(z_{0})+(Ny_{0})^{-1}.
		\end{eqnarray*}
		
		Hence we have from (\ref{eq:bound of m_N-m by int_y^y0}) that
		\begin{equation}
		|m_{N}(z)-m(z)|\prec2\operatorname{Im} m(z_{0})+(Ny_{0})^{-1}+q\le\frac{CNy+1}{Ny}+q\le\frac{CN^{\tau}}{Ny}+q.\label{eq:check the bound for lemma below}
		\end{equation}
		
		Let $\psi_{E_{1},E_{2},\eta}$ be the function in Lemma \ref{lem:integral of rhoDelta bound}.
		Applying Lemma \ref{lem:integral of rhoDelta bound} below with $c_{N}=N^{\tau}$,
		we obtain that for any $\eta=N^{-1}$
		\[
		\Big|\int_{\mathbb{R}}\psi_{E_{1},E_{2},\eta}(\lambda)\varrho_{N}(d\lambda)-\int_{\mathbb{R}}\psi_{E_{1},E_{2},\eta}(\lambda)\varrho(d\lambda)\Big|\prec N^{-1+\tau}.
		\]
		
		Integrating with respect to $\varrho_{N}({\rm d}\lambda)$ and $\varrho({\rm d}\lambda)$
		on both sides of the following elementary inequality
		\[
		\mathbf{1}_{[x-\eta,x+\eta]}(\lambda)\le\frac{2\eta^{2}}{(\lambda-x)^{2}+\eta^{2}}\qquad\forall x,\lambda\in\mathbb{R},
		\]
		and using (\ref{eq:check the bound for lemma below}), Lemma \ref{lem:basic property of m}
		and the definitions of $y_{0}$ and $q$, we get that for some constant $C>0$
		\[
		\mathfrak{n}_{N}(x-\eta,x+\eta)\le C\eta\operatorname{Im} m_{N}(x+\imath\eta)\le Cy_{0}\operatorname{Im} m_{N}(x+\imath y_{0})\prec N^{-1+\tau},
		\]
		and
		\[
		\mathfrak{n}(x-\eta,x+\eta)\le C\eta\operatorname{Im} m(x+\imath\eta)\le Cy_{0}\operatorname{Im} m(x+\imath y_{0})\prec N^{-1+\tau},
		\]
		uniformly for $x$ in a small neighborhood of $\lambda_{+}$.
		
		We note that
		\begin{eqnarray*}
			|\mathfrak{n}_{N}(E_{1},E_{2})-\mathfrak{n}(E_{1},E_{2})| & = & |\int_{E_{1}}^{E_{2}}\varrho_{N}({\rm d}\lambda)-\int_{E_{1}}^{E_{2}}\varrho({\rm d}\lambda)|\\
			& \le & |\int_{E_{1}+\eta}^{E_{2}-\eta}\psi_{E_{1},E_{2},\eta}(\lambda)\varrho_{N}({\rm d}\lambda)-\int_{E_{1}+\eta}^{E_{2}-\eta}\psi_{E_{1},E_{2},\eta}(\lambda)\varrho({\rm d}\lambda)|\\
			&  & +|\int_{E_{1}}^{E_{1}+\eta}\varrho_{N}({\rm d}\lambda)|+|\int_{E_{1}}^{E_{1}+\eta}\varrho({\rm d}\lambda)|\\
			&  & +|\int_{E_{2}-\eta}^{E_{2}}\varrho_{N}({\rm d}\lambda)|+|\int_{E_{2}-\eta}^{E_{2}}\varrho({\rm d}\lambda)|\\
			& \prec & (\frac{1}{N^{-1+\tau}}+q^3+q^2(\sqrt{\kappa_{E_1}}-\sqrt{\kappa_{E_2}}),
		\end{eqnarray*}
		Since $\tau$ is arbitrary, the desired result follows.
		
		Finally, we are ready to show (\ref{eq:rigidity of eigenvalues}).
		\begin{proof}[Proof of (\ref{eq:rigidity of eigenvalues})]
			With the choice of $q<N^{-1/3}$ by Lemma \ref{eq:local law on small scale} we have that if $\lambda_i,\gamma_i\geqslant\lambda_{+}-N^{c}N^{-2/3}$ for some $c>0$, then, with high probability
			\begin{equation}\label{eq:x near the edge}
			|\lambda_i-\gamma_i|\leqslant N^{-\epsilon}N^{-2/3},
			\end{equation}
			for some $\epsilon>0$. By the square root behavior of $\varrho$, we have $\mathfrak{n}(x)\sim(\lambda_1-x)^{3/2}$ when $x$ is near the edge. That is
			\[
			\mathfrak{n}(\gamma_j)=\frac{j}{N}\sim(\lambda_1-\gamma_j)^{3/2}.
			\]
			
			Thus we have proved the case where $j\leqslant N^{c}$ for small $c$. Together with (\ref{eq:x near the edge}), we conclude (\ref{eq:rigidity of eigenvalues}). For the rest of $j$'s, one can refer to \cite{Ding2018}, so we omit the details since we only need the result near the right edge.
		\end{proof}
		
		\section{Complete the proof of Theorem \ref{thm:Edge universality with small support}.}
		
		We introduce the notation of functional calculus.
		Specifically, for a function $f(\cdot)$ and a matrix $H$, $f(H)$
		denotes the matrix whose eigenvectors are those of $H$ and eigenvalues are the values
		of $f$ applied to each eigenvalue of $H$.
		
		First, we present a lemma for the approximation of the eigenvalue
		counting function. For any $\eta>0$, define
		\[
		\vartheta_{\eta}(x)=\frac{\eta}{\pi(x^{2}+\eta^{2})}=\frac{1}{\pi}\operatorname{Im}\frac{1}{x-\imath\eta}.
		\]
		
		We notice that for any $a,b\in\mathbb{R}$ with $a\le b$, the convolution
		of $\mathbf{1}_{[a,b]}$ and $\vartheta_{\eta}$ applied to the eigenvalues
		$\lambda_{i}$, $i=1,\dots,N$ yields that
		\[
		\sum_{i=1}^{N}\mathbf{1}_{[a,b]}*\vartheta_{\eta}(\lambda_{i})=\frac{N}{\pi}\int_{a}^{b}\operatorname{Im} m_{N}(x+\imath\eta){\rm d}x.
		\]
		In terms of the functional calculus notation, we have
		\begin{eqnarray*}
			\sum_{i=1}^{N}\mathbf{1}_{[a,b]}(\lambda_{i}) =  {\rm Tr}\mathbf{1}_{[a,b]}(W),\quad
			\sum_{i=1}^{N}\mathbf{1}_{[a,b]}*\vartheta_{\eta}(\lambda_{i})  =  {\rm Tr}\mathbf{1}_{[a,b]}*\vartheta_{\eta}(W).
		\end{eqnarray*}
		For $a,b\in\mathbb{R}\cup\{-\infty,\infty\}$, define ${\cal N}(a,b)=N\int_{a}^{b}\varrho_{N}({\rm d}x)$
		as the number of eigenvalues of $W$ in $[a,b]$.
		
		The following lemma shows that ${\rm Tr}\mathbf{1}_{[a,b]}(W)$ can
		be well approximated by its smoothed version ${\rm Tr}\mathbf{1}_{[a,b]}*\vartheta_{\eta}(W)$
		for $a,b$ around the edge $\lambda_{+}$ so that the problem can
		be converted to comparison of the Stieltjes transform.
		\begin{lem}
			\label{lem:smooth approximation of indicator}Let $\varepsilon>0$
			be an arbitrarily small number. Set $E_{\varepsilon}=\lambda_{+}+N^{-2/3+\varepsilon}$,
			$\ell_{1}=N^{-2/3-3\varepsilon}$ and $\eta_{1}=N^{-2/3-9\varepsilon}$.
			Then for any $E$ satisfying $|E-\lambda_{+}|\le\frac{3}{2}N^{-2/3+\varepsilon}$,
			it holds with high probability that
			\[
			|{\rm Tr}\mathbf{1}_{[E,E_{\varepsilon}]}(W)-{\rm Tr}\mathbf{1}_{[E,E_{\varepsilon}]}*\vartheta_{\eta}(W)|\le C(N^{-2\varepsilon}+{\cal N}(E-\ell_{1},E+\ell_{1})).
			\]
		\end{lem}
		\begin{proof}
			See Lemma 4.1 of \cite{pillai2014} or Lemma 6.1 of \cite{ERDOS2012}.
		\end{proof}
		Let $q:\mathbb{R\to\mathbb{R}}_{+}$ be a smooth cutoff function such
		that
		\[
		q(x)=\begin{cases}
		1 & \text{if }|x|\le1/9,\\
		0 & \text{if }|x|\ge2/9,
		\end{cases}
		\]
		so $q(x)$ is decreasing for $x\ge0$. Then we have
		the following corollary.
		\begin{cor}
			\label{cor:bound of N(E,infty)}Let $\varepsilon,\ell_{1},\eta_{1},E_{\varepsilon}$
			be defined in Lemma \ref{lem:smooth approximation of indicator}. Set $\ell=\ell_{1}N^{2\varepsilon}/2=N^{-2/3-\varepsilon}/2$.
			Then for all $E$ such that
			\begin{equation}
			|E-\lambda_{+}|\le N^{-2/3+\varepsilon},\label{eq:choice of E in Corollary 5.1}
			\end{equation}
			the inequality
			\begin{equation}
			{\rm Tr}\mathbf{1}_{[E+\ell,E_{\varepsilon}]}*\vartheta_{\eta_{1}}(W)-N^{-\varepsilon}\le{\cal N}(E,\infty)\le{\rm Tr}\mathbf{1}_{[E-\ell,E_{\varepsilon}]}*\vartheta_{\eta_{1}}(W)+N^{-\varepsilon}\label{eq:bound of N(E,infty)}
			\end{equation}
			holds with high probability. Furthermore, for any $D>0$, there exists
			$N_{0}\in\mathbb{N}$ independent of $E$ such that for all $N\ge N_{0}$,
			\begin{multline}
			\mathbb{E}q\Big\{{\rm Tr}\mathbf{1}_{[E-\ell,E_{\varepsilon}]}*\vartheta_{\eta_{1}}(W)\Big\}\\
			\le\mathbb{P}({\cal N}(E,\infty)=0)\le\mathbb{E}q\Big\{{\rm Tr}\mathbf{1}_{[E+\ell,E_{\varepsilon}]}*\vartheta_{\eta_{1}}(W)\Big\}+N^{-D}.\label{eq:bound of P(N(E,infty))}
			\end{multline}
		\end{cor}
		\begin{proof}
			Notice that for $E$ satisfying $|E-\lambda_{+}|\le N^{-2/3+\varepsilon}$,
			we have $|E-\ell-\lambda_{+}|\le|E-\lambda_{+}|+\ell\le\frac{3}{2}N^{-2/3+\varepsilon}$.
			Therefore, Lemma \ref{lem:smooth approximation of indicator} holds
			with $E$ replaced by any $x\in[E-\ell,E]$. By the mean value theorem,
			we obtain that with high probability,
			\begin{eqnarray*}
				{\rm Tr}\mathbf{1}_{[E,E_{\varepsilon}]}(W) & \le & \ell^{-1}\int_{E-\ell}^{E}{\rm Tr}\mathbf{1}_{[x,E_{\varepsilon}]}(W){\rm d}x\\
				& \le & \ell^{-1}\int_{E-\ell}^{E}{\rm Tr}\mathbf{1}_{[x,E_{\varepsilon}]}*\vartheta_{\eta_{1}}(W){\rm d}x\\
				&  & +C\ell^{-1}\int_{E-\ell}^{E}\{N^{-2\varepsilon}+{\cal N}(x-\ell_{1},x+\ell_{1})\}{\rm d}x\\
				& \le & {\rm Tr}\mathbf{1}_{[E-\ell,E_{\varepsilon}]}*\vartheta_{\eta_{1}}(W){\rm d}x\\
				&  & +CN^{-2\varepsilon}+C\frac{\ell_{1}}{\ell}{\cal N}(E-2\ell,E+\ell),
			\end{eqnarray*}
			where the last inequality follows from the fact that each eigenvalue in $[E-2\ell,E+\ell]$
			will contribute to the integral $\int_{E-\ell}^{E}{\cal N}(x-\ell_{1},x+\ell_{1}){\rm d}x$
			at most $2\ell_{1}$ mass, since the length of the interval $[x-\ell_{1},x+\ell_{1}]$
			is $2\ell_{1}$ for each $x\in[E-\ell,E]$. From Theorem \ref{thm:rigidity},
			(\ref{eq:choice of E in Corollary 5.1}), $\ell_{1}/\ell=2N^{-2\varepsilon}$,
			$\ell\le N^{-2/3}$ and the square root behavior of $\varrho$ (see
			e.g. Lemma 2.1 of \cite{Bao2013}), we get that
			
			\[
			\begin{split}
			\frac{\ell_{1}}{\ell}{\cal N}(E-2\ell,E+\ell)=& 2N^{1-2\varepsilon}\{\mathfrak{n}(E-2\ell,E+\ell)+O_{\prec}(N^{-1})\}\\
			=&2N^{1-2\varepsilon}\{\int_{E-2\ell}^{E+\ell}\varrho({\rm d}x)+O_{\prec}(N^{-1})\}\\
			\le & 2N^{1-2\varepsilon}\{\int_{E-2\ell}^{(E+\ell)\land\lambda_{+}}O(1)\sqrt{\lambda_{+}-x}\varrho({\rm d}x)+O_{\prec}(N^{-1})\}\\
			\le & CN^{-5\varepsilon/2}+O_{\prec}(N^{-2\varepsilon}).
			\end{split}
			\]
			
			We have thus proved that
			\[
			{\cal N}(E,E_{\varepsilon})={\rm Tr}\mathbf{1}_{[E,E_{\varepsilon}]}(W)\le{\rm Tr}\mathbf{1}_{[E-\ell,E_{\varepsilon}]}*\vartheta_{\eta_{1}}(W)+N^{-\varepsilon}
			\]
			holds with high probability.
			
			Using Theorem \ref{thm:rigidity}, it follows that we can replace
			${\cal N}(E,E_{\varepsilon})$ by ${\cal N}(E,\infty)$ with a loss
			of probability of at most $N^{-D}$ for any large $D>0$. This proves
			the upper bound of (\ref{eq:bound of N(E,infty)}). The lower bound
			of (\ref{eq:bound of N(E,infty)}) can be shown analogously.
			
			When (\ref{eq:bound of N(E,infty)}) holds, the event ${\cal N}(E,\infty)=0$
			implies that ${\rm Tr}\mathbf{1}_{[E+\ell,E_{\varepsilon}]}*\vartheta_{\eta_{1}}(W)\le1/9$.
			Thus we have
			\[
			\mathbb{P}({\cal N}(E,\infty)=0)\le\mathbb{P}({\rm Tr}\mathbf{1}_{[E+\ell,E_{\varepsilon}]}*\vartheta_{\eta_{1}}(W)\le1/9)+N^{-D},
			\]
			which together with Markov's inequality proves the upper bound of (\ref{eq:bound of P(N(E,infty))}).
			For the lower bound, by using the upper bound of (\ref{eq:bound of N(E,infty)})
			and the fact that ${\cal N}(E,\infty)$ is an integer, we see that
			\begin{multline*}
			\mathbb{E}q\big({\rm Tr}\mathbf{1}_{[E-\ell,E_{\varepsilon}]}*\vartheta_{\eta_{1}}(W)\big)\le\mathbb{P}\big({\rm Tr}\mathbf{1}_{[E-\ell,E_{\varepsilon}]}*\vartheta_{\eta_{1}}(W)\le2/9\big)\\
			\le\mathbb{P}\big({\cal N}(E,\infty)\le2/9+N^{-\varepsilon}\big)=\mathbb{P}\big({\cal N}(E,\infty)=0\big).
			\end{multline*}
			
			This completes the proof of Corollary \ref{cor:bound of N(E,infty)}.
		\end{proof}
		
		\begin{proof}[Proof of Theorem \ref{thm:Edge universality with small support}]
			Let $\varepsilon>0$ be an arbitrary small number. Let $E=\lambda_{+}+sN^{-2/3}$
			for some $|s|\le N^{\varepsilon}$. Define $E_{\varepsilon}=\lambda_{+}+N^{-2/3+\varepsilon}$,
			$\ell=N^{-2/3-\varepsilon}/2$ and $\eta_{1}=N^{-2/3-9\varepsilon}$.
			Define $\tilde{W}$, $\tilde{{\cal N}}$ to be the analogs of $W$, ${\cal N}$
			but with $\tilde{X}$ in place of $X$.
			
			Using Corollary \ref{cor:bound of N(E,infty)}, we have
			\begin{equation}
			\mathbb{E}q\big({\rm Tr}\mathbf{1}_{[E-\ell,E_{\varepsilon}]}*\vartheta_{\eta_{1}}(\tilde{W})\big)\le\mathbb{P}\big(\tilde{{\cal N}}(E,\infty)=0\big).\label{eq:proof of universality bound 1}
			\end{equation}
			
			Recall that by definition
			\[
			{\rm Tr}\mathbf{1}_{[E-\ell,E_{\varepsilon}]}*\vartheta_{\eta_{1}}(W)=\frac{N}{\pi}\int_{E-\ell}^{E_{\varepsilon}}\operatorname{Im} m_{N}(x+\imath\eta_{1}){\rm d}x.
			\]
			
			Theorem \ref{thm:GreenFuncitonComparision} applied to the case where
			$E_{1}=E-\ell$ and $E_{2}=E_{\varepsilon}$ shows that there exists
			$\delta>0$ such that
			\begin{equation}
			\mathbb{E}q\big({\rm Tr}\mathbf{1}_{[E-\ell,E_{\varepsilon}]}*\vartheta_{\eta_{1}}(W)\big)\le\mathbb{E}q\big({\rm Tr}\mathbf{1}_{[E-\ell,E_{\varepsilon}]}*\vartheta_{\eta_{1}}(\tilde{W})\big)+N^{-\delta}.\label{eq:difference between q and qtilde}
			\end{equation}
			Then applying Corollary \ref{cor:bound of N(E,infty)} to the left-hand
			side of (\ref{eq:difference between q and qtilde}), we have for arbitrarily
			large $D>0$
			\begin{equation}
			\mathbb{P}\big({\cal N}(E-2\ell,\infty)=0\big)\le\mathbb{E}q\big({\rm Tr}\mathbf{1}_{[E-\ell,E_{\varepsilon}]}*\vartheta_{\eta_{1}}(W)\big)+N^{-D}\label{eq:proof of universality bound 2}
			\end{equation}
			as $N$ is sufficiently large.
			
			Using the bounds (\ref{eq:proof of universality bound 1}), (\ref{eq:difference between q and qtilde})
			and (\ref{eq:proof of universality bound 2}), we get that
			\[
			\mathbb{P}\big({\cal N}(E-2\ell,\infty)=0\big)\le\mathbb{P}\big(\tilde{{\cal N}}(E,\infty)=0\big)+2N^{-\delta}
			\]
			for sufficiently small $\varepsilon>0$ and sufficiently large $N$.
			Recall that $E=\lambda_{+}+sN^{-2/3}$. The proof of the first inequality
			of Theorem \ref{thm:Edge universality with small support} is thus complete. By switching
			the roles of $X$ and $\tilde{X}$, the second inequality follows.
			The proof is done.
		\end{proof}

		\section{\label{Appendix D}Some useful results.}
		
		\begin{lem}[Sherman-Morrison formula]
			\label{lem:Sherman-Morrison formula}Let $A$ be an invertible matrix
			and $\mathbf{x}$ be a column vector such that $\mathbf{x}\mathbf{x}^{*}$
			has the same size as $A$. Then
			\begin{equation}
			(A+\mathbf{x}\mathbf{x}^{*})^{-1}=A^{-1}-\frac{A^{-1}\mathbf{x}\mathbf{x}^{*}A^{-1}}{1+\mathbf{x}^{*}A^{-1}\mathbf{x}}.\label{eq:Sherman-Morrison formula}
			\end{equation}
		\end{lem}
		\begin{proof}
			Multiplying $A+\mathbf{x}\mathbf{x}^{*}$ on both sides of (\ref{eq:Sherman-Morrison formula}),
			we get the identity $I=I$.
		\end{proof}
		
		\begin{thm}[Moments of uniform spherical distribution]
			\label{thm:mix_moment}Let $\mathbf{u}=(u_{1},\dots,u_{M})^{\prime}$
			be an $M$-dimensional random vector of spherical uniform distribution.
			Let $n\in\{1,\dots,M\}$, $i_{1},\dots,i_{n}\in\{1,\dots,M\}$ and
			$k_{1},\dots,k_{n}$ be positive integers. Defining $k_{0}=k_{1}+\cdots+k_{n}$,
			we have
			\[
			\mathbb{E}|u_{i_{1}}^{k_{1}}\cdots u_{i_{n}}^{k_{n}}|=\frac{\Gamma(\frac{M}{2})\prod_{i=1}^{n}\Gamma(\frac{k_{i}+1}{2})}{\pi^{n/2}\Gamma(\frac{k_{0}+M}{2})}\le C_{k_{0},n}M^{-k_{0}/2},
			\]
			where $C_{k_{0},n}>0$ is a constant depending on $k_{0}$ and $n$
			only. Moreover, if for some $j\in\{1,\dots,n\}$, $k_{j}$ is odd.
			Then
			\[
			\mathbb{E}(u_{i_{1}}^{k_{1}}\cdots u_{i_{n}}^{k_{n}})=0.
			\]
		\end{thm}
		\begin{proof}[Proof of Theorem \ref{thm:mix_moment}]
			Write $\mathbf{u}=\mathbf{z}/\|\mathbf{z}\|$ for some $N_{M}(0,I)$
			random vector $\mathbf{z}$. Then $\|\mathbf{z}\|$ follows a half
			normal distribution with scale parameter $M$ and $\mathbf{u}$ is
			independent with $\|\mathbf{z}\|$ (see e.g. Page 37 of \cite{Muirhead2005}). So
			\[
			\mathbb{E}(\|\mathbf{z}\|^{k_{0}}|u_{i_{1}}^{k_{1}}\cdots u_{i_{n}}^{k_{n}}|)=\mathbb{E}|\mathbf{z}_{i_{1}}^{k_{1}}\cdots\mathbf{z}_{i_{n}}^{k_{n}}|=\frac{2^{k_{0}/2}\prod_{i=1}^{n}\Gamma(\frac{k_{i}+1}{2})}{\pi^{n/2}}.
			\]
			
			We see that
			\begin{eqnarray*}
				\mathbb{E}(\|\mathbf{z}\|^{k_{0}}) & = & \frac{2^{k_{0}/2}\Gamma(\frac{k_{0}+M}{2})}{\Gamma(\frac{M}{2})}\\
				& = & \begin{cases}
					2^{k_{0}/2}\prod_{i=1}^{k_{0}/2}(\frac{k_{0}+M}{2}-i) & \text{if }k_{0}\text{ is even},\\
					2^{k_{0}/2}\prod_{i=1}^{(k_{0}-1)/2}(\frac{k_{0}+M}{2}-i)\frac{\Gamma(\frac{M+1}{2})}{\Gamma(\frac{M}{2})} & \text{if }k_{0}\text{ is odd},
				\end{cases}\\
				& \ge & \begin{cases}
					M^{k_{0}/2} & \text{if }k_{0}\text{ is even},\\
					\sqrt{\frac{1}{2}}M^{k_{0}/2} & \text{if }k_{0}\text{ is odd},
				\end{cases}
			\end{eqnarray*}
			where for odd $k_{0}$, we have used Wendel's inequality (see e.g.
			\cite{Qi2013}) that
			\[
			\frac{\Gamma(x+s)}{\Gamma(x)}\ge x^{s}\Big(\frac{x}{x+s}\Big)^{1-s}\qquad\forall x>0,\ 0<s<1.
			\]
			
			The second result simply follows from the fact that
			\[
			\mathbb{E}(u_{i_{1}}^{k_{1}}\cdots u_{i_{j}}^{k_{j}}\cdots u_{i_{n}}^{k_{n}})=\mathbb{E}\{u_{i_{1}}^{k_{1}}\cdots(-u_{i_{j}})^{k_{j}}\cdots u_{i_{n}}^{k_{n}}\}=\mathbb{E}\{u_{i_{1}}^{k_{1}}\cdots(-u_{i_{j}}^{k_{j}})\cdots u_{i_{n}}^{k_{n}}\}=0.
			\]
		\end{proof}
		
		\begin{lem}
			\label{lem:YY<C w.h.p} Under Conditions \ref{cond:1} and \ref{cond:3},
			there exists a constant $C>0$ such that $\|XX^{*}\|\le C$ with
			high probability.
		\end{lem}
		\begin{proof}
			Recall that $XX^{*}=\Sigma^{1/2}U\mathscr{D}^{2}U^{*}\Sigma^{1/2}$.
			We observe that
			\begin{equation}
			\max_{i}\xi_{i}^{2}=\max_{i}(\xi_{i}^{2}-MN^{-1})+MN^{-1}=MN^{-1}+O_{\prec}(N^{-1/2}).\label{eq:max xi^2}
			\end{equation}
			
			From (\ref{eq:max xi^2}), the assumption that $\|\Sigma\|$ is bounded
			and the inequality $\|XX^{*}\|\le\|\Sigma\|\|UU^{*}\|$ $\|\mathscr{D}^{2}\|$,
			we see that it suffices to show that there exists a constant $C>0$
			such that $\|UU^{*}\|\le C$ with high probability. Write the $i,j$-th
			entry of $U$ as $U_{ij}$. We see that $\mathbb{E}\sqrt{M}U_{ij}=0$,
			$\mathbb{E}(\sqrt{M}U_{ij})^{2}=1$ and $\mathbb{E}(\sqrt{M}U_{ij})^{k}<\infty$
			for all $k\ge3$. The desired result then follows from the arguments
			in \cite{Yin1988} applied to the matrix $MN^{-1}UU^{*}$ which shows
			that for any $k\in\{1,2,\cdots\}$ and $c>(1+\sqrt{M/N})^{2}$, $\mathbb{E}{\rm Tr}(MN^{-1}UU^{*})^{k}\le c^{k}$
			for all large $N$. Indeed, the matrix $U$ violates the assumption
			in \cite{Yin1988} that all entries of $U$ are mutually independent.
			However, this will not invalidate the proof because the strategy of
			\cite{Yin1988} is to bound the probability $\mathbb{P}({\rm Tr}(MN^{-1}UU^{*})^{k}>c^{k})$,
			for which the only inputs needed are the bounds on
			\[
			|\mathbb{E}U_{i_{1}j_{1}}U_{i_{2}j_{1}}\cdots U_{i_{k}j_{k}}U_{i_{1}j_{k}}|.
			\]
			
			We note that the only consequence that the violation of independence
			leads to is that the expectation of products of the $U$ entires from
			the same column do not factor into products of expectation. However,
			this is not a problem since the expectation of the product of dependent
			$U$ entries can be bounded by product of individual expectations.
			To be specific, we consider, without loss of generality, the $U$ entries
			from the first column. Let $m\le M$ be a positive integer, $i_{1},\dots,i_{m}\in\{1,\dots,M\}$
			be distinct $m$ integers and $a_{1},\dots,a_{m}$ be $m$ positive
			integers. Denote $a_{0}=a_{1}+\cdots+a_{m}$. Then we claim that
			\begin{equation}
			\mathbb{E}(U_{i_{1}1}^{a_{1}}\cdots U_{i_{m}1}^{a_{m}})\le\mathbb{E}(U_{i_{1}1}^{a_{1}})\cdots\mathbb{E}(U_{i_{m}1}^{a_{m}}).\label{eq:EU<factored}
			\end{equation}
			
			If one of $a_{1},\dots,a_{m}$ is odd, (\ref{eq:EU<factored})
			is true because by symmetry of the $U$ entries, both $\mathbb{E}(U_{i_{1}1}^{a_{1}}\cdots U_{i_{m}1}^{a_{m}}).$
			and $\mathbb{E}(U_{i_{1}1}^{a_{1}})\cdots\mathbb{E}(U_{i_{m}1}^{a_{m}})$
			are $0$. Hence from now on, we assume that all of $a_{1},\dots,a_{m}$
			are even.
			
			Let $\mathbf{z}\equiv(\mathbf{z}_{1},\dots,\mathbf{z}_{M})^{*}$ be
			a real-valued $M$ dimensional standard normal random vector. Then
			we have that $\mathbf{z}/\|\mathbf{z}\|\sim U(\mathbb{S}^{M-1})$
			and $\mathbf{z}/\|\mathbf{z}\|$ are independent with $\|\mathbf{z}\|$.
			It thus follows that
			\[
			\mathbb{E}(U_{i_{1}1}^{a_{1}}\cdots U_{i_{m}1}^{a_{m}})\mathbb{E}\|\mathbf{z}\|^{a_{0}}=\mathbb{E}(\mathbf{z}_{i_{1}}^{a_{1}})\cdots\mathbb{E}(\mathbf{z}_{i_{m}}^{a_{m}}).
			\]
			
			Therefore,
			\begin{multline*}
			\frac{\mathbb{E}(U_{i_{1}1}^{a_{1}}\cdots U_{i_{m}1}^{a_{m}})}{\mathbb{E}(U_{i_{1}1}^{a_{1}})\cdots\mathbb{E}(U_{i_{m}1}^{a_{m}})}=\frac{\mathbb{E}(\|\mathbf{z}\|^{a_{1}})\cdots\mathbb{E}(\|\mathbf{z}\|^{a_{m}})}{\mathbb{E}(\|\mathbf{z}\|^{a_{0}})}\le\frac{\mathbb{E}(\|\mathbf{z}\|^{a_{1}})\cdots\mathbb{E}(\|\mathbf{z}\|^{a_{m}})}{\mathbb{E}(\|\mathbf{z}\|^{a_{1}})\mathbb{E}(\|\mathbf{z}\|^{a_{0}-a_{1}})}\\
			\le\frac{\mathbb{E}(\|\mathbf{z}\|^{a_{1}})\cdots\mathbb{E}(\|\mathbf{z}\|^{a_{m}})}{\mathbb{E}(\|\mathbf{z}\|^{a_{1}})(\mathbb{E}\|\mathbf{z}\|^{a_{2}})\mathbb{E}(\|\mathbf{z}\|^{a_{0}-a_{1}-a_{2}})}\le\cdots\le\frac{\mathbb{E}(\|\mathbf{z}\|^{a_{1}})\cdots\mathbb{E}(\|\mathbf{z}\|^{a_{m}})}{\mathbb{E}(\|\mathbf{z}\|^{a_{1}})\cdots\mathbb{E}(\|\mathbf{z}\|^{a_{m}})}=1,
			\end{multline*}
			where the first to the last inequalities follow from the fact that
			$\|\mathbf{z}\|^{a_{k}}$ and $\|\mathbf{z}\|^{a_{0}-\sum_{j=1}^{k}a_{j}}$
			are positive correlated for all $k=1,\dots,m-1$. Now the proof of
			the claim is complete and this lemma is shown.
		\end{proof}

		\
	\end{appendix}


\begin{thebibliography}{4}
		\bibitem{fang1990}
		\textsc{Anderson, T.} and \textsc{Fang, K.} (1990).
		\textit{Statistical inference in elliptically contoured and related
			distributions}.
		Technical report. Standford University, Stanford, California.
		https://apps.dtic.mil/dtic/tr/fulltext/u2/a230672.pdf.
		
		
		\bibitem[\protect\astroncite{Bai and Ng}{2002}]{BaiandNg2002}
		{\sc Bai, J.} and {\sc Ng, S.} (2002).
		\textit{ Determining the number of factors in approximate factor models}.
		Econometrica, 70(1):191--221.
		
		\bibitem[\protect\astroncite{Bai}{1999}]{Bai1999review}
		{\sc Bai, Z.~D.} (1999).
		\textit {Methodologies in spectral analysis of large dimensional random
			matrices, a review}.
		Statist. Sinica, 9:611--677.
		
		\bibitem[\protect\astroncite{Bai and Silverstein}{1998}]{BaiandSilverstein1998}
		{\sc Bai, Z.~D.} and {\sc Silverstein, J.~W.} (1998).
		\textit{No eigenvalues outside the support of the limiting spectral
			distribution of large dimensional sample covariance matrices}.
		Ann. Probab., 26:316--345.
		
		\bibitem[\protect\astroncite{Bai and Silverstein}{2010}]{BaiandSilverstein2010}
		{\sc Bai, Z.~D.} and {\sc Silverstein, J.~W.} (2010).
		\textit {Spectral Analysis of Large Dimensional Random Matrices}.
		Springer, 2nd edition.
		
		\bibitem[\protect\astroncite{Bai and Zhou}{2008}]{BaiandZhou2008}
		{\sc Bai, Z.~D.} and {\sc Zhou, W.} (2008).
		\textit{Large sample covariance matrices without independence structures in
			columns}.
		Statistica Sinica, 18(2):425--442.
		
		\bibitem[\protect\astroncite{Baik et~al.}{2005}]{baik2005}
		{\sc Baik, J.}, {\sc Ben~Arous, G.}, and {\sc {P\'ech\'e}, S.} (2005).
		\textit {Phase transition of the largest eigenvalue for nonnull complex sample
			covariance matrices}.
		Ann. Probab., 33(5):1643--1697.
		
		\bibitem[\protect\astroncite{Bao et~al.}{2015}]{bao2015}
		{\sc Bao, Z.}, {\sc Pan, G.}, and {\sc Zhou, W.} (2015).
		\textit{Universality for the largest eigenvalue of sample covariance matrices
			with general population}.
		Ann. Statist., 43(1):382--421.
		
		\bibitem[\protect\astroncite{Bao et~al.}{2013}]{Bao2013}
		{\sc Bao, Z.~G.}, {\sc Pan, G.~M.}, and {\sc Zhou, W.} (2013).
		\textit{Local density of the spectrum on the edge for sample covariance matrices with general population}.
		Ann. Statist., 2015, 43(1): 382-421.
		
		\bibitem[\protect\astroncite{{Benaych-Georges} and
			{Knowles}}{2016}]{locallawlecturenotes}
		{\sc {Benaych-Georges}, F.} and {\sc {Knowles}, A.} (2016).
		\textit {Lectures on the local semicircle law for Wigner matrices}.
		arXiv e-prints, page arXiv:1601.04055.
		
		\bibitem[\protect\astroncite{Burkholder}{1973}]{burkholder1973}
		{\sc Burkholder, D.~L.} (1973).
		\textit {Distribution function inequalities for martingales}.
		Ann. Probab., 1(1):19--42.
		
		\bibitem[\protect\astroncite{Cai et~al.}{2015}]{Cai2015}
		{\sc Cai, T.}, {\sc Ma, Z.}, and {\sc Wu, Y.} (2015).
		\textit{Optimal estimation and rank detection for sparse spiked covariance
			matrices}.
		Probab. Theory Related Fields, 161(3):781--815.
		
		\bibitem[\protect\astroncite{Ding}{2017}]{Ding2017}
		{\sc Ding, X.} (2017).
		\textit {Asymptotics of empirical eigen-structure for high dimensional sample
			covariance matrices of general form}.
		arXiv e-prints, page arXiv:1708.06296.
		
		\bibitem[\protect\astroncite{Ding and Yang}{2018}]{Ding2018}
		{\sc Ding, X.} and {\sc Yang, F.} (2018).
		\textit{A necessary and sufficient condition for edge universality at the largest singular values of covariance matrices}.
		Ann. Appl. Probab., 28(3):
		1679-1738.
		
		\bibitem[\protect\astroncite{{Dobriban} and {Liu}}{2018}]{Dobriban2018b}
		{\sc {Dobriban}, E.} and {\sc {Liu}, S.} (2018).
		\textit {A New Theory for Sketching in Linear Regression}.
		arXiv e-prints, page arXiv:1810.06089.
		
		\bibitem[\protect\astroncite{{Dobriban} and {Sheng}}{2018}]{Dobriban2018a}
		{\sc {Dobriban}, E.} and {\sc {Sheng}, Y.} (2018).
		\textit {Distributed linear regression by averaging}.
		arXiv e-prints, page arXiv:1810.00412.
		
		\bibitem[\protect\astroncite{El~Karoui}{2007}]{elkaroui2007}
		{\sc El~Karoui, N.} (2007).
		\textit{Tracy-widom limit for the largest eigenvalue of a large class of
			complex sample covariance matrices}.
		Ann. Probab., 35(2):663--714.
		
		\bibitem[\protect\astroncite{El~Karoui}{2009}]{elkaroui2009}
		{\sc El~Karoui, N.} (2009).
		\textit {Concentration of measure and spectra of random matrices: Applications
			to correlation matrices, elliptical distributions and beyond}.
		Ann. Appl. Probab., 19(6):2362--2405.
		
		\bibitem[\protect\astroncite{Erd{\"o}s}{2011}]{ErdosSurvey}
		{\sc Erd{\"o}s, L.} (2011).
		\textit{Universality of wigner random matrices: a survey of recent results}.
		Russian Math. Surveys, 66(3):507.
		
		\bibitem[\protect\astroncite{Erd{\H{o}}s et~al.}{2013}]{Erdos2013}
		{\sc Erd{\H{o}}s, L.}, {\sc Knowles, A.}, and {\sc Yau, H.-T.} (2013).
		\textit {Averaging fluctuations in resolvents of random band matrices}.
		Ann, H. Poincar{\'e}, 14(8):1837--1926.
		
		\bibitem[\protect\astroncite{Erd{\H{o}}s et~al.}{2012}]{Erdos2012_bulk}
		{\sc Erd{\H{o}}s, L.}, {\sc Yau, H.-T.}, and {\sc Yin, J.} (2012).
		\textit {Bulk universality for generalized wigner matrices}.
		Probab. Theory Related Fields, 154(1):341--407.
		
		\bibitem[\protect\astroncite{Erd\"{o}s et~al.}{2012}]{ERDOS2012}
		{\sc Erd\"{o}s, L.}, {\sc Yau, H.-T.}, and {\sc Yin, J.} (2012).
		\textit{Rigidity of eigenvalues of generalized wigner matrices}.
		Advances in Mathematics, 229(3):1435 -- 1515.
		
		\bibitem[\protect\astroncite{F\'eral and
			P\'ech\'e}{2009}]{Peche_largesteig2009}
		{\sc F\'eral, D.} and {\sc P\'ech\'e, S.} (2009).
		\textit {The largest eigenvalues of sample covariance matrices for a spiked
			population: Diagonal case}.
		J. Math. Phys., 50(7):073302.
		
		\bibitem[\protect\astroncite{Hu et~al.}{2019}]{ZhouWang2018a}
		{\sc Hu, J.}, {\sc Li, W.}, {\sc Liu, Z.}, and {\sc Zhou, W.} (2019).
		\textit{High-dimensional covariance matrices in elliptical distributions with
			application to spherical test}.
		Ann. Statist., 47(1):527--555.
		
		\bibitem[\protect\astroncite{Hu et~al.}{2019}]{ZhouWang2018b}
		{\sc Hu, J.}, {\sc Li, W.}, and {\sc Zhou, W.} (2019).
		\textit{Central limit theorem for mutual information of large mimo systems with elliptically correlated channels}.
		IEEE Transactions on Information Theory., 65(11):7168-7180.
		
		\bibitem[\protect\astroncite{Jing et~al.}{2010}]{Jing2010}
		{\sc Jing, B.}, {\sc Pan, G.}, {\sc Shao, Q.}, and {\sc Zhou, W.} (2010).
		\textit{Nonparametric estimation of spectral density functions of sample
			covariance matrices: a first step}.
		Ann. Statist., 38:3724--3750.
		
		\bibitem[\protect\astroncite{Johnstone}{2001}]{Johnstone2001}
		{\sc Johnstone, I.~M.} (2001).
		\textit{On the distribution of the largest eigenvalue in principal components
			analysis}.
		Ann. Statist., 29:295--327.
		
		\bibitem[\protect\astroncite{Johnstone}{2008}]{johnstone2008}
		{\sc Johnstone, I.~M.} (2008).
		\textit {Multivariate analysis and jacobi ensembles: Largest eigenvalue,
			tracy-widom limits and rates of convergence}.
		Ann. Statist., 36(6):2638--2716.
		
		\bibitem[\protect\astroncite{Johnstone and Nadler}{2017}]{Johnstone&Nadler2017}
		{\sc Johnstone, I.~M.} and {\sc Nadler, B.} (2017).
		\textit {Roy's largest root test under rank-one alternatives}.
		Biometrika, 104(1):181--193.
		
		\bibitem[\protect\astroncite{Knowles and Yin}{2017}]{Knowles2017}
		{\sc Knowles, A.} and {\sc Yin, J.} (2017).
		\textit{Anisotropic local laws for random matrices}.
		Probab. Theory Related Fields, 169(1):257--352.
		
		\bibitem[\protect\astroncite{Lee and Schnelli}{2016}]{lee2016}
		{\sc Lee, J.~O.} and {\sc Schnelli, K.} (2016).
		\textit{Tracy-widom distribution for the largest eigenvalue of real sample
			covariance matrices with general population}.
		Ann. Appl. Probab., 26(6):3786--3839.
		
		
		\bibitem[\protect\astroncite{Lee and Yin}{2014}]{Yin2014}
		{\sc Lee, J.~O.} and {\sc Yin, J.} (2014).
		\textit{A necessary and sufficient condition for edge universality of Wigner matrices}.
		Duke Math. J., 163 117-173. MR3161313.
		
		
		\bibitem[\protect\astroncite{Li and Yao}{2018}]{LiandYao2018}
		{\sc Li, W.} and {\sc Yao, J.} (2018).
		\textit{On structure testing for component covariance matrices of a high
			dimensional mixture}.
		J. R. Statist. Soc. B, 80(2):293--318.
		
		\bibitem[\protect\astroncite{Muirhead}{2005}]{Muirhead2005}
		{\sc Muirhead, R.~J.} (2005).
		\textit {Aspects of Multivariate Statistical Theory}.
		Wiley.
		
		
		\bibitem[\protect\astroncite{Onatski}{2008}]{onatski2008}
		{\sc Onatski, A.} (2008).
		\textit{The tracy-widom limit for the largest eigenvalues of singular complex
			wishart matrices}.
		Ann. Appl. Probab., 18(2):470--490.
		
		\bibitem[\protect\astroncite{Onatski}{2009}]{Onatski2009}
		{\sc Onatski, A.} (2009).
		\textit{Testing hypotheses about the number of factors in large factor
			models}.
		Econometrica, 77(5):1447--1479.
		
		\bibitem[\protect\astroncite{Pillai and Yin}{2014}]{pillai2014}
		{\sc Pillai, N.~S.} and {\sc Yin, J.} (2014).
		\textit{Universality of covariance matrices}.
		Ann. Appl. Probab., 24(3):935--1001.
		
		\bibitem[\protect\astroncite{Qi and Luo}{2013}]{Qi2013}
		{\sc Qi, F.} and {\sc Luo, Q.-M.} (2013).
		\textit{Bounds for the ratio of two gamma functions: from wendel's asymptotic
			relation to elezovi{\'{c}}-giordano-pe{\v{c}}ari{\'{c}}'s theorem}.
		J. Inequal. Appl., 2013(1):542.
		
		
		
		\bibitem[\protect\astroncite{Silverstein}{2009}]{Silverstein_notes}
		{\sc Silverstein, J.} (2009).
		\textit { The {Stieltjes} transform and its role in eigenvalue behavior of large dimensional random matrices. In},
		Random Matrix Theory and Its Applications
		(Z.D. Bai, Y. Chen and Y.-C. Yang, ed.), 1--25.
		Lecture Notes Series, Institute for Mathematical Sciences, National University of Singapore.
		
		
		\bibitem[\protect\astroncite{Silverstein and Bai}{1995}]{BaiandSilverstein1995}
		{\sc Silverstein, J.~W.} and {\sc Bai, Z.~D.} (1995).
		\textit{On the empirical distribution of eigenvalues of a class of large
			dimensional random matrices}.
		J. Multivariate Anal., 54:175--192.
		
		\bibitem[\protect\astroncite{Soshnikov}{2002}]{Soshnikov2002}
		{\sc Soshnikov, A.} (2002).
		\textit {A note on universality of the distribution of the largest eigenvalues in certain sample covariance matrices.}
		J. Statist. Phys., 108(5):1033--1056.
		
		\bibitem[\protect\astroncite{Tracy and Widom}{1994}]{Tracy1994}
		{\sc Tracy, C.~A.} and {\sc Widom, H.} (1994).
		\textit {Level-spacing distributions and the airy kernel.}
		Comm. Math. Phys., 159(1):151--174.
		
		\bibitem[\protect\astroncite{{Yang} et~al.}{2017}]{Yang2017}
		{\sc {Yang}, X.}, {\sc {Zheng}, X.}, {\sc {Chen}, J.}, and {\sc {Li}, H.} (2020).
		\textit {Testing high-dimensional covariance matrices under the elliptical distribution and beyond}.
		Journal of Econometrics.
		
		
		
		\bibitem[\protect\astroncite{Yin et~al.}{1988}]{Yin1988}
		{\sc Yin, Y.~Q.}, {\sc Bai, Z.~D.}, and {\sc Krishnaiah, P.~R.} (1988).
		\textit{On the limit of the largest eigenvalue of the large dimensional
			sample covariance matrix}.
		Probab. Theory Related Fields, 78(4):509--521.
		
		
		
	\end{thebibliography}
\end{document}